\documentclass[12pt, leqno]{amsart}

\usepackage{graphicx}
\usepackage{amssymb,amsmath,amsthm,amscd}
\usepackage{mathrsfs}
\usepackage{enumerate}
\usepackage{enumitem}
\usepackage{verbatim}
\usepackage[usenames,dvipsnames]{color}
\usepackage[colorlinks=true, pdfstartview=FitV,
linkcolor=blue,citecolor=blue,urlcolor=blue]{hyperref}
\usepackage[all]{xy}
\usepackage{tikz}
\usetikzlibrary{shapes, arrows}
\usepackage{tikz-cd}
\usetikzlibrary{matrix, arrows}
\usepackage{stmaryrd} 
\usepackage{dsfont}
\usepackage{bm} 
\usepackage{scalerel}[2016/12/29] 
\usepackage{amsthm}
\usepackage{adjustbox}
\usepackage{float}
\usepackage{array}
\usepackage{multirow} 
\usepackage{caption} 
\usepackage{extarrows}
\usetikzlibrary{decorations.pathmorphing}
\usetikzlibrary{fit}

\makeatletter
\newcommand*{\Relbarfill@}{\arrowfill@\Relbar\Relbar\Relbar}
\newcommand*{\xeq}[2][]{\ext@arrow 0055\Relbarfill@{#1}{#2}}
\makeatother

\makeatletter
\newcommand{\notocsection}[1]{%
  \begingroup
  \let\@tocwrite\@gobbletwo
  \section*{#1}%
  \endgroup
}
\makeatother

\newcommand{\yellowcell}[1]{%
 \tikz[baseline=(X.base)]{ \node[fill=yellow,inner sep=2pt] (X) {$#1$}; }%
 }

\allowdisplaybreaks[4]

\DeclareMathSymbol{\shortminus}{\mathbin}{AMSa}{"39}

\newcommand{\nc}{\newcommand}

\numberwithin{equation}{section}

\usepackage[colorinlistoftodos]{todonotes}
\usepackage{mathtools}
\mathtoolsset{showonlyrefs,showmanualtags}
\usepackage{empheq}

\theoremstyle{plain}
\newtheorem{lem}{Lemma}[section]
\newtheorem{pro}[lem]{Proposition}
\newtheorem{thm}[lem]{Theorem}
\newtheorem{cor}[lem]{Corollary}
\newtheorem{defi}[lem]{Definition}

\newcommand{\Pro}{\begin{pro}}
	\newcommand{\enpro}{\end{pro}}
\newcommand{\Lem}{\begin{lem}}
	\newcommand{\enlem}{\end{lem}}
\newcommand{\Thm}{\begin{thm}}
	\newcommand{\enthm}{\end{thm}}
\newcommand{\Cor}{\begin{cor}}
	\newcommand{\encor}{\end{cor}}
\newcommand{\Defi}{\begin{defi}}
	\newcommand{\enDefi}{\end{defi}}
\newcommand{\Proof}{\begin{proof}}
	\newcommand{\enproof}{\end{proof}}

\newtheorem*{assum}{Assumptions}

\theoremstyle{definition} 
\newtheorem{rem}[lem]{Remark}

\newtheorem{Convention}[lem]{Convention}

\newcommand{\Conv}{\begin{Convention}}
	\newcommand{\enconv}{\end{Convention}}
\nc{\Rem}{\begin{rem}}
	\nc{\enrem}{\end{rem}}

\newcommand{\arxiv}[1]{\href{http://arxiv.org/abs/#1}{\tt arXiv:\nolinkurl{#1}}}

\newtheorem{maintheorem}{Theorem}

\nc{\epito}{\twoheadrightarrow}

\nc{\rmkend}{\hfill$\triangledown$}
\nc{\defend}{\hfill$\triangle$}

\nc{\ccc}{\mathfrak{c}}
\nc{\CCC}{\mathfrak{C}}
\nc{\Ck}{\mathfrak{C}}

\nc{\kor}{\mathbb{C}}
\nc{\indx}{\mathbb{I}}

\nc{\CC}{C}
\nc{\cc}{c}
\nc{\sss}{s}
\nc{\ck}{\mathfrak{c}}

\nc{\Bg}{B}
\nc{\Ag}{A}
\nc{\Asmone}{\pmb{\Ag}^{(-1)}(z)}
\nc{\Aps}{\pmb{\Ag}_+(z)}
\nc{\Apsone}{\pmb{\Ag}^{(1)}_+(z)}
\nc{\Apsmone}{\pmb{\Ag}^{(-1)}_+(z)}
\nc{\Ams}{\pmb{\Ag}_-(z)}
\nc{\Hg}{H}
\nc{\Thg}{\Theta}
\nc{\Thgs}{\pmb{\Theta}}
\nc{\Thgsr}{\pmb{\grave{\Theta}}}
\nc{\Apsvar}[1]{\pmb{\Ag}^{(#1)}_+(z)} 

\nc{\smin}{{\shortminus}}

\nc{\Ui}{\widetilde{\mathbf{U}}^\imath}
\nc{\ui}{\mathbf{U}^\imath} 
\nc{\uig}{\mathbf{U}^\imath_{\mu, \boldsymbol\gamma}} 
\nc{\DrUi}{{}^{\mathrm{Dr}}\Ui}
\nc{\fext}[2]{{#1}[\negthinspace[#2]\negthinspace]}

\nc{\gTh}{\grave{\Theta}}

\nc{\DD}{\pmb{D}}
\nc{\KK}{\mathbb{K}}

\nc{\xgp}{x^+}
\nc{\Xgp}{\pmb{x}^+}
\nc{\xgm}{x^-}
\nc{\Xgm}{\pmb{x}^-}
\nc{\xgpm}{x^{\pm}}
\nc{\psig}{\phi^+}
\nc{\phig}{\phi^-}
\nc{\phipm}{\phi^\pm}
\nc{\Psig}{\pmb{\phi^+}}
\nc{\Phig}{\pmb{\phi^-}}
\nc{\Phipm}{\pmb{\phi^\pm}}
\nc{\hg}{h}

\nc{\bY}{\mathbf{Y}}
\nc{\bT}{\mathbf{T}}

\nc{\Eg}{e}
\nc{\Kg}{K}

\nc{\Rep}{\on{Rep}} 

\DeclareRobustCommand{\sqbin}{\genfrac{[}{]}{0pt}{}}

\nc{\qq}{(q-q^{-1})^{-1}}
\nc{\factor}{\Omega}

\nc{\chring}{\Z[Y_a^{\pm 1}]_{a \in \C^\times}}
\nc{\chmap}{\chi_q}
\nc{\chmod}{\Z[\mathbf{Y}_a^{\pm 1}]_{a \in \C^\times}}
\nc{\ourchmap}{\boldsymbol{\chi}_q}
\nc{\Yring}{\mathcal{Y}} 
\nc{\Ymod}{\pmb{\mathcal{Y}}}
\nc{\Yg}{\mathbf{Y}_{i,a}}
\nc{\Ys}{\mathbf{Y}_{i,\mathbf{s}}}
\nc{\Kmat}{\mathcal{K}^0} 

\nc{\Tr}{\on{Tr}}
\nc{\id}{\on{id}}

\nc{\ourR}{\mathbf{R}}
\nc{\ourQ}{\mathbf{Q}}
\nc{\ourY}{\mathbf{Y}}

\nc{\Oq}{\widetilde{\mathcal{O}}_q}
\nc{\Oqi}{\mathcal{O}_q^{[i]}}
\nc{\Oqc}{\mathcal{O}_q^{\mathbf{c}}(\widehat{\g})}
\nc{\car}{\mathcal{H}}
\nc{\qaa}{U_q(\widehat{\mathfrak{g}})}
\nc{\qla}{U_q(L\mathfrak{g})}
\nc{\uqla}{\widetilde{U}_q(L\mathfrak{g})}
\nc{\drqaa}{{}^{\mathrm{Dr}}\qaa}
\nc{\uqsl}{U_q(\widehat{\mathfrak{sl}}_2)}
\nc{\uLsl}{U_qL\mathfrak{sl}_2}
\nc{\Serre}{\mathsf{Serre}}
\nc{\Sym}{\on{Sym}}
\nc{\UXp}{UX_+} 

\nc{\Tbr}{\mathbf{T}}
\nc{\degdr}{\on{deg}^{\mathrm{Dr}}}
\nc{\Uq}{\mathbf{U}}
\nc{\Uu}{\widetilde{\mathbf{U}}}
\nc{\gaf}{\widehat{\mathfrak{g}}}

\nc{\ada}{\ad_{\bhg_{-1}}}
\nc{\adb}{\ad_{\bhg_{1}}}
\nc{\adc}{\ad_{\bHg_{1}^{(2)}}}

\nc{\tAg}{\widetilde{A}}
\nc{\hAg}{\widehat{A}}

\nc{\bH}{\overline{\Hg}}
\nc{\bhg}{\overline{\hg}}
\nc{\tHg}{\widetilde{H}}

\nc {\DrOq}{{}^{\mathrm{Dr}}\Oq}
\nc{\gup}[1]{^{(#1)}}
\nc{\gupp}{^{(1),+}}
\nc{\mi}{^{-1}}
\nc{\adh}{\operatorname{ad}_{\bar{\hg}_{-1}}}
\nc{\adhp}{\operatorname{ad}_{\bar{\hg}_{1}}}
\nc{\Omg}{\Omega^{-1}}
\nc{\Htwo}{\overline{H}_1^{(2)}}
\nc{\ad}{\operatorname{ad}}

\nc{\sspan}{\on{span}}

\newcommand{\commentout}[1]{}

\newcommand{\on}{\operatorname}
\nc{\be}{\begin{enumerate}}
	\nc{\ee}{\end{enumerate}}
\newcommand{\eq}{\begin{equation}}
	\newcommand{\eneq}{\end{equation}}
\nc{\bc}{\begin{cases}}
	\nc{\ec}{\end{cases}}
\newcommand{\eqn}{\begin{eqnarray*}}
	\newcommand{\eneqn}{\end{eqnarray*}}
\newcommand{\ba}{\begin{array}}
	\newcommand{\ea}{\end{array}}

\newcommand{\C}{{\mathbb C}}

\newcommand{\Z}{{\mathbb Z}}
\newcommand{\N}{{\mathbb N}}

\newcommand{\g}{{\mathfrak{g}}}
\nc{\Ad}{\operatorname{Ad}}

\nc{\gr}{\on{gr}}

\nc{\Aut}{\operatorname{Aut}}

\nc{\coker}{\operatorname{coker}}

\nc{\Img}{\on{Im}}
\nc{\res}{\on{res}}

\nc{\bA}{\mathbf{A}}
\nc{\bTh}{\boldsymbol\Theta}
\nc{\aTh}{\boldsymbol{\acute{\Theta}}}
\nc{\As}[2]{\mathbf{A}_{#1}(#2)} 
\nc{\Ths}[2]{\boldsymbol{\Theta}_{#1}(#2)} 
\nc{\aThs}[2]{\boldsymbol{\acute{\Theta}}_{#1}(#2)}
\nc{\bDel}{\boldsymbol\Delta}
\nc{\Dels}{\boldsymbol\Delta(z)} 
\nc{\ai}{\alpha^\vee_i}
\nc{\half}{\textstyle \frac12}
\nc{\bv}{\mathbf{v}}
\nc{\bu}{\mathbf{u}}
\nc{\bw}{\mathbf{w}} 
\nc{\bz}{\mathbf{z}} 
\nc{\hdiff}{D_\mu^\lambda} 
\nc{\hcor}{\C} 
\nc{\Pol}{\on{Pol}_\mu^\lambda}
\nc{\Vi}{\widetilde{\mathbb{U}}^{\imath}} 
\nc{\Ckk}{C} 
\nc{\RHS}{\on{RHS}} 
\nc{\LHS}{\on{LHS}}
\nc{\hlf}{^{\frac12}}
\nc{\mhlf}{^{-\frac12}}

\nc{\modv}[1]{{#1}\operatorname{-mod}}

\nc{\bl}{\bigl(}
\nc{\br}{\bigr)}

\newlength{\mylength}
\DeclareRobustCommand{\SkipTocEntry}[5]{}

\AtBeginDocument{%
   \def\MR#1{}
}

\title[Shifted quantum affine symmetric pairs]
{Shifted quantum affine symmetric pairs: finite presentations and graded coideal structures}

\author[J.-R. Li]{Jian-Rong Li}
\address{Faculty of Mathematics, University of Seville, Calle Tarfia s/n, 41012 Seville, Spain, OrciD: 0000-0001-7896-7391} 
\email{\href{mailto:lijr07@gmail.com}{lijr07@gmail.com}}

\author[T. Prze\'{z}dziecki]{Tomasz Prze\'{z}dziecki}
\address{Faculty of Mathematics, University of Vienna, Oskar Morgenstern Platz 1, 1090 Vienna, Austria, OrciD: 0000-0001-9700-1007}
\email{\href{mailto:tomasz.przezdziecki@univie.ac.at}{tomasz.przezdziecki@univie.ac.at}}



\thanks{JL was supported by the Beatriz Galindo Senior grant BG24/00114 from the Spanish Ministry of Science, Innovation and Universities. TP was supported by the EPSRC grant No.\ EP/W022834/1 \emph{Kac--Moody quantum symmetric pairs, KLR algebras and generalized Schur--Weyl duality}, as well as by the Austrian Science Fund (FWF): PAT 9039323, Grant-DOI 10.55776/PAT9039323.}

\begin{document}

\begin{abstract}
Shifted quantum affine symmetric pair algebras were recently introduced as a trigonometric analogue of shifted twisted Yangians, whose truncations are expected to quantize fixed point loci in generalized affine Grassmannian slices. They are defined via a loop presentation with infinitely many generators and relations. In this paper, we consider anti-dominantly shifted quantum affine symmetric pair algebras of simply-laced split type. Our main result is a construction of a finite presentation \`{a} la Levendorskii. We apply it in type $\mathsf{A}$ to construct an embedding of a shifted quantum symmetric pair algebra into the corresponding shifted quantum affine algebra, yielding a graded coideal structure. 
This provides evidence that the Kolb--Letzter theory of quantum symmetric pairs admits a broader generalization to the shifted setting. 
\end{abstract}

\maketitle

\setcounter{tocdepth}{1}
\tableofcontents


\section{Introduction}
\nc{\Uh}{U_q(\widetilde{\mathfrak{h}})}
\nc{\Uhz}{\fext{U_q(\widetilde{\mathfrak{h}})}{z}}
\nc{\ichmap}{\chi_q^\imath}

\subsection{Background}

Yangians were introduced by Drinfeld \cite{drinfeld-hopf}, in relation to solving the quantum Yang--Baxter equation, as a conceptual framework for the quantum inverse scattering method. While the original definition involved the so-called $J$-presentation, Drinfeld soon found a different presentation in terms of currents, which became known as ``Drinfeld's second realization'' or ``Drinfeld's new presentation'' \cite{drinfeld-dp}. The current presentation proved especially useful for representation theory, yielding, for example, a classification of finite-dimensional simple modules \cite{drinfeld-dp, CP-ens, chari-pressley}. 
One of the disadvantages of the current presentation is that it interacts badly with the coproduct. Indeed, there is no closed formula for the coproduct on the infinitely many generators in Drinfeld's second realization. Approximate formulae, which are sufficient in representation-theoretic applications, such as $q$-character theory \cite{knight}, were obtained in \cite{CP-fr}. 
This problem motivated the search for a finite, and hence more manageable, version of Drinfeld's new presentation. 
Levendorskii \cite{Lev-pr} showed that a small finite subset of generators and relations suffices to derive all the relations in Drinfeld's new presentation. In fact, it suffices to only take the first two modes in each of the generating currents with their corresponding relations. This construction became known as the ``Levendorskii presentation'' of the Yangian.    

The problem resurfaced, about two decades later, in the theory of \emph{shifted} Yangians. The latter were introduced by Brundan and Kleshchev \cite{BrKl-W}, whose seminal result identified truncated shifted Yangians with finite $W$-algebras of type $\mathsf{A}$. Kamnitzer, Webster, Weekes and Yacobi \cite{KWWY} later showed that, for dominant shifts, these truncations emerge naturally as quantizations of slices in the affine Grassmannian, with the shift parameter corresponding to weights indexing opposite Schubert cells. This result was generalized to arbitrary shifts in \cite{BFN}, with affine Grassmannian slices replaced by certain Coulomb branches associated to quiver gauge theories. 
In light of these geometric realizations, the connection between $W$-algebras and truncated shifted Yangians can be seen as a quantum analogue of the celebrated Mirkovi\'{c}--Vybornov isomorphism \cite{MVold, MVnew, WWY}. 

Since shifted Yangians are defined by shifting the range of Cartan generators in Drinfeld's second realization, it is not clear that they should admit any coproduct generalizing the natural Hopf algebra structure on ordinary Yangians. 
First indications that such a shifted coproduct structure should exist already appeared in Brundan and Kleshchev's original paper \cite{BrKl-W} under the guise of baby comultiplication. 
Finkelberg, Kamnitzer, Pham, Rybnikov and Weekes \cite{FKPRW} later solved this problem systematically using the following three-step approach.   
They first constructed a minimalistic presentation \`{a} la Levendorskii for anti-dominantly shifted Yangians, used it to define coproduct maps $$\Delta_{\mu_1, \mu_2} \colon Y_{\mu_1+\mu_2} \to Y_{\mu_1} \otimes Y_{\mu_2},$$ and then extended this construction to all shifted Yangians via shift homomorphisms. 
A geometric interpretation in terms of multiplication of affine Grassmannian slices later emerged in \cite{KPWham}. 
A counterpart of the results from \cite{FKPRW} for shifted quantum affine algebras was given by Finkelberg and Tsymbaliuk \cite{FT19}. Further constructions of Levendorskii-style presentations and associated coproducts have also appeared in the toroidal \cite{GNWco, Laurie2024} and super \cite{Li-super} settings. 

In recent years, partial generalizations of some of the aforementioned results emerged in the context of  shifted \emph{twisted} Yangians (now also often referred to as $\imath$Yangians).  
Twisted Yangians, introduced by Olshanski \cite{Olshan}, are coideal subalgebras in ordinary Yangians. They arise in mathematical physics, in connection to the Yang--Baxter and reflection equations, as algebras controlling the symmetries of integrable systems with boundaries. From an algebraic point of view, they appear naturally in the context of Gelfand--Tsetlin theory for simple Lie algebras of classical types \cite{Molev07}. Twisted Yangians were originally defined via an RTT presentation. Later, realizations in the spirit of Drinfeld's $J$-presentation were given in \cite{Jpr1, Jpr2}.  
A current presentation was constructed only quite recently by Lu, Wang and Zhang \cite{ty-2, ty-1, ty-3} using a Gauss decomposition approach and degeneration from similar results for quantum symmetric pair algebras (also known as $\imath$quantum groups). 

The current presentation was a key development which allowed much of the theory of shifted Yangians to be generalized to the twisted setting. Analogues of the Brundan--Kleshchev isomorphism for $W$-algebras of classical types were obtained in \cite{LPTTW, TT-Yan}. Truncations via GKLO representations and partial quantization results, in the setting of fixed point loci or symmetric quotients of the affine Grassmannian, were established in \cite{BPT, LWW1, LWW2, SSX, Zichang}, while interpretations in terms of Coulomb branches were proposed in \cite{LWW1, NakajimaQSP, SSX, SSX2, Zichang}. Moreover, trigonometric versions of shifted twisted Yangians - called shifted quantum affine symmetric pair algebras or shifted $\imath$quantum groups - and their GKLO representations were introduced in \cite{LP-gklo, LWW-i}. 

Nevertheless, fundamental algebraic questions regarding minimalistic presentations and possible comodule structures on shifted twisted Yangians and their trigonometric analogues have remained largely unexplored. The recent results of Huang and Lu \cite{min-1, min-2}, who established a Levendorskii-style presentation for \emph{unshifted} twisted Yangians of quasi-split type, are a notable exception. This presentation allowed them to obtain a new explicit description of the embeddings of twisted Yangians into ordinary Yangians and recover the coideal structures previously known to arise from the RTT and J-presentations \cite{Olshan, Jpr1, Jpr2}. This provides further motivation for investigating minimalistic presentations in the shifted twisted setting, where latter approaches are not (yet) available.

\subsection{Main results}

In this paper, we will work in the trigonometric setting. Quantum symmetric pair coideal subalgebras of Kac--Moody type were first systematically studied by Kolb \cite{kolb-14}. They are known to admit a Serre-type presentation, with  generators corresponding to nodes of a decorated Dynkin (i.e., Satake) diagram \cite{DeClercq21, KolbYakimov21}. 
In the affine case, the first results suggesting the existence of a Drinfeld-type presentation were obtained, for the $q$-Onsager algebra, by Baseilhac and Kolb \cite{bas-kol-20}, who constructed real and imaginary root vectors, derived commutation relations among these root vectors, and established a PBW basis. Motivated by the theory of $\imath$Hall algebras and intrinsic braid group symmetries of $\imath$quantum groups, Lu and Wang \cite{lu-wang-21} later gave an analogue of Drinfeld's presentation in the split $\mathsf{ADE}$ case, in a form more closely resembling Drinfeld's original second realization. These results were then generalized in a series of works \cite{WangZhang2, LWZ-quasi, ZhangDr, LPWZ} to an arbitrary quasi-split type.
Here we will restrict ourselves to quantum affine symmetric pair algebras of \emph{simply-laced split type}\footnote{In fact, we will work in a slightly more general framework which also encompasses, for example, toroidal versions of quantum symmetric pair algebras. We set this subtlety aside in the introduction.}, i.e., those corresponding to undecorated Satake diagrams with no multiple edges. In their loop realization, these algebras can be presented using  ``real'' $\bA_i(z)$ and ``imaginary'' $\bTh_i(z)$ generating series, indexed by the vertices of the underlying finite Satake diagram. The imaginary currents are required to commute. Moreover, there are real--real and real--imaginary  quadratic relations, together with an analogue of Drinfeld--Serre relations. 

Now let $\mu$ be an arbitrary coweight. The shifted quantum affine symmetric pair algebras $\Ui_\mu$ introduced in \cite{LP-gklo, LWW-i} involve exactly the same relations as for $\mu=0$, but the range of the imaginary series $\bTh_i(z)$ is shifted by $\mu_i = \ai(\mu)$. The shift has a profound effect on the structure of the algebra, with anti-dominant shifts being of particular interest. It is worth emphasizing that this loop presentation gives the \emph{definition} of $\Ui_\mu$. In particular, $\Ui_\mu$ does not \emph{by default} admit a Serre-type presentation as in \cite{kolb-14}, or a description as a coideal subalgebra, unless $\mu = 0$. 
The following Theorem is the main result of the paper. 

\begin{maintheorem} \label{main A}
If $\mu$ is anti-dominant then $\Ui_\mu$ admits a presentation with generators $A_{i,0}, \cdots, A_{i, \mu_i -1}$ and a finite subset of the relations from the loop presentation. 
\end{maintheorem} 

Theorem \ref{main A} appears in the text as Theorem \ref{thm: 1Lev} (rank one case) and Theorem \ref{thm: Lev presentation gen} (general case). The precise list of relations can be found in \S \ref{ss: Lev relations 1 rank} (rank one case) and \S \ref{subsec: gen Lev presen}. 
Some higher rank redundancies among the relations are established in Theorem \ref{thm: redundancies in high rank}. 
The following Theorem is a key application of Theorem \ref{main A}. 

\begin{maintheorem} \label{main B} 
In type $\mathsf{A}$, if $\mu$ is anti-dominant, then there exists an explicit homomorphism 
\[
\iota_\mu \colon \ \uig \to \mathbf{U}_\mu
\]
from a central reduction $\uig$ of $\Ui_\mu$ to the corresponding shifted quantum affine algebra $\mathbf{U}_\mu$  with the same shift. 
\end{maintheorem} 

Theorem \ref{main B} appears in the text as Theorem \ref{thm: homo rank 1} (rank one case) and Theorem \ref{thm: homo rank 2} (general case). In rank one, the formulae describing the homomorphism have an especially elegant form closely resembling the usual embedding of the $q$-Onsager algebra into quantum affine $\mathfrak{sl}_2$, with the boundary Levendorskii generators $A_0, A_{\mu-1}$ analogous to the standard generators $B_0, B_1$ of the $q$-Onsager algebra. In higher rank, the images of the lower boundary generators $A_{i, \mu_i-1}$ need to be modified by rather cumbersome correction terms, somewhat resembling the correction terms appearing in the higher rank comultiplication formulae in \cite{FT19}. 

\begin{maintheorem} \label{main C} 
The homomorphisms $\iota_\mu$ are injective and satisfy the following ``graded coideal'' property: 
\[
\Delta_{\mu_1, \mu_2} \circ \iota_\mu ( \uig ) \subseteq \Uq_{\mu_1} \otimes \mathbf{U}^\imath_{\mu_2,\boldsymbol\gamma}, 
\]
for any $\mu_1, \mu_2$ anti-dominant with $\mu_1 + \mu_2 = \mu$. 
\end{maintheorem} 

Theorem \ref{main C} appears in the text as Theorem \ref{thm: coideal property 1} (rank one case) and Theorems \ref{thm: injectivity} and \ref{thm: coideal property 2} (general case). 

\subsection{Techniques} 

We split the proof of Theorem \ref{main A} into two steps. We first tackle the rank one case, i.e., the shifted $q$-Onsager algebra, in \S \ref{sec: Levendorskii proof}. Subsequently, we derive all the mixed vertex relations in \S \ref{sec: gen sl split c}.   
The general architecture of our proof resembles, on the surface, the corresponding ``simultaneous induction'' argument for shifted quantum affine algebras due to Finkelberg and Tsymbaliuk \cite{FT19}. 
Nevertheless, our proof differs from that in \emph{op.\ cit.} in several key regards and relies on genuinely new techniques.  

Firstly, the derivation of the quadratic  relations between real and imaginary series is much more challenging in our setting. We follow a completely different approach than in \emph{op.\ cit.} by first establishing a certain identity in a free non-commutative algebra, which reduces the problem to checking that a small number of relations between the real generators holds. 

The second complication concerns the base step of the induction aiming to establish the quadratic real--real same vertex relations $R_{i,i}(r,s)$ (see \eqref{eq: 1grel3}). 
This problem does not occur at all in the setting of shifted quantum affine algebras, but requires a very careful treatment in our situation. The issue is caused, roughly speaking, by the behaviour of the first non-trivial imaginary mode $H_{i,1}$ (in particular, the relation $[H_{i,1}, A_{i,r}] = [2](A_{i,r+1} - A_{i, r-1})$), and the fact that the RHS of $R_{i,i}(r,s)$ depends only on $|r-s|$. This forces us to provide an independent argument establishing $R_{i,i}(r,r)$ for all $r \in \Z$ to initiate the induction.

The third major difficulty concerns the derivation of the $\imath$Drinfeld--Serre relations \eqref{eq: drel6}.  
The source of the problem is that, unlike for quantum affine algebras, the RHS of these relations does not vanish.  
We first establish some general identities without using any $\imath$Serre relations, and then use these to set up an induction, where the base case presents a similar additional challenge to the quadratic relations and again needs to be treated with special care. By comparison, the derivation of the Drinfeld--Serre relations for shifted quantum affine algebras relies on a simple non-degeneracy argument and is much more straightforward. 

The flow of the induction and the logical dependencies between the various hypotheses are overall quite complicated. We have endeavoured to clarify these dependencies with multiple figures throughout the manuscript 
and taken great care to make all assumptions used at each step as explicit as possible. A basic schematic explanation of the induction in rank one is presented in Figure \ref{fig: induction}. 

The proof of Theorem \ref{main B} in rank one is based on an explicit calculation with the Levendorskii generators. In higher rank, where the formulae describing the homomorphism are more complicated due to correction terms, we use a more abstract argument which reduces the problem to the unshifted case via shift homomorphisms and specialization maps. The correction terms were determined by making precise the approximations from \cite{Przez-23}. 
The proof of Theorem \ref{main C} relies on an associated graded argument, together with some auxiliary information provided by the shuffle realization \cite{FT19b} of the negative and positive halves of $\mathbf{U}_\mu$. As a corollary, we also obtain a PBW theorem for $\uig$ -- see Corollary \ref{cor: PBW}. 

\subsection{Perspectives} 

Our work raises a number of broader conceptual questions. 
There is a certain tension in the theory of quantum symmetric pairs or $\imath$quantum groups, reflected in the terminology itself, between viewing them as intrinsic algebraic objects and subalgebras of ordinary quantum groups.  
The discovery of a Drinfeld--style presentation for quantum affine symmetric pair algebras is connected to numerous  developments which support the ``intrinsic'' point of view. In particular, \emph{shifted} analogues of quantum affine symmetric pair algebras could be defined without any reference to a background shifted quantum group. 
However, the results of the present paper show that the full family of anti-dominantly shifted quantum affine symmetric pair algebras can indeed be seen as a graded coideal subalgebra inside the corresponding family of shifted quantum affine algebras $\Uq_\mu$, establishing the two as a genuine \emph{pair} $(\bigoplus_\mu \mathbf{U}^\imath_\mu, \ \bigoplus_\mu \Uq_\mu)$. 
  
It is rather interesting that this result has emerged in the reverse order to the unshifted case. Namely, quantum affine symmetric pair algebras were first constructed as coideal subalgebras, then described intrinsically via a Serre--style presentation, which was later shown to be equivalent to a newly constructed Drinfeld--style presentation. In contrast,  shifted analogues were first defined via a loop presentation, and the Levendorskii--style presentation established here has provided the bridge needed to construct embeddings into shifted quantum groups and the associated coideal structures. 
This sequence is perhaps quite natural given that the correction terms in our embedding formulae are rather complicated. 

Overall, we find it quite remarkable that the ``intrinsic'' and ``embedded'' viewpoints coincide in the shifted case. 
Accordingly, it would be interesting to consider whether the Kolb--Letzter theory of quantum symmetric pairs could be extended more broadly to the shifted setting, perhaps pointing towards a systematic theory of compatible quantizations of fixed-point subalgebras in shifted loop algebras.

We also expect that our results should have an application to the theory of boundary $q$-characters, developed in \cite{Przez-23, LP, LP-LMP, LP-quasi}. The construction of $q$-characters therein relies on certain properties of the coproduct, namely, the fact that the series $\bTh_i(z)$ exhibit group-like behaviour modulo $\mathbf{U}^\imath \otimes \Uq_+$. This was established using rather tedious computations comparing Lusztig's braid group action with the intrinsic braid symmetries of $\Ui$. It is likely that these can be simplified via the Levendorskii presentation, following the approach of \cite{min-1, min-2} in the case of Yangians. We also expect that the construction of boundary $q$-characters can be generalized to the shifted case.  

Furthermore, our results suggest that anti-dominantly shifted quantum affine symmetric pair algebras should admit a reflection equation presentation, in analogy to the RTT presentation of shifted quantum affine algebras obtained in \cite{FT19, FPT-lax1, FPT-lax2}. It is likely that the rather complicated formulae for the embeddings $\iota_\mu$ will emerge more naturally in this framework. 

Finally, we note that there is another approach to coproducts for shifted quantum affine algebras, generalizing Drinfeld's ``formal'' or ``topological'' coproduct \cite{hernandez-05, hernandez-07}. In this approach, the topological coproduct formulae directly generalize across all shifts in series form, and do not rely on the Levendorskii presentation. 
A topological approach to affine $\imath$quantum groups has recently been developed by Lu and Pan \cite{LuPan} based on $\imath$Hopf algebra techniques. In particular, they obtained a uniform conceptual derivation of Drinfeld presentations as well as a coideal subalgebra realization with respect to the topological coproduct. In \cite{ALPV}, the present authors, together with Appel and Vlaar, adapted the Kolb--Letzter approach to the topological setting, constructed a meromorphic $K$-matrix, and related it to the theory of boundary $q$-characters. Moreover, we sketched how the topological coideal structure can be adapted to the shifted setting and used to recover a version of $\imath$GKLO representations. 
Thus, this approach also manifests     
$(\bigoplus_\mu \mathbf{U}^\imath_\mu, \ \bigoplus_\mu \Uq_\mu)$ as a genuine pair from the topological perspective.

\addtocontents{toc}{\SkipTocEntry}
	
\section*{Acknowledgements} 

We would like to thank Andrea Appel, Pascal Baseilhac, Istvan Heckenberger, Stefan Kolb, Duncan Laurie and Bart Vlaar for helpful discussions and suggestions.  
We are also grateful to the Erwin Schr\"{o}dinger Institute (ESI) in Vienna, where parts of this project were carried out,  for its hospitality and support.


\section{Shifted $q$-Onsager algebra} 
\label{s: qOns}

The base field is $\C(q)$ throughout the paper. Alternatively, we may work over $\C$ with $q$ specialized to an element which is not a root of unity. 
We will use the following standard notation for quantum integers, factorials, binomial coefficients and brackets: 
\[
[n] = \frac{q^n - q^{-n}}{q-q^{-1}}, \qquad
[n]! = \prod_{k=1}^{n}[k], \qquad 
\begin{bmatrix}m\\k\end{bmatrix}
=\frac{[n]!}{[k]![n-k]!}, \qquad 
 [x,y]_q = xy - q yx. 
\]

\subsection{Shifted $q$-Onsager algebra} 

Let $\mu \in \Z$. 
We will use the following notation for generating series: 
\[
\As{}{z} = \sum_{r \in \Z} A_{r} z^r, \qquad \Ths{}{z} = \KK (q-q\mi)\mi z^{-\mu} +  \sum_{r > -\mu} \Theta_{r} z^r, \qquad \Dels = \sum_{r \in \Z} \Ck^r z^r. 
\]
By convention, $\Theta_{-\mu} = \KK(q-q\mi)\mi$ and $\Theta_{r} = 0$ for $r < -\mu$. We will refer to $\As{}{z}$ and $\Ths{}{z}$ as the \emph{real} and \emph{imaginary} series, respectively. We will also refer to their modes in the same fashion. 

\Defi \label{defi: shifted qOns}
Let $\Ui_\mu$ be the algebra generated by $A_{s}$ and 
$\Theta_{r}$, as well as central elements 
$\KK^{\pm1}$, $\Ck^{\pm1}$ $(r > -\mu, \ s\in\Z)$, subject to the following relations: 
\begin{align}
\bTh(z) \bTh(w) &= \bTh(w) \bTh(z), \label{eq: 1drel1} \\[2pt] 
\bTh (z)  \bA(w) &= \frac{(1 -q^{-2}zw^{-1}) (1 -q^{2} zw \CCC)}{(1 -q^{2}zw^{-1})(1 -q^{-2}zw \CCC)} \bA(w) \bTh (z), \label{eq: 1drel2} \\[2pt]
\Sym_{z,w} (q^2z - w) \bA(z) \bA(w) &= 
(w-z) \bDel(zw) \big(\bTh(z) - \bTh(w) \big).  
 \label{eq: 1drel3}
\end{align} 
\enDefi 

The algebra $\Ui_\mu$ is a special case of shifted quantum symmetric pair algebras introduced in \cite{LP-gklo}, and shifted $\imath$quantum groups considered in \cite{LWW-i}, in the context of GKLO representations. Here we will refer to $\Ui_\mu$ as the \emph{$\mu$-shifted $q$-Onsager algebra}. 

\nc{\btheta}{\overline{\Theta}} 

We will now rewrite the defining relations term-by-term in several ways. 
Let 
\[
(q-q\mi)  \Ths{}{z} z^{\mu} = \KK \exp \left( (q-q^{-1}) \sum_{r \geq 1}^\infty \Hg_{r} z^r \right) 
\] 
and 
\begin{align} 
Y(m,s) &:= 
[\Theta_{m+1}, A_{s}] + [\Theta_{m-1}, A_{s}] \Ck + q^{-2}\left( [A_{s+1}, \Theta_{m}]_{q^{4}} -  [\Theta_{m}, A_{s-1}]_{q^{4}}\Ck \right) = 0, \\
\overline{Y}(m,s) &= Y(m,s) - \delta_{m\neq2}\CCC Y(m{-}2,s). 
\end{align}

\Lem \label{lem: Y-H equiv}
The following are equivalent:
\be
\item $Y(m-\mu,s) = 0$ for all $0 \leq m \leq N$ and $s \in \Z$, 
\item $\overline{Y}(m-\mu,s) = 0$ for all $0 \leq m \leq N$ and $s \in \Z$, 
\item $[\Hg_{m+1}, \Ag_{s}] = \textstyle \frac{[2m]}{m} (\Ag_{s+m}- \Ag_{s-m}\CCC^m)$ for all $0 \leq m \leq N$ and $s \in \Z$.  
\ee
\enlem

\Proof
See, e.g., \cite[Proposition 2.8]{lu-wang-21}. 
\enproof

It follows from Lemma \ref{lem: Y-H equiv} that the relations from Definition \ref{defi: shifted qOns} may be rewritten as: 
\begin{align} 
P(m,n): &\qquad& [\Theta_{m},\Theta_{n}] &=0, \label{eq: 1grel1} \\
Q(m,r): &\qquad& [\Hg_{m}, \Ag_{r}] &= \textstyle \frac{[2m]}{m} (\Ag_{r+m}- \Ag_{r-m}\CCC^m), \label{eq: 1grel2} \\ 
\label{eq: 1grel3}
R(r,s): &\qquad& [\Ag_{s+1}, \Ag_{r}]_{q^{2}} + [\Ag_{r+1}, \Ag_{s}]_{q^{2}}
&= -\CCC^r\Thg_{s-r+1} + \CCC^{r+1}\Thg_{s-r-1}  \\ \notag
&\qquad& &\quad  - \CCC^s\Thg_{r-s+1} + \CCC^{s+1}\Thg_{r-s-1}. 
\end{align} 
Observe that relation \eqref{eq: 1grel3} is symmetric in $r$ and $s$, so we may assume $s \geq r$. Denote the LHS of \eqref{eq: 1grel3} as $R_L(r,s)$, and abbreviate 
\[ 
\btheta_{m} = \Theta_{m} - \delta_{m \neq 2} \CCC \Theta_{m-2}. 
\] 
If $\mu \leq 0$ then \eqref{eq: 1grel3} can be rephrased as
\[
R_L(r,s) = - 2^{\delta_{r,s}} \Ck^{r} \btheta_{s-r+1}. 
\] 
Moreover, 
\begin{align} 
\overline{Y}(m,s) &= 
[\btheta_{m+1}, A_{s}] + [\btheta_{m-1}, A_{s}] \Ck + q^{-2}\left( [A_{s+1}, \btheta_{m}]_{q^{4}} -  [\btheta_{m}, A_{s-1}]_{q^{4}}\Ck \right) \\
&= [\Theta_{m+1}, A_{s}] - [\Theta_{m-3}, A_{s}] \Ck^2 + q^{-2}\left( [A_{s+1}, \btheta_{m}]_{q^{4}} -  [\btheta_{m}, A_{s-1}]_{q^{4}}\Ck \right). 
\end{align} 

\subsection{Levendorskii presentation - shifted case} 
\label{ss: Lev relations 1 rank}

Given a non-zero \emph{anti-dominant} shift $\mu \leq -1$, let $\Vi_\mu$ be the algebra generated by 
\[
A_{r} \ ( \mu-1 \leq r \leq 0 ), \quad H_{1}, \quad \KK^{\pm1}, \quad \Ck^{\pm1}, 
\] 
(which we call `Levendorskii generators'), with $\KK^{\pm1}, \Ck^{\pm1}$ central, 
subject to the relations
\begin{align} 
 [\Hg_{1}, \Ag_{r}] &= \textstyle [2] (\Ag_{r+1}- \Ag_{r-1}\CCC), \label{eq: 1Urel1} \\  \label{eq:  1Urel2}
 [\Ag_{s+1}, A_{r}]_{q^{2}} + [\Ag_{r+1}, \Ag_{s}]_{q^{2}} 
&= \left\{ \begin{array}{l l}
- \KK \Ck^{\mu-1} H_{1} & \mbox{if } |r{-}s| = -\mu, \\[2pt] 
- 2^{\delta_{\mu,-1}} \rho \KK \Ck^{\min(r,s)} & \mbox{if } |r{-}s| = {-}\mu{-}1 , \\[2pt] 
0 & \mbox{otherwise}, 
\end{array} \right.
\end{align} 
for any indices such that these relations make sense, 
and
\eq \label{eq: 1HHlev} 
[H_1, \Theta_{2-\mu}] = 0, 
\eneq
with $\Theta_{2-\mu}$ defined via \eqref{eq: J(s)} below and 
\[
\rho = (q-q\mi)\mi, \qquad \bH_1 = [2]\mi H_1. 
\]
In the sequel we will call the above relations the `Levendorskii relations' or the `defining relations'. 

\Rem \label{rem: bound gen}
The equality $|r-s| = -\mu$ can only hold (up to switching $r \leftrightarrow s$) if  $r=-1$ and $s = \mu - 1$. On the other hand, $|r-s| = -\mu -1$ can only hold (up to switching $r \leftrightarrow s$) if (a) $r=-1$ and $s = \mu$ or (b) $r=-2$ and $s = \mu - 1$. Also note that $H_1$ can be defined in terms of the other generators and has only been added as a generator for convenience. In fact, going further, the entire presentation could be expressed in terms of the ``boundary'' generators $A_0, A_{-1}, A_{\mu}, A_{\mu-1}$ alone. 
\enrem


Inductively define 
\begin{align} \label{eq: I(r)}
I(r): &\qquad& 
A_{r} &= 
 \left\{ \begin{array}{l l}
	 [\bH_{1}, A_{r-1}] + \Ck A_{r-2} & \mbox{if } r > 0, \\[2pt]  
	\Ck\mi A_{r+2} - \Ck\mi [\bH_{1}, A_{r+1}]   & \mbox{if } r < \mu - 1, 
\end{array} \right. \\[2pt]  \label{eq: J(s)}
J(s): &\qquad& \Theta_{s} &= -\Ck^{\lceil \frac{s+1}{2} \rceil} R_L(-\lceil \textstyle\frac{s+1}{2} \rceil, \lfloor \textstyle\frac{s-3}{2} \rfloor) + \Ck \Theta_{s - 2},  
 \end{align}
for $s > 1 - \mu$. 
As before, abbreviate 
\[ 
\btheta_{s} = \Theta_{s} - \CCC \Theta_{s-2}.  
\]
Note that the definition of $\Theta_s$ is independent of $\mu$. Table \ref{table theta} shows the $R_L$ term occurring in the first few $\overline{\Theta}_s$ (if $s \leq 1 -\mu$, then this happens by virtue of the Levendorskii relations rather than definition $J(s)$). 
\begin{table}[h] 
\begin{tabular}{!{\vrule width 1.2pt}c|c|c|c|c|c|c!{\vrule width 1.2pt}}
\noalign{\hrule height 1.2pt}
$\overline{\Theta}_3$ & $\overline{\Theta}_4$ & $\overline{\Theta}_5$ & $\overline{\Theta}_6$ &
$\overline{\Theta}_7$ & $\overline{\Theta}_8$ & $\overline{\Theta}_9$ \\ \hline
$R_L(-2,0)$ & $R_L(-3,0)$ & $R_L(-3,1)$ & $R_L(-4,1)$ &
$R_L(-4,2)$ & $R_L(-5,2)$ & $R_L(-5,3)$ \\
\noalign{\hrule height 1.2pt}
\end{tabular}
\caption{Approximate form of $\overline{\Theta}_s$.}
\label{table theta}
\end{table}

\Rem \label{rem: theta} 
The choice of definition \eqref{eq: J(s)} is of course not unique. The formula chosen here is convenient from the point of view of the presentation of \S \ref{sec: Levendorskii proof}. In \S \ref{sec: sub real1}--\ref{sec: sub real2}, we prefer to work with (any) different convention such that $\overline{\Theta}_{2-\mu} = - \Ck^{1-\mu} R_L(\mu-1,0)$. The proof in \S \ref{sec: Levendorskii proof} still works for this convention and, essentially, involves shifting down the seven-square yellow ``honeycomb'' in Figures \ref{fig: yellow1}--\ref{fig: yellow2}. 
\enrem

\Thm \label{thm: 1Lev}
For $\mu \leq -1$, there is an algebra isomorphism
\[
\Vi_\mu \to \Ui_\mu
\]
sending the Levendorskii generators to the corresponding Drinfeld generators (denoted in the same way). 
\enthm

\Proof
This is an algebra homomorphism since the Levendorskii relations are a subset of the Drinfeld relations, also it is  surjective since all the generators of $\Ui_\mu$ are in its image. So it suffices to check that all the Drinfeld relations also hold in $\Vi_\mu$. This is done in section \S \ref{sec: Levendorskii proof} below. 
\enproof

\subsection{Levendorskii presentation - unshifted case} 

Recall from \cite{Ter01, lu-wang-21} that the (universal) $q$-Onsager algebra $\Oq$ is the algebra generated by $B_0, B_1$ and invertible central elements $\KK_0, \KK_1$ subject to the $q$--Dolan--Grady relations:
\eq
\label{eq: Ons rel} \tag*{DG$(i,j)$:}
\Serre(i,j) := \sum_{r=0}^3 (-1)^r \sqbin{3}{r} \Bg_i^{3-r}\Bg_j\Bg_i^r = - q\mi \KK_i [2]^2[\Bg_i, \Bg_j] 
\eneq 
for $0 \leq i \neq j \leq 1$. 
By \cite[Theorem 2.16]{lu-wang-21} (see also \cite[Theorem II]{bas-kol-20}), there is an algebra isomorphism 
\eq \label{eq: qOns iso}
\Oq \xrightarrow{\sim} \Ui_0, \qquad B_1 \mapsto A_0, \quad B_0 \mapsto A_{-1} \Ck \KK\mi, \quad \KK_1 \mapsto \KK, \quad \KK_0 \mapsto \Ck \KK\mi. 
\eneq

Let $\mathcal{F}$ be the free non-commutative algebra generated by $B_0, B_1, \KK_0$ and $\KK_1$. Define elements $A_0, A_{-1}, \KK, \Ck \in \mathcal{F}$ using the assignments in \eqref{eq: qOns iso}. Moreover, set $H_1 = - \KK\mi \Ck [A_0, A_{-1}]_{q^2}$, $\Theta_1 = \KK H_1$ and define $A_{1}, A_{-2}, \Theta_2$ according to \eqref{eq: I(r)}--\eqref{eq: J(s)}. 

\Lem \label{lem: equiv qOns}
The following relations are equivalent in $\mathcal{F}$: 
\be
\item DG$(1,0)$ and $R(0,0)$; 
\item DG$(0,1)$ and $R(-2,-2)$. 
\ee
\enlem

\Proof
By definition, $H_1 = \gamma [A_0, A_{-1}]_{q^2}$ with $\gamma = - \KK\mi \Ck$. Let us expand both sides of $R(0,0)$:
\begin{align}
\textstyle\frac12 R_L(0,0) = [A_1, A_0]_{q^2} &= \big[ [2]\mi [H_1, A_0] + \Ck A_{-1}, A_0]_{q^2} \\ 
&= \gamma [2]\mi \Big[\big[ [A_0, A_{-1}]_{q^2}, A_0 \big], A_0 \Big]_{q^2} + \Ck [A_{-1}, A_0]_{q^2} \\
&= -q^2[2]\mi \Serre(1,0) + \KK_1 [B_0, B_1]_{q^2}, \\
\textstyle\frac12 R_R(0,0) = - \Theta_1 &= \Ck [A_0, A_{-1}]_{q^2} = \KK_1 [B_1, B_0]_{q^2}.  
\end{align} 
Rearranging the identities above yields DG$(1,0)$, giving the first equivalence. The proof of the second equivalence is analogous. 
\enproof 

Lemma \ref{lem: equiv qOns} allows us to formulate the following Levendorskii-style presentation of the unshifted algebra $\Ui_0$. 

\Pro \label{pro: no shift Lev}
The algebra $\Ui_0$ is isomorphic to the algebra generated by
\[
\{ A_1, A_0, A_{-1}, A_{-2}, H_1, \KK^{\pm1}, \Ck^{\pm1} \},
\]
with $\KK^{\pm1}, \Ck^{\pm1}$ central, subject to the relations 
\begin{align*}
 [\Hg_{1}, \Ag_{0}] &= \textstyle [2] (\Ag_{1}- \Ag_{-1}\CCC), \\
 [\Hg_{1}, \Ag_{-1}] &= \textstyle [2] (\Ag_{0}- \Ag_{-2}\CCC), \\ 
 [A_1, A_{0}]_{q^2} &= -\KK H_1, \\ 
 [A_0, A_{-1}]_{q^2} &= -\KK \Ck^{-1} H_1, \\ 
 [A_{-1}, A_{-2}]_{q^2} &= -\KK \Ck^{-2} H_1. 
\end{align*}
\enpro 

\Proof
By Lemma \ref{lem: equiv qOns}, the algebra with the above presentation is isomorphic to $\Oq$, which, by \eqref{eq: qOns iso}, is isomorphic to $\Ui$. 
\enproof

The presentation from Proposition \ref{pro: no shift Lev} is rather different from the presentation from Theorem \ref{thm: 1Lev}. Namely, it involves additional generators $A_1, A_{-2}$ and relations $R(0,0)$, $R(-2,-2)$, while an analogue of relation \eqref{eq: 1HHlev} is not required.  

\section{Proof of the Levendorskii presentation} 
\label{sec: Levendorskii proof}

The entire section is devoted to a proof of Theorem \ref{thm: 1Lev}. The general architecture of the proof resembles the corresponding proof for shifted quantum affine algebras due to Finkelberg and Tsymbaliuk \cite{FT19}. In particular, the core of the proof is formed by a simultaneous induction, set up in \S \ref{subsec: conditions}. However, our proof differs from that in \emph{op.\ cit.} in two important respects. Firstly, the derivation of the Drinfeld relations between real and imaginary series  is much more difficult than for shifted quantum affine algebras. We use a completely different approach than in \emph{op.\ cit.} by first establishing a certain identity in a free non-commutative algebra (Lemma \ref{lem: key id}), which reduces the problem to checking that certain relations between the real generators hold (Lemmata \ref{lem: H2 comm},  \ref{lem: real-mag} and \S \ref{sss: CN}). The second complication concerns the base step (\S \ref{ss: early B}) in the induction aiming to establish that a symmetrized $q^2$-commutator of real series equals an imaginary series. This problem does not occur at all in the setting of shifted quantum affine algebras, but requires a very careful treatment in our situation. 

\subsection{Preliminaries} 

The defining Levendorskii relation $Q(1,r)$ together with the inductive definition $I(r)$ of $A_r$ implies that 
\eq \label{eq: ind_H1A}
Q(1,r): \qquad [\Hg_{1}, \Ag_{r}] = \textstyle [2] (\Ag_{r+1}- \Ag_{r-1}\CCC) 
\eneq
holds for all $r \in \Z$. 
Applying \ref{eq: ind_H1A} to $R_L(r,s)$ yields 
\eq \label{eg: rule H1} 
[\overline{H}_1, R_L(r,s)] = R_L(r+1,s) + R_L(r,s+1) - \Ck R_L(r-1,s) - \Ck R_L(r,s-1). 
\eneq 
In the sequel, we will often use the following lemma. 

\Lem
\label{lem: H1 rule lemma}
Suppose that (i) $[\overline{H}_1, R_R(r,s)] = 0$, (ii) $R(r,s)$ and (iii) three of the four relations corresponding to the terms on the RHS of \eqref{eg: rule H1} hold. Then the fourth relation holds as well. 
\enlem

\Proof
Suppose we want to prove $R(r,s+1)$. We can assume without loss of generality that $s \geq r$. Since $R_L(r+1,s) - \Ck R_L(r,s-1) = 0$ by assumption, we get $R_L(r,s+1) = \Ck R_L(r-1,s) = - \Ck^r \overline{\Theta}_{s-r+2}$, which is precisely $R(r,s+1)$. 
\enproof

Moreover, define 
\[
H_2 := \KK\mi \Theta_{2-\mu} - \textstyle\frac12 \KK\mi  (q-q\mi) H_1^2, \qquad \overline{H}_2 = \textstyle\frac{2}{[4]} H_2. 
\] 

The following lemma is equivalent to ($\mathtt{C}_1$) defined below in \S \ref{subsec: conditions}. We treat it separately because it does not fit into the inductive logic of \S \ref{subsec: conditions}. In particular, the proof of $(\mathtt{B}_{\mu-1})$--$(\mathtt{B}_{\mu+1})$ depends on it, which would create a circularity. Below we present a legitimate proof of ($\mathtt{C}_1$) which does \emph{not} assume $(\mathtt{B}_{\mu-1})$--$(\mathtt{B}_{\mu+1})$. 

\Lem \label{lem: H2 comm}
We have $Y(1-\mu, s) = 0$ for all $s \in \Z$. Hence 
\eq \label{eq: H2 Ar comm}
[\overline{H}_2, A_r] = A_{r+2} - \Ck^2 A_{r-2}. 
\eneq
\enlem

\Proof
The first statement will follow from Lemma \ref{lem: real-mag}, with $m=1-\mu$, if we check that the assumptions are satisfied. Assumption $(\mathtt{A}_{1})$ (see \S \ref{subsec: conditions}) holds thanks to \ref{eq: 1HHlev}. Let $\mu \leq -3$. Then the first five relations (those not involving $n$) as well as the relations $R(x-1, n-2)$ and $R(x-2, n-2)$ are instances of the defining relations. The remaining four relation can be deduced by applying Lemma \ref{lem: H1 rule lemma}. This process is best illustrated diagrammatically. Using the conventions of Figures \ref{fig: yellow1}--\ref{fig: yellow2}, we apply $[ \overline{H}_1, - ]$ to relations $R(-1,-1), \cdots, R(\mu, -1)$ on the outer black diagonal. This yields relations $R(-1,0), \cdots, R(\mu,0)$. We then apply $[ \overline{H}_1, - ]$ to $R(-2,0), \cdots, R(\mu-1,0)$ and repeat this process for as long as possible. Diagrammatically, each step yields a new green diagonal terminating in the $(1-\mu)$-th column, supplying $R(x-1, n-2), R(x, n-1)$ and $R(x, n-2)$. Finally, applying $[\overline{H}_1, - ]$ to $R(\mu, 0)$, we get the remaining relation $R(x-2, n-2)$. If $\mu \in \{-1, -2\}$, the process is similar, except we cannot assume $R(x-2, x-2)$. In these special cases the assumption is, however, redundant. 
The identity \eqref{eq: H2 Ar comm} follows from $Y(1-\mu, s) = 0$ by Lemma \ref{lem: Y-H equiv}. 
\enproof

\begin{figure}[h] \begin{tikzcd}[row sep=-7pt, column sep=-1.5pt] & \textcolor{rgb,255:red,92;green,214;blue,92}{R(-1,0)} && \yellowcell{\textcolor{rgb,255:red,92;green,214;blue,92}{R(-2,1)}} & \\ {R(-1,-1)} && \yellowcell{\textcolor{rgb,255:red,92;green,214;blue,92}{R(-2,0)}} && \yellowcell{\textcolor{rgb,255:red,153;green,92;blue,214}{R(-3,1)}} \\ & {R(-2,-1)} && \yellowcell{\textcolor{rgb,255:red,92;green,214;blue,92}{R(-3,0)}} \\ \yellowcell{R(-2,-2)} && \yellowcell{R(-3,-1)} && \yellowcell{\textcolor{rgb,255:red,92;green,214;blue,92}{R(-4,0)}} \\ & \yellowcell{R(-3,-2)} && \yellowcell{R(-4,-1)} \\ \yellowcell{R(-3,-3)} && {R(-4,-2)} && \textcolor{rgb,255:red,92;green,214;blue,92}{R(-5,-1)} \\ & \yellowcell{R(-4,-3)} && \textcolor{rgb,255:red,92;green,214;blue,92}{R(-5,-2)} \\ \yellowcell{R(-4,-4)} && \textcolor{rgb,255:red,92;green,214;blue,92}{R(-5,-3)} && \textcolor{rgb,255:red,92;green,214;blue,92}{R(-6,-2)} \\ & \textcolor{rgb,255:red,92;green,214;blue,92}{R(-5,-4)} && \textcolor{rgb,255:red,92;green,214;blue,92}{R(-6,-3)} \end{tikzcd} 
\captionof{figure}{Let $\mu=-3$ so that $m=4$, $x=-2$ and $n=2$. The instances of relations \eqref{eq: 1grel3} which are defining relations in the Levendorskii presentation are marked in black font. Relations which can be deduced using Lemma \ref{lem: H1 rule lemma} are marked green, and the relation given by $J(5)$ is marked purple. The yellow background shows the relations assumed by Lemma \ref{lem: real-mag}. The diagram is drawn so that the diagonal decreases in the first entry and the antidiagonal increases in the second entry while moving to the right.} 
\label{fig: yellow1} 
\end{figure}

\begin{figure}[h]
\begin{tikzcd}[row sep=-9pt, column sep=-2pt] & \textcolor{rgb,255:red,92;green,214;blue,92}{{R(-1,0)}} && \textcolor{rgb,255:red,92;green,214;blue,92}{{R(-2,1)}} && \textcolor{rgb,255:red,92;green,214;blue,92}{{R(-3,2)}} \\ {R(-1,-1)} && \textcolor{rgb,255:red,92;green,214;blue,92}{{R(-2,0)}} && \yellowcell{\textcolor{rgb,255:red,92;green,214;blue,92}{R(-3,1)}} \\ & {R(-2,-1)} && \yellowcell{\textcolor{rgb,255:red,92;green,214;blue,92}{R(-3,0)}} && \yellowcell{\textcolor{rgb,255:red,153;green,92;blue,214}{{R(-4,1)}}} \\ {R(-2,-2)} && {R(-3,-1)} && \yellowcell{\textcolor{rgb,255:red,92;green,214;blue,92}{R(-4,0)}} \\ & {R(-3,-2)} && \yellowcell{{R(-4,-1)}} && \yellowcell{\textcolor{rgb,255:red,92;green,214;blue,92}{R(-5,0)}} \\ \yellowcell{R(-3,-3)} && {R(-4,-2)} && \yellowcell{R(-5,-1)} \\ & \yellowcell{R(-4,-3)} && {R(-5,-2)} && \textcolor{rgb,255:red,92;green,214;blue,92}{{R(-6,-1)}} \\ \yellowcell{R(-4,-4)} && {R(-5,-3)} && \textcolor{rgb,255:red,92;green,214;blue,92}{{R(-6,-2)}} \\ & \yellowcell{R(-5,-4)} && \textcolor{rgb,255:red,92;green,214;blue,92}{{R(-6,-3)}} && \textcolor{rgb,255:red,92;green,214;blue,92}{{R(-7,-2)}} \\ \yellowcell{R(-5,-5)} && \textcolor{rgb,255:red,92;green,214;blue,92}{{R(-6,-4)}} && \textcolor{rgb,255:red,92;green,214;blue,92}{{R(-7,-3)}} \\ & \textcolor{rgb,255:red,92;green,214;blue,92}{{R(-6,-5)}} && \textcolor{rgb,255:red,92;green,214;blue,92}{{R(-7,-4)}} && \textcolor{rgb,255:red,92;green,214;blue,92}{{R(-8,-3)}} \end{tikzcd} 
\captionof{figure}{Relations for $\mu=-4$, with $m=5$, $x=-3$ and $n=2$. } 
\label{fig: yellow2} 
\end{figure}

\subsection{Simultaneous induction}
\label{subsec: conditions}

We will simultaneously induct on the following, for $N \geq \mu - 1$:
\begin{itemize}
\item[($\mathtt{A}_N$)] $[H_1, \Theta_{m-\mu}] = 0$ for $1 \leq m \leq N+1$; 
\item[($\mathtt{B}_N$)] $R(r,s)$ for all $r,s \in \Z$ with $0 \leq |s-r| \leq N-\mu+1$;  
\item[($\mathtt{C}_N$)] $\overline{Y}(m-\mu,s) = 0$ for all $0 \leq m \leq N$ and $s \in \Z$; 
\item[($\mathtt{D}_N$)] $[\Theta_{m-\mu}, \Theta_{n-\mu}] = 0$ for $m,n \geq 1$ with $2 \leq m+n \leq N + 2$. 
\end{itemize}
This means that, proving, e.g., $(\mathtt{C}_{N+1})$, we are assuming $(\mathtt{A}_{N+1})$, $(\mathtt{B}_{N+1})$, $(\mathtt{C}_{N})$ and $(\mathtt{D}_{N})$. 

Whenever the choice of $N$ leads to one of the index ranges being empty, we declare that the corresponding condition holds by convention. This is the case for $(\mathtt{A}_N), (\mathtt{C}_N)$ and $(\mathtt{D}_N)$ with $N \leq -1$. The statements $(\mathtt{A}_0), (\mathtt{C}_0)$ and $(\mathtt{D}_0)$ hold trivially. Indeed, $(\mathtt{A}_{0})$ holds since $H_1 = \KK\mi \Theta_{1-\mu}$, condition $(\mathtt{C}_{0})$ reduces to \eqref{eq: ind_H1A} and $(\mathtt{D}_{0})$ is tautological. 

Condition $(\mathtt{A}_1)$, which is the same as the defining relation \eqref{eq: 1HHlev}, should be regarded as the base of the overall induction, with the caveat that an initial segment of the series $(\mathtt{B}_{N})$ is independent of the other conditions and requires a separate base of induction. This base consists of $(\mathtt{B}_{\mu-1})$--$(\mathtt{B}_{\mu+1})$ and is considered in \S \ref{ss: early B}. Since, as pointed out in \S \ref{ss: proof of B}, the inductive proof of $(\mathtt{B}_{N})$ only relies on $(\mathtt{B}_{N-1})$ if $N \leq 0$, this does not create any problems.  

 Therefore, the overall proof can be divided into three phases: (i) base cases $(\mathtt{B}_{\mu-1})$--$(\mathtt{B}_{\mu+1})$, (ii) induction to show $(\mathtt{B}_{N+1})$ for $N \leq -1$, (iii) simultaneous induction to show $(\mathtt{A}_{N+2})$ and $(\mathtt{B}_{N+1})$--$(\mathtt{D}_{N+1})$ for $N \geq 0$. The structure of the induction is illustrated in Figure \ref{fig: induction}.

\begin{figure}[h]
\begin{tikzcd}[column sep = 15pt, 
  execute at end picture={
    \node[
      draw,
      fit=(bmu-1)(bmu)(bmu+1),
      inner xsep=6pt,
      inner ysep=6pt
    ] (bbox) {};
    \draw[->]
      (bbox.east) --
      (bbox.east -| bmu+2.west);
    \draw[->]
      (c1.south) --
      (c1.south |- bbox.north);
  }
]
  & |[alias=c1]| {\mathtt{C}_1}
  &&&&&& {\mathtt{A}_1 = \mathtt{D}_1}
  & {\mathtt{A}_2}
  & {\mathtt{A}_3}
  \\
  |[alias=bmu-1]| {\mathtt{B}_{\mu-1}}
  & |[alias=bmu]| {\mathtt{B}_{\mu}}
  & |[alias=bmu+1]| {\mathtt{B}_{\mu+1}}
  & |[alias=bmu+2]| {\mathtt{B}_{\mu+2}}
  & \cdots
  & {\mathtt{B}_{-1}}
  & {\mathtt{B}_0}
  & {\mathtt{B}_1}
  & {\mathtt{B}_2}
  & {\mathtt{B}_3}
  \\
  &&&&&&&& {\mathtt{C}_2} & {\mathtt{C}_3}
  \\
  &&&&&&&& {\mathtt{D}_2} & {\mathtt{D}_3}
  \arrow[from=1-8, to=1-2]
  \arrow[from=1-8, to=2-8]
  \arrow[from=1-9, to=2-9]
  \arrow[from=1-10, to=2-10]
  \arrow[from=2-4, to=2-5]
  \arrow[from=2-5, to=2-6]
  \arrow[from=2-6, to=2-7]
  \arrow[from=2-7, to=2-8]
  \arrow[from=2-8, to=1-9]
  \arrow[from=2-9, to=3-9]
  \arrow[from=2-10, to=3-10]
  \arrow[from=3-9, to=4-9]
  \arrow[from=3-10, to=4-10]
  \arrow[from=4-9, to=1-10]
\end{tikzcd}
\captionof{figure}{Schematic explanation of the flow of the induction.} 
\label{fig: induction} 
\end{figure}

\subsection{Proof of $(\mathtt{B}_{\mu-1})$--$(\mathtt{B}_{\mu+1})$}
\label{ss: early B}

Below we establish the base of the induction for the series $(\mathtt{B}_{N})$. 

\begin{assum}
Only Lemma \ref{lem: H2 comm} is assumed. 
\end{assum}

The following relations are instances of the defining Levendorskii relations:
\begin{align}
&R(s,s) \quad &&\mbox{for   } \ \mu - 1 \leq s \leq -1, \\
&R(s-1,s) \quad &&\mbox{for   }\ \mu \leq s \leq -1, \\
&R(s-2,s) \quad &&\mbox{for   }\ \mu + 1 \leq s \leq -1. 
\end{align}
If $\mu = -1$, the third range above is empty, but $R(-2,0)$ holds by definition $J(3)$. 
The specific relations for $\mu = -1, -2$ are illustrated in Figures \ref{fig: pra}--\ref{fig: prb}.

\begin{figure}[h]
\begin{tikzcd}[row sep=-10pt]
	{R(-1,-1)} && \\
	& {R(-2,-1)} \\
	{R(-2,-2)} && {R(-3,-1)} \\
	& {R(-3,-2)} \\
	{R(-3,-3)}
\end{tikzcd}
\captionof{figure}{Instances of relations \eqref{eq: 1grel3} which are defining relations in the Levendorskii presentation for $\mu = -2$. The triangular pattern is typical. } 
\label{fig: pra}
\end{figure} 

\begin{figure}[h]
\begin{tikzcd}[row sep=-10pt]
	{R(-1,-1)} && R(-2,0) \\
	& {R(-2,-1)} \\
	{R(-2,-2)} && 
\end{tikzcd}
\captionof{figure}{Instances of relations \eqref{eq: 1grel3} which are defining relations in the Levendorskii presentation, or hold by $J(3)$, for $\mu = -1$. } 
\label{fig: prb}
\end{figure}

Since $R_R(s,s)$ and $R_R(s-1,s)$ are always a (possibly zero) scalar (for any $s \in \Z$), we get 
\eq \label{eq: H RR}
[\overline{H}_1, R_R(s,s)] = 0, \qquad [\overline{H}_2, R_R(s,s)] = 0, \qquad [\overline{H}_1, R_R(s-1,s)] = 0.  
\eneq
Moreover, \eqref{eq: ind_H1A}--\eqref{eg: rule H1} and Lemma \ref{lem: H2 comm} imply that 
\begin{align}
[\overline{H}_1, R_L(s,s)] &= 2 R_L(s,s+1) - 2 \Ck R_L(s-1,s), \label{eq: H1 R comm} \\
[\overline{H}_2, R_L(s,s)] &= 2R_L(s,s+2) - 2 \Ck^2 R_L(s-2,s),  \label{eq: H2 R comm}  \\
[\overline{H}_1, R_L(s,s+1)] &= R_L(s+1,s+1) + R_L(s,s+2)  \label{eq: H1 R comm2} \\
& \quad  - \Ck R_L(s,s) - \Ck R_L(s-1,s+1).  
\end{align}
Applying the last equality with $s=-2$ and $s=\mu-1$ (together with \eqref{eq: H RR}, $R(-2,-1)$, $R(\mu-1,\mu)$ and Lemma \ref{lem: H1 rule lemma}), we get $R(-2,0)$ and $R(\mu-2,\mu)$. This step is illustrated below (for $\mu = -2$) in Figure \ref{fig: pr} below. 
\vspace{0.25cm}

\begin{figure}[h] 
\begin{tikzcd}[row sep=-7pt]
	{R(-1,-1)} && \textcolor{rgb,255:red,92;green,214;blue,92}{{R(-2,0)}} \\
	& \textcolor{rgb,255:red,69;green,89;blue,237}{{R(-2,-1)}} \\
	{R(-2,-2)} && {R(-3,-1)} \\
	& \textcolor{rgb,255:red,69;green,89;blue,237}{{R(-3,-2)}} \\
	{R(-3,-3)} && \textcolor{rgb,255:red,92;green,214;blue,92}{{R(-4,-2)}}
	\arrow[color={rgb,255:red,69;green,89;blue,237}, no head, from=1-1, to=2-2]
	\arrow[color={rgb,255:red,92;green,214;blue,92}, from=2-2, to=1-3]
	\arrow[color={rgb,255:red,69;green,89;blue,237}, no head, from=3-1, to=2-2]
	\arrow[color={rgb,255:red,69;green,89;blue,237}, no head, from=3-1, to=4-2]
	\arrow[color={rgb,255:red,69;green,89;blue,237}, no head, from=3-3, to=2-2]
	\arrow[color={rgb,255:red,69;green,89;blue,237}, no head, from=3-3, to=4-2]
	\arrow[color={rgb,255:red,92;green,214;blue,92}, from=4-2, to=5-3]
	\arrow[color={rgb,255:red,69;green,89;blue,237}, no head, from=5-1, to=4-2]
\end{tikzcd}
\captionof{figure}{Propagation using \eqref{eq: H1 R comm2}. Black and blue relations are assumed to be known. The commutator is applied to the blue relations. The blue lines connect to the other relevant known relations. The green relation, also indicated by the green arrow, is deduced.} 
\label{fig: pr}
\end{figure}

This establishes the base of the induction. All the desired relations are now obtained using upward and downward induction, with upward inductive hypothesis ($s \geq -1$)
\begin{align}
\operatorname{IH}^\uparrow(s):&\qquad& R(s,s), \quad R(s-1,s), \quad R(s-2,s), \quad R(s-1,s+1), 
\end{align}
and downward inductive hypothesis ($s \leq -\mu-1$) 
\begin{align}
\operatorname{IH}^\downarrow(s):&\qquad& R(s,s), \quad R(s,s+1), \quad R(s,s+2), \quad R(s-1,s+1). 
\end{align}
In the upward case, the inductive hypothesis $\operatorname{IH}^\uparrow(s)$ and \eqref{eq: H RR}, together with \eqref{eq: H1 R comm}--\eqref{eq: H2 R comm}, yield $R(s,s+1)$ and $R(s, s+2)$. 
 Then \eqref{eq: H1 R comm2} yields $R(s+1, s+1)$, which establishes $\operatorname{IH}^\uparrow(s+1)$. These steps are illustrated in Figures \ref{fig: pr2}--\ref{fig: pr3} below.  The downward induction is analogous. 

\begin{figure}[h] 
\begin{tikzcd}[row sep=-7pt]
	&& \textcolor{rgb,255:red,92;green,214;blue,92}{{R(s,s+2)}} \\
	& \textcolor{rgb,255:red,92;green,214;blue,92}{{R(s,s+1)}} \\
	\textcolor{rgb,255:red,69;green,89;blue,237}{{R(s,s)}} && {R(s-1,s+1)} \\
	& {R(s-1,s)} \\
	&& {R(s-2,s)}
	\arrow[color={rgb,255:red,92;green,214;blue,92}, bend left=30, squiggly, from=3-1, to=1-3]
	\arrow[color={rgb,255:red,92;green,214;blue,92}, from=3-1, to=2-2]
	\arrow[color={rgb,255:red,69;green,89;blue,237}, no head, from=4-2, to=3-1]
	\arrow[color={rgb,255:red,69;green,89;blue,237}, bend left=30, squiggly, no head, from=5-3, to=3-1]
\end{tikzcd}
\captionof{figure}{Propagation using \eqref{eq: H1 R comm} and \eqref{eq: H2 R comm}.}
\label{fig: pr2}
\end{figure}

\begin{figure}[h] 
\begin{tikzcd}[row sep=-7pt]
	\textcolor{rgb,255:red,92;green,214;blue,92}{{R(s+1,s+1)}} && {R(s,s+2)} \\
	& \textcolor{rgb,255:red,69;green,89;blue,237}{{R(s,s+1)}} \\
	{R(s,s)} && {R(s-1,s+1)} \\
	& {R(s-1,s)} \\
	&& {R(s-2,s)}
	\arrow[color={rgb,255:red,69;green,89;blue,237}, no head, from=1-3, to=2-2]
	\arrow[color={rgb,255:red,92;green,214;blue,92}, from=2-2, to=1-1]
	\arrow[color={rgb,255:red,69;green,89;blue,237}, no head, from=3-1, to=2-2]
	\arrow[color={rgb,255:red,69;green,89;blue,237}, no head, from=3-3, to=2-2]
\end{tikzcd}
\captionof{figure}{Propagation using \eqref{eq: H1 R comm2}.}
\label{fig: pr3}
\end{figure}


\subsection{The inductive step} 

We will now prove $(\mathtt{A}_{N+1})$--$(\mathtt{D}_{N+1})$, where $N$ has the following range for each series of conditions: 
\[
\mathtt{A} : N \geq 1, \qquad \mathtt{B} : N \geq \mu+1 \geq 0, \qquad \mathtt{C}-\mathtt{D} : N \geq 0. 
\]
Figure \ref{fig: log depend} shows the logical dependencies between the different conditions. 
\begin{figure}[h]
\begin{tikzcd}
    & {\mathtt{A}_{N+1}} \\
    {\mathtt{B}_N} & {\mathtt{B}_{N+1}} \\
    & {\mathtt{C}_{N+1}} \\
    {\mathtt{D}_N} & {\mathtt{D}_{N+1}}
    \arrow[from=1-2, to=2-2]
    \arrow[bend left=50, from=1-2, to=3-2]
    \arrow[from=2-1, to=1-2]
    \arrow[from=2-1, to=2-2]
    \arrow[from=2-1, to=3-2]
    \arrow[bend left=50, from=2-2, to=4-2]
    \arrow[from=3-2, to=4-2]
    \arrow[bend left=75, from=4-1, to=1-2]
\end{tikzcd}
\caption{Logical dependencies between the inductive conditions.}
\label{fig: log depend}
\end{figure}

\subsubsection{Preparation} 

In this step we will use the following assumptions.     

\begin{assum}
Only $(\mathtt{C}_{N})$ and $(\mathtt{D}_{N-1})$ are assumed. 
\end{assum}

\nc{\sh}{{\mathsf{h}}}

Consider the polynomial algebra $B:=\C(q)[\{x_r\}_{m=1}^\infty]$,
with $\N$-grading given by $\deg(x_m)=m$. Define elements
$\{\sh_m\}_{m=1}^\infty$ of $B$ via
\[
(q-q\mi)\mi \exp\left((q-q^{-1})\sum_{m=1}^\infty \sh_rz^{r}\right)=1+\sum_{m=1}^\infty x_m z^{m}.
\]
Then $\sh_m = x_m + p_m(x_1, \cdots, x_{m-1})$ with $p_m$ satisfying $\deg(p_m(x_1,\ldots,x_{r-1}))=m$. Using the polynomials $p_m$, we define
$H_{1},\cdots,H_{N+1}\in \Vi_\mu$ as 
\begin{equation}\label{def: h_k}
  H_{m}:= \KK^{-1}\Theta_{m-\mu} +
           p_m(\KK^{-1}\Theta_{1-\mu},\cdots, \KK^{-1}\Theta_{m-\mu-1})
  \qquad (1\leq m \leq N+1).
\end{equation}
These $H_{m}$ are well-defined and are independent of the choice of $N>m-1$ thanks to assumption $(\mathtt{D}_{N-1})$ and the aforementioned degree condition on $p_m$. Set 
\eq \label{eq: Hm overline}
\overline{H}_m = \textstyle\frac{m}{[2m]} H_m.  
\eneq
Using $(\mathtt{C}_{N})$ and similar arguments to those in Lemma \ref{lem: Y-H equiv}, one gets 
\begin{align}
[\overline{H}_m, A_s] &= A_{s+m} - \Ck^m A_{s-m}, \\ \label{eq: Hm comm}
\ \ \ [\overline{H}_m, R(r,s)] &= R_L(r{+}m, s) + R_L(r, m {+} s) - \Ck^m R_L(r{-}m, s) - \Ck^m R_L(r, s{-}m). 
\end{align}

\subsubsection{Proof of $(\mathtt{A}_{N+1})$}  

We need to prove $[H_1, \Theta_{N+2-\mu}] = 0$ or, equivalently, \linebreak $[\overline{H}_1, \overline{\Theta}_{N+2-\mu}] = 0$. By $(\mathtt{B}_N)$, 
\[
\overline{\Theta}_{N+2-\mu} = R_L(r, s) \qquad (r \in \Z, \ s = N+1+r-\mu). 
\] 
Hence
\begin{align*}
[\overline{H}_1, \overline{\Theta}_{N+2-\mu}] &= R_L(r+1, s) + R_L(r, s+1) - \Ck R_L(r-1, s) - \Ck R_L(r, s-1) \\ 
&= R_L(r, s+1) - \Ck R_L(r-1, s),
\end{align*}
where the first equality follows from \eqref{eq: ind_H1A} and the second equality follows from $(\mathtt{B}_{N-1})$. 
Adding these up for $r = 0, 1$, we get 
\[
2 [\overline{H}_1, \overline{\Theta}_{N+2-\mu}] = R_L(0, N-\mu+2) - \Ck^2 R_L(-2, N-\mu). 
\]
On the other hand, 
\begin{align*}
[\overline{H}_2, \overline{\Theta}_{N+1-\mu}] &= [\overline{H}_2, R_L(0, N-\mu)] = R_L(2, N-\mu) + R_L(0, N-\mu+2) \\
& \quad - \Ck^2 R_L(-2, N-\mu) - \Ck^2 R_L(r, N-\mu-2) \\ 
&= R_L(0, N-\mu+2) - \Ck^2 R_L(-2, N-\mu), 
\end{align*}
by $(\mathtt{B}_{N-1})$ and \eqref{eq: Hm comm}. Hence
\[
2 [\overline{H}_1, \overline{\Theta}_{N+2-\mu}] = [\overline{H}_2, \overline{\Theta}_{N+1-\mu}]. 
\]
By a similiar reasoning, we also get 
\[
(N+1) [\overline{H}_1, \overline{\Theta}_{N+2-\mu}] = [\overline{H}_{N+1}, \overline{\Theta}_{2-\mu}]. 
\]
On the other hand, \eqref{def: h_k} and $(\mathtt{D}_N)$ yield 
\[
[\overline{H}_m, \overline{\Theta}_{N+3-\mu-m}] = \textstyle\frac{m}{[2m]}\KK\mi [\overline{\Theta}_{m-\mu}, \overline{\Theta}_{N+3-\mu-m}]
\]
for $1 \leq m \leq N+1$. Therefore, 
\[
[\overline{H}_1, \overline{\Theta}_{N+2-\mu}] = [4]\mi\KK\mi [\overline{\Theta}_{2-\mu}, \overline{\Theta}_{N+1-\mu}] = [2N+2]\mi\KK\mi [\overline{\Theta}_{N+1-\mu}, \overline{\Theta}_{2-\mu}]. 
\]
If $N = 1$ then the RHS is manifestly zero since $[\overline{\Theta}_{2-\mu}, \overline{\Theta}_{2-\mu}] = 0$. Otherwise, $[4] \neq -[2N+2]$ forces $[\overline{\Theta}_{N+1-\mu}, \overline{\Theta}_{2-\mu}] = 0$, finishing the proof. 

\subsubsection{Proof of $(\mathtt{B}_{N+1})$}
\label{ss: proof of B}

In this step we will use the following assumptions.

\begin{assum}
Only $(\mathtt{B}_{N})$ and $(\mathtt{A}_{N+1})$ are assumed. If $\mu \leq N \leq -1$ then $(\mathtt{A}_{N+1})$ is trivial. 
\end{assum}

Let $N \geq \mu+1$. 
Take any $s \geq r$ with $s-r = N-\mu +2$. If we apply $[\overline{H}_1, - ]$ to $R(r,s-1)$, we obtain 
\begin{align}
0 \xeq{(\mathtt{A}_{N+1})} \Ck^r [\btheta_{N{-}\mu{+}2}, \overline{H}_1] &\xeq{(\mathtt{B}_{N})} [\overline{H}_1, R_L(r,s{-}1)] \notag \\
&\xeq{\eqref{eq: ind_H1A}} R_L(r,s) - \Ck R_L(r{-}1,s{-}1) + R_L(r{+}1,s{-}1) - \Ck R_L(r,s{-}2) \\ 
&\xeq{(\mathtt{B}_{N-1})} R_L(r,s) - \Ck R_L(r{-}1,s{-}1). \label{eq: RLrs}
\end{align}
We will now perform nested upward and downward inductions. Let $x = -\lceil \frac{N-\mu}{2} \rceil -2$ and $y = N-\mu+2$ and set  
\[
\mathtt{B}_{N+1}(r) = 
\begin{cases}
R(m, m+y) \text{ holds for } x \leq m \leq r & \text{if } r \geq x, \\
R(m, m+y) \text{ holds for } x \geq m \geq r & \text{if } r \leq x.   
\end{cases}
\] 
By the inductive definition $J(N-\mu+3)$ (if $N \geq -1$), or the defining relations (otherwise), we know that 
$R(x, x+y)$, which is equivalent to $\mathtt{B}_{N+1}(x)$.  

Let $r \geq \mu$. Assuming $\mathtt{B}_{N+1}(r)$, we have 
\[
R_L(r+1,r+y+1) \xeq{\eqref{eq: RLrs}} \Ck R_L(r,r{+}y)  \xeq{\mathtt{B}_{N+1}(r)} \Ck^{r+1} \btheta_{N-\mu+3}, 
\] 
which proves $\mathtt{B}_{N+1}(r+1)$. 
The argument for $r < \mu$ is similar. Assuming $\mathtt{B}_{N+1}(r)$, we get 
\[
R_L(r{-}1,r+y{-}1) \xeq{\eqref{eq: RLrs}} \Ck\mi R_L(r,r+y)  \xeq{\mathtt{B}_{N+1}(r)} \Ck^{r-1} \btheta_{N-\mu+3}, 
\] 
which proves $\mathtt{B}_{N+1}(r-1)$.

\subsubsection{Proof of $(\mathtt{C}_{N+1})$}   
\label{sss: CN}

In this step we will use the following assumptions.     

\begin{assum}
Only $(\mathtt{A}_{N+1})$ and $(\mathtt{B}_{N})$ are assumed. 
\end{assum}

If $N=0$ then $\mathtt{C}_{1}$ is equivalent to Lemma \ref{lem: H2 comm}. So let $N \geq 1$. 
It suffices to show that $\overline{Y}(N-\mu+1, s) = 0$ for all $s \in \Z$. This will follow from Lemma \ref{lem: real-mag}, with $m=N-\mu+1$, if we check that the assumptions are satisfied. Condition $(\mathtt{A}_{N+1})$ holds by construction of the inductive loop. By $(\mathtt{B}_{N})$, the relation $R(r,s)$ holds whenever $|s-r| \leq m$, which covers all the cases required by Lemma \ref{lem: real-mag}. 

\subsubsection{Proof of $(\mathtt{D}_{N+1})$}  
In this step we will use the following assumptions.\footnote{For $N=0$, the condition $(\mathtt{D}_1)$ is equivalent to      $(\mathtt{A}_1)$. The proof below still works, but it is redundant. Also note that we cannot use this argument to derive $(\mathtt{A}_1)$ since it assumes $(\mathtt{B}_1)$.} 

\begin{assum}
Only $(\mathtt{B}_{N+1})$ and $(\mathtt{C}_{N+1})$ are assumed. 
\end{assum}

If $m$ or $n$ equals zero, the assertion is trivial, so let us take $1 \leq m,n$ with $m+n \leq N+3$. It suffices to prove the equivalent statement $[\overline{H}_m, \Theta_{n-\mu}] = 0$. Indeed, 
\begin{align*}
[\overline{H}_m, \Theta_{n-\mu}] &\xeq{(\mathtt{B}_{N})}  [\overline{H}_m, R(0, n-\mu-1)] \\
&\xeq{(\mathtt{C}_{N+1})} R(m, n-\mu-1) + R(0, m+n-\mu-1) \\ 
&\qquad - \Ck^m R(-m,n-\mu-1) - \Ck^m R(0,n-m-\mu-1) \xeq{(\mathtt{B}_{N+1})} 0. 
\end{align*}

\subsection{Proof of the real--imaginary commutation relations} 

The goal of this subsection is to prove the key Lemma \ref{lem: real-mag}, which was used in induction step $(\mathtt{C}_{N+1})$ and the preliminary Lemma \ref{lem: H2 comm}. 

\Lem \label{lem: key id}
For any $s,t \in \Z$, the following identity holds: 
\begin{align}
\ \ &[R_L(s-1, t-1), A_s] + [R_L(s, t-2), A_s] \ + \\ \label{eq: key id}
\ \ &+  q^{-2} \big( [A_{s+1}, R_L(s-1,t-2)]_{q^4} - [R_L(s,t-1), A_{s-1}]_{q^4} \big) = \\ 
\ \ = \ & [A_{t-1}, R_L(s,s-1)]  +  \textstyle\frac12 q^{-2} \big( [R_L(s,s), A_{t-2}]_{q^4} - [A_{t}, R_L(s-1,s-1)]_{q^4}   \big). 
\end{align}
\enlem

\Proof
This statement is a general algebraic identity and the proof does not involve any relations specific to the $q$-Onsager algebra. We will compare the LHS and RHS of \eqref{eq: key id} by decomposing them into homogeneous polynomials in the generators. We will use the following auxiliary identities, which hold for arbitrary elements $a,b,c$ in a free non-commutative algebra and are easily verified by direct calculation: 
\begin{align} \label{eq: abc 1}
[[a,c]_{q^2},b] - q^{-2} [[a,b]_{q^2},c]_{q^4} &= - q^{-2} [a,[b,c]_{q^2}]_{q^4}, \\ \label{eq: abc 2} 
[[a,b]_{q^2}, a] + [[b,a]_{q^2},a] &=   [b, [a,a]_{q^2}], \\ \label{eq: abc 3} 
q^{-2}\big( [a,[b,c]_{q^2}]_{q^4} - [[a,b]_{q^2}, c]_{q^4} \big) &= [b, [a,c]_{q^2}].  
\end{align}
Using 
\eqref{eq: abc 1} in the first equation below,  \eqref{eq: abc 2} in the second, and \eqref{eq: abc 3} in the third and fourth equations, we get 
\begin{align*}
[[A_t, A_{s-1}]_{q^2}, A_s] -q^{-2} [[A_t, A_s]_{q^2}, A_{s-1}]_{q^4} &= - q^{-2}[A_{t}, [A_s, A_{s-1}]_{q^2}]_{q^4}, \\
 [[A_s,A_{t-1}]_{q^2}, A_s] + [[A_{t-1}, A_s]_{q^2}, A_s] &= [A_{t-1}, [A_s, A_s]_{q^2}], \\ 
q^{-2}\big( [A_{s+1}, [A_{t-1}, A_{s-1}]_{q^2}]_{q^4} - [[A_{s+1}, A_{t-1}]_{q^2}, A_{s-1}]_{q^4} \big) &= 
[A_{t-1}, [A_{s+1}, {A_{s-1}}]_{q^2}], \\ 
[[A_{s+1}, A_{t-2}]_{q^2}, A_s] + q^{-2} [A_{s+1}, [A_s, A_{t-2}]_{q^2}] &= q^{-2} [[A_{s+1}, A_{s}]_{q^2}, A_{t-2}]_{q^4}. 
\end{align*}
Summing the above equalities sidewise, we recover \eqref{eq: key id}. 
\enproof

Pick $m \geq 1 - \mu \geq 2$ and let 
\[
x = - \lceil \textstyle\frac{m}{2} \rceil, \qquad n = m + x = \lfloor \textstyle\frac{m}{2} \rfloor.  
\] 

\Lem \label{lem: real-mag}
Suppose that $(\mathtt{A}_{m+\mu})$ and relations
\begin{align*}
&R(x, x), \quad &&R(x-1, x-1), \quad &&R(x-2, x-2), \\ 
&R(x, x-1), \quad &&R(x-1,x-2), \quad && R(x-2,n-3), \\ 
&R(x,n-1), \quad &&R(x-1, n-2), \quad && R(x-1, n-1), \\
&R(x, n-2), \quad &&R(x-1, n-3), \quad && R(x-2, n-2) 
\end{align*}
hold. Then $\overline{Y}(m,s) = 0$, for all $s \in \Z$. Moreover, if \((\mu,m)=(-1,2)\) or \((\mu,m)=(-2,3)\) then the assumption $R(x-2, x-2)$ is not necessary. 
\enlem

\Proof
By definition, 
\eq \label{eq: Yms proof}
\overline{Y}(m,x) = 
[\btheta_{m+1}, A_{x}] + [\btheta_{m-1}, A_{x}] \Ck + q^{-2}\left( [A_{x+1}, \btheta_{m}]_{q^{4}} -  [\btheta_{m}, A_{x-1}]_{q^{4}}\Ck \right). 
\eneq
Applying $R(x-1,n-1)$, $R(x, n-2)$, $R(x-1, n-2)$ and $R(x, n-1)$, the RHS of \eqref{eq: Yms proof} becomes the LHS of \eqref{eq: key id}, with $s=x$ and $t=n$, multiplied by $-\Ck^{1-x}$. 
Hence, by Lemma \ref{lem: key id}, and using $R(x,x-1)$, $R(x,x)$ and $R(x-1,x-1)$, we get  
\[
\overline{Y}(m,x) = [A_{n-1}, \Theta_2]  +  q^{-2} \big( \Ck [\Theta_1, A_{n-2}]_{q^4} - [A_{n}, \Theta_1]_{q^4} \big) = - Y(1, n-1) = 0. 
\]
The last equality holds trivially if $\mu \leq -2$, or follows from \eqref{eq: ind_H1A} if $\mu = -1$. An analogous calculation, with shifts $x \mapsto x-1$ and $n \mapsto n-1$ throughout, yields $\overline{Y}(m,x-1) = 0$. The new calculation uses the relations
\begin{align*}
&R(x-2,n-2), \quad &&R(x-1, n-3), \quad &&R(x-2, n-3), \\ 
&R(x-1, n-2), \quad &&R(x-1,x-2), \quad &&R(x-1, x-1), \\
&R(x-2, x-2).  
\end{align*}
These seven relations have two relations in common with those used in the previous calculation. Together they form the twelve relations assumed in the statement of the Lemma. If \((\mu,m)=(-1,2)\) or \((\mu,m)=(-2,3)\) then $\overline{Y}(m,x-1) = 0$ can also be established without assumption $R(x-2, x-2)$ - we do this computationally in Sage, using Gr\"{o}bner basis methods (the GBNP GAP Package \cite{GBNP}), see \cite{LevQSP}. 

Finally, we deduce that $\overline{Y}(m,s) = 0$, for all $s$, using upward and downward induction. Consider the upward case. Let $s \geq x+1$ and assume that $\overline{Y}(m,r)$ holds for $x-1 \leq r \leq s-1$. 
Then 
\begin{align*} 
\overline{Y}(m,s) &= 
[\Theta_{m+1}, A_{s}] + [\Theta_{m-1}, A_{s}] \Ck + q^{-2}\big( [A_{s+1}, \Theta_{m}]_{q^{4}} -  [\Theta_{m}, A_{s-1}]_{q^{4}}\Ck \big) \\ 
&= [\bH_1, \overline{Y}(m,s-1)] + \Ck \overline{Y}(m, s-2) = 0. 
\end{align*} 
The second equality holds by \eqref{eq: ind_H1A} and $(\mathtt{A}_{m+\mu})$, and the third by the inductive hypothesis. 
The downward induction is similar. 
\enproof

\section{General simply-laced split case} 
\label{sec: gen sl split c}

In this section, we generalize the results and methods from \S \ref{s: qOns}--\S \ref{sec: Levendorskii proof} to shifted quantum affine symmetric pair algebras of an arbitrary simply-laced split type. In particular, we establish a Levendorskii presentation for these algebras for any anti-dominant shift. The main new difficulty is to derive all the $\imath$Drinfeld--Serre relations. This is tackled in \S \ref{sss: iSerre proof gen}. 
The proof uses a version of techniques from \S \ref{sec: Levendorskii proof}: we first establish some general identities without using $\imath$Drinfeld--Serre relations (Lemma \ref{lem: Serre key lem}), and then use these to set up an induction, where the base case needs to be treated with special care. This is somewhat parallel to our treatment of the quadratic real-imaginary and real-real relations in \S \ref{sss: iSerre proof gen}. 
This aspect again distinguishes the quantum symmetric pair setting from quantum affine algebras - the derivation of the Drinfeld--Serre relations in the latter case (c.f.\ \cite{FT19}) is rather straightforward. 

\subsection{Shifted quantum affine symmetric pair algebras} 
\label{subsec: Gamma}

Consider a simply-laced undirected graph $\boldsymbol\Gamma$ without edge-loops and with a finite vertex set $\indx$. Let $(a_{ij})_{i,j \in \indx}$ be the corresponding generalized Cartan matrix, to which we associate Lie-theoretic data in the usual way. In particular, let $\Lambda$ (resp. $\Lambda^\vee$) denote the coweight (resp.\ weight) lattice, with simple roots $\alpha^\vee_i$ and fundamental weights $\omega^\vee_i$ ($i \in \indx$). We endow $\Lambda$ with the usual partial order, i.e., $\mu \leq \lambda$ if and only if $\lambda - \mu$ can be written as a sum of positive coroots with non-negative integral coefficients. Let $\Lambda_- \subset \Lambda$ denote the set of anti-dominant coweights (i.e., $\ai(\mu) \leq 0$ for all $i \in \indx$). 

\subsubsection{Series definition}

Let $\mu \in \Lambda$ be \emph{any} coweight and set 
\[
\mu_i = \ai(\mu). 
\]
We will use the following notation for generating series: 
\[
\As{i}{z} = \sum_{r \in \Z} A_{i,r} z^r, \quad \Ths{i}{z} = \KK_i (q-q\mi)\mi z^{-\mu_i} +  \sum_{r > -\mu_i} \Theta_{i,r} z^r, \quad \Dels = \sum_{r \in \Z} \Ck^r z^r. 
\]
By convention, $\Theta_{i,-\mu_i} = \KK_i (q-q\mi)\mi$ and $\Theta_{i,r} = 0$ for $r < -\mu_i$. 
Below we define the $\mu$\emph{-shifted quantum affine symmetric pair algebra} $\Ui_\mu$ associated to $\boldsymbol\Gamma$. Note that $\boldsymbol\Gamma$ need not be of finite type. In particular, a simply-laced graph of affine type yields a ``toroidal'' or ``double affine'' shifted quantum symmetric pair algebra. 

\Defi \label{defi: shifted iqsp}
Let $\Ui_\mu$ be the algebra generated by $A_{i,s}$ and 
$\Theta_{i,r}$, as well as central elements 
$\KK_{i}^{\pm1}$, $\Ck^{\pm1}$ $(i\in \indx, \ r > -\mu_i, \ s\in\Z)$, subject to the following relations: 
\begin{align} 
\bTh_i(z) \bTh_j(w) &= \bTh_j(w) \bTh_i(z), \label{eq: drel1} \\[2pt]
\bTh_i (z)  \bA_j(w) &= \frac{(1 -q^{-a_{ij}}zw^{-1}) (1 -q^{a_{ij}} zw \CCC)}{(1 -q^{a_{ij}}zw^{-1})(1 -q^{-a_{ij}}zw \CCC)} \bA_j(w) \bTh_i (z), \label{eq: drel2} \\[2pt] 
\bA_i(z)\bA_j(w) &= \bA_j(w)\bA_i(z), \qquad\qquad\qquad\qquad \text{ if }a_{ij}=0,  \label{eq: drel3} \\[2pt]
(q^{-1}z -w) \bA_i(z) \bA_j(w) &= (z-q^{-1}w) \bA_j(w) \bA_i(z), \quad \quad \quad \ \text{ if }a_{ij}=-1, \label{eq: drel4} \\[2pt]
\Sym_{z,w} (q^2z - w) \bA_i(z) \bA_i(w) &= 
(w-z) \bDel(zw) \big(\bTh_i(z) - \bTh_i(w) \big) \quad \text{ if }  a_{ij} =2, 
 \label{eq: drel5}
\end{align} 
\begin{align} 
&(q w_2 - z)(q w_1 - z) \Sym_{w_1,w_2}[ \bA_i(w_1), [\bA_i(w_2), \bA_j(z)]_{q\mi}]_q = \\  \label{eq: drel6}
&\qquad\qquad = \bDel(w_1w_2) z (q-q\mi)( w_1 - w_2 ) (  \bTh_i(w_2) - \bTh_i(w_1) ) \bA_j(z)
\qquad 
\text{ if }a_{ij}=-1. 
\end{align} 
\enDefi 

\noeqref{eq: drel2} 
\noeqref{eq: drel3}
\noeqref{eq: drel4}

The algebra $\Ui_\mu$ defined above is the same as the one introduced in \cite{LP-gklo}, but there are some differences in the presentation. Our series $\Ths{i}{z}$ corresponds to $(q-q\mi)\mi\KK_i\aThs{i}{z}$ from \emph{op.\ cit}. The equivalence of \eqref{eq: drel5} and \cite[(2.5)]{LP-gklo} was shown in \cite[\S 4.2.3]{LP-gklo} while the equivalence of \eqref{eq: drel6} and \cite[(2.6)]{LP-gklo} was demonstrated in \cite[Lemmata 4.2, 4.5]{SW-C} (see also \cite[Lemma 2.3]{LWW-i}). We will refer to \eqref{eq: drel6} as the $\imath$Drinfeld--Serre relations. 

\subsubsection{Relations term-by-term I} 

Relations \eqref{eq: drel1}--\eqref{eq: drel5} from Definition \ref{defi: shifted iqsp} can be written in terms of modes, using 
\[
(q-q\mi)  \Ths{i}{z} z^{\mu} = \KK_i \exp \left( (q-q^{-1}) \sum_{r \geq 1}^\infty \Hg_{i,r} z^r \right), 
\] 
in the following way: 
\begin{align} 
\ \ \ \ P_{i,j}(m,n): && [\Theta_{i, m},\Theta_{j, n}] &=0, \label{eq: grel1} \\
\ \ \ \ Q_{i,j}(m,r): & & [H_{i,m}, A_{j,r}] &= \textstyle\frac{[ma_{ij}]}{m} (A_{j,r+m} - A_{j,r-m}\CCC^m), \label{eq: grel2}  \\ 
\label{eq: grel3a}
\ \ \ \ R_{i,j}(r,s): &  & [\Ag_{j, s}, \Ag_{i, r}] 
&= 0 \qquad \text{ if }a_{ij}=0,\\
\label{eq: grel3b}
\ \ \ \ R_{i,j}(r,s): & & [\Ag_{j, s+1}, \Ag_{i, r}]_{q^{-1}} + [\Ag_{i, r+1}, \Ag_{j, s}]_{q^{-1}}
&= 0 \qquad \text{ if }a_{ij}=-1, \\ 
\label{eq: grel3c}
\ \ \ \ R_{i,i}(r,s): &  &  [\Ag_{i, s+1}, \Ag_{i, r}]_{q^{2}} + [\Ag_{i, r+1}, \Ag_{i, s}]_{q^{2}}
&= -\CCC^r\Thg_{i, s-r+1} + \CCC^{r+1}\Thg_{i, s-r-1}  \\ \notag
&\qquad& &\quad  - \CCC^s\Thg_{i, r-s+1} + \CCC^{s+1}\Thg_{i, r-s-1}.
\end{align} 

Denote the LHS of $R_{i,j}(r,s)$ as $R_{L,i,j}(r,s)$, and abbreviate $\btheta_{i, m} = \Theta_{i,m} - \delta_{m \neq 2} \CCC \Theta_{i,m-2}$. 
Relations \eqref{eq: grel2} can also be written equivalently (c.f.\ Lemma \ref{lem: Y-H equiv}, \cite[Proposition 3.12]{lu-wang-21}) as  
\begin{align} \label{eq: Yijms1}
Y_{i,j}(m,s) &:= 
[\Theta_{i,m+1}, A_{j,s}] + [\Theta_{i,m-1}, A_{j,s}] \Ck \\ 
&\quad + q^{-a_{ij}}\left( [A_{j,s+1}, \Theta_{i,m}]_{q^{2a_{ij}}} -  [\Theta_{i,m}, A_{j,s-1}]_{q^{2a_{ij}}}\Ck \right) = 0. 
\end{align}
Set $\overline{Y}_{i,j}(m,s) = Y_{i,j}(m,s) - \delta_{m\neq2}\CCC Y_{i,j}(m{-}2,s)$. If $\mu \in \Lambda_-$ then 
\begin{align} \label{eq: Yijms}
\overline{Y}_{i,j}(m,s) &= 
[\btheta_{i,m+1}, A_{j,s}] + [\btheta_{i,m-1}, A_{j,s}] \Ck \\
&\quad + q^{-a_{ij}}\left( [A_{j,s+1}, \btheta_{i,m}]_{q^{2a_{ij}}} -  [\btheta_{i,m}, A_{j,s-1}]_{q^{2a_{ij}}}\Ck \right). 
\end{align}  

\subsubsection{Relations term-by-term II} 

We next consider the $\imath$Drinfeld--Serre relation \eqref{eq: drel6}. 
Take $i,j\in \indx$ with $a_{ij}=-1$. 
Let us first reproduce \cite[(2.6)]{LP-gklo} in the conventions of the present paper: 
\begin{align}
&\operatorname{Sym}_{w_1,w_2} 
[\bA_i(w_1), [\bA_i(w_2), \bA_j(z)]_{q^{-1}}]_q = \\ \label{eq: LP Serre} 
&\qquad \qquad =-\bDel(w_1w_2)\operatorname{Sym}_{w_1,w_2} \Bigg[ \frac{1-\Ck w_2^2}{1 - q^{-2} \Ck w_2^2}
\Bigg(
\frac{[2]zw_1^{-1}}{1-q^2w_2w_1^{-1}}
[\bTh_i(w_2),\bA_j(z)]_{q^{-2}}
\\
&\hspace{160pt}
+\frac{1+w_2w_1^{-1}}{1-q^2w_2w_1^{-1}}
[\bA_j(z),\bTh_i(w_2)]_{q^{-2}}
\Bigg) \Bigg]. 
\end{align}
To write the $\imath$Drinfeld--Serre relation \eqref{eq: drel6} term-by-term, we need to introduce some notation. Set
\begin{align} \label{eq: Sij defi}
S_{L,i,j}(r,s;t) &= [A_{i,r},[A_{i,s},A_{j,t}]_{q^{-1}}]_q +[A_{i,s},[A_{i,r},A_{j,t}]_{q^{-1}}]_q. 
\end{align}
Clearly, $S_{L,i,j}(r,s;t) = S_{L,i,j}(s,r;t)$. Moreover, set 
\begin{align} \label{eq: SRij}
S_{R,i,j}(r,s;t) &= -2^{\delta_{r,s}}\Ck^k
\Bigg(
[A_{j,t},\Theta_{i,d}]_{q^{-2}} + \sum_{b\geq1}(q^{2b}+q^{-2b})\Ck^b
[A_{j,t},\Theta_{i,d-2b}]_{q^{-2}}\\
&\quad +\sum_{b\geq0}(q^{2b+1}+q^{-2b-1})\Ck^{b+1}
[\Theta_{i,d-2b-1},A_{j,t-1}]_{q^{-2}}
\Bigg), 
\end{align} 
with $d=|r-s|$ and $k=\min(r,s)$. Since $\bTh_{i}(z)$ is bounded from below, the sums above are finite. 
Then, via \eqref{eq: LP Serre} (and straightforward coefficient extraction), \eqref{eq: drel6} is equivalent to 
\begin{align} \label{eq: Serre modewise}
S_{i,j}(r,s;t): &&  S_{L,i,j}(r,s;t) &= S_{R,i,j}(r,s;t). 
\end{align} 

It is immediate that expressions $S_{R,i,j}(r,s;t)$ have the following properties: 
\begin{align} \label{eq: Sprop1}
S_{R,i,j}(r,s;t) &= S_{ij}(s,r;t), \\ \label{eq: Sprop2}
S_{R,i,j}(r,s;t) &= \Ck S_{ij}(r-1,s-1;t), \\ \label{eq: Sprop3}
S_{R,i,j}(r,r;t) &= -2q\mi \delta_{\mu_i,0} \Ck^r \KK_i A_{j,t}, \\
\label{eq: Sprop4}
S_{R,i,j}(r,s;t) &= 0 \qquad (|r-s| < - \mu_i). 
\end{align} 
We will also need the following Lemma. 

\Lem \label{lem: Serre prekey lem}
For any $r,s,t\in\mathbb Z$, the following hold, assuming only relation \eqref{eq: Yijms1}: 
\begin{align} \label{eq: preLem iSerre1}
S_{R,i,j}(r{+}1,s;t)+S_{R,i,j}(r,s{+}1;t) &= [2]S_{R,i,j}(r,s;t{+}1) \\
&\quad - 2^{\delta_{r,s}}\Ck^{\min(r,s)}[\overline{\Theta}_{i,|r-s|+1},A_{j,t}]_{q^{-2}}, \\ \label{eq: preLem iSerre2}
S_{R,i,j}(r{-}1,s;t)+S_{R,i,j}(r,s{-}1;t) &= [2]S_{R,i,j}(r,s;t{-}1)  \\
&\quad -2^{\delta_{r,s}} \Ck^{\min(r,s)-1}
[A_{j,t},\overline{\Theta}_{i,|r-s|+1}]_{q^{-2}}.
\end{align}
\enlem 

\begin{proof} 
We will only prove the first identity explicitly as the other one is analogous. 
Given \eqref{eq: Sprop1}, we may assume, without loss of generality, that $s\geq r$. Set $d=s-r$ and abbreviate 
\begin{align}
G_d(t)
={}&[A_{j,t},\Theta_{i,d}]_{q^{-2}} +\sum_{b\geq1}(q^{2b}+q^{-2b})\Ck^b
[A_{j,t},\Theta_{i,d-2b}]_{q^{-2}}\\
&\quad +\sum_{b\geq0}(q^{2b+1}+q^{-2b-1})\Ck^{b+1}
[\Theta_{i,d-2b-1},A_{j,t-1}]_{q^{-2}}.
\end{align}
Then \eqref{eq: SRij} can be reformulated as 
\[
S_{R,ij}(r,s;t)
=-2^{\delta_{r,s}} \Ck^rG_d(t).  
\] 
Using the elementary identities 
\begin{align}
(q^{2b+1}+q^{-2b-1}) + (q^{2b-1}+q^{-2b+1}) &= 2^{-\delta_{b,0}} [2](q^{2b} + q^{-2b}),\\
q^{2b}+q^{-2b} + (q^{2b-2}+q^{-2b+2}) 
&=[2](q^{2b-1}+q^{-2b+1}), 
\end{align}
for $b \geq 0$, we get 
\begin{align} \label{eq:serre-coefficient-identity}
G_{d+1}(t) + \Ck G_{d-1}(t) &= [2]G_d(t+1) + [\overline{\Theta}_{i,d+1}, A_{j,t}]_{q^{-2}} \\
&\quad -2^{-\delta_{b,0}}q\mi [2] \sum_{b\geq0} (q^{2b} + q^{-2b}) \Ck^b Y_{i,j}(d-2b,t).
\end{align} 
Since $Y_{i,j}(d-2b,t)=0$ by \eqref{eq: Yijms1}, equality 
\eqref{eq:serre-coefficient-identity} gives
\begin{equation}
\begin{aligned}
G_{d+1}(t) + \Ck G_{d-1}(t) = [2]G_d(t+1) + [\Theta_{i,d+1} - \Ck \Theta_{i,d-1},
A_{j,t}]_{q^{-2}}, 
\end{aligned}
\label{eq:serre-G-recurrence}
\end{equation}
where we set $G_{-1}(t)=0$. Since 
\begin{align}
S_{R,ij}(r+1,s;t)+S_{R,ij}(r,s+1;t) &= - \Ck^r\bigl(\Ck G_{d-1}(t)+G_{d+1}(t)\bigr), 
\end{align} 
equality \eqref{eq:serre-G-recurrence} now immediately implies \eqref{eq: preLem iSerre1}.  
\end{proof}

\subsubsection{Central reductions}
\label{sss: central reduct} 

As we connect shifted quantum affine symmetric pair algebras to shifted quantum affine algebras in \S \ref{sec: sub real1}--\ref{sec: sub real2}, it will be more convenient to work with central reductions $\Ui_\mu$. 
Let $\boldsymbol\gamma = (\gamma_1, \cdots, \gamma_{N}, C) \in (\C^\times)^{N+1}$, where $N = |\indx|$.   
We define $\ui_{\mu,\boldsymbol\gamma}$ to be the quotient of $\Ui_\mu$ by the two-sided ideal generated by 
\[
\KK_i^{\pm1} - \gamma_i^{\pm1} \quad (i \in \indx), \qquad \Ck^{\pm1} - C^{\pm1}. 
\]

\subsection{Levendorskii presentation} 
\label{subsec: gen Lev presen}

Recall that we abbreviated $\rho = (q-q\mi)\mi$. 
Fix an \emph{anti-dominant} coweight $\mu \in \Lambda_-$. 
Define $\Vi_\mu$ to be the algebra generated by
\begin{align} \label{eq: Lev gen set g} 
A_{i,r} \ (\mu_i{-}1 \leq r \leq 0), \quad H_{i,1}, \quad \KK_i^{\pm1}, \quad  \Ck^{\pm1}, 
\end{align} 
where $i \in \indx$, 
with $\KK_i^{\pm1}, \Ck^{\pm1}$ central, 
subject to the following three groups of relations. Group 1 relations  
\begin{align} 
[H_{i,1}, H_{j,1}] &= 0, \label{eq: Urel0} \\
[\Hg_{i,1}, \Ag_{j,r}] &= \textstyle [a_{ij}] (\Ag_{j,r+1}- \Ag_{j,r-1}\CCC), \label{eq: Urel1} \\
[\Ag_{i,r}, \Ag_{j,s}] &= 0 \quad \mbox{if} \quad a_{ij} = 0,  \label{eq: Urel2} \\ \label{eq: Urel3a} 
[\Ag_{j,s+1}, A_{i,r}]_{q^{a_{ij}}} + [\Ag_{i,r+1}, \Ag_{j,s}]_{q^{a_{ij}}} &= 0  \quad \mbox{if} \quad a_{ij}=-1,
\end{align}
\begin{align} \label{eq: Urel3} \ \ \ \ \ 
[\Ag_{i,s+1}, A_{i,r}]_{q^{2}} + [\Ag_{i,r+1}, \Ag_{i,s}]_{q^{2}} 
&= \left\{ \begin{array}{l l}
- 2^{\delta_{\mu_i,0}} \KK_i \Ck^{\min(r,s)} H_{i,1} & \mbox{if } |r{-}s| = {-}\mu_i, \\[2pt] 
- 2^{\delta_{\mu_i,-1}} \rho \KK_i \Ck^{\min(r,s)} & \mbox{if } |r{-}s| = {-}\mu_i{-}1 , \\[2pt] 
0 & \mbox{otherwise}, 
\end{array} \right. 
\end{align}
hold for any indices $r,s$ such that all the generators involved are in \eqref{eq: Lev gen set g}. Group 2 relations are special relations which depend on whether or not $\mu_i,\mu_j = 0$: 
\begin{align} \label{eq: sprel1}
[H_{i,1}, \Theta_{i,2-\mu_i}] &= 0 \qquad  &&(\mbox{if } \mu_i < 0), \\  \label{eq: sprel3}
[A_{i,1}, A_{i,0}]_{q^2} &= - \KK_i H_{i,1} \qquad   &&(\mbox{if } \mu_i = 0), \\ \label{eq: sprel4}
[A_{i,-1}, A_{i,-2}]_{q^2} &= - \KK_i \Ck^{-2} H_{i,1} \qquad &&(\mbox{if } \mu_i = 0), \\ \label{eq: sprel2}
\ \ \ \ \ \  [A_{j,-1}, A_{i,-1}]_{q\mi} + [A_{i,0}, A_{j,-2}]_{q\mi} &= 0 \qquad &&(\mbox{if } \mu_j = 0, \ a_{ij} = -1), \\ \label{eq: sprel5}
\ \ \ \ \ \ [A_{j,-1}, A_{i,-2}]_{q\mi} + [A_{i,-1}, A_{j,-2}]_{q\mi} &= 0 \qquad &&(\mbox{if } \mu_i \mbox{ or } \mu_j = 0, \ a_{ij} = -1), 
\end{align} 
with $\Theta_{i, 2-\mu_i}$ defined below via \eqref{eq: gJ(s)}, and $A_{i,1}, A_{i,-2}$ defined below via \eqref{eq: Air defi g}. Group 3 consists of the $\imath$Serre relations 
\begin{align} \label{eq: Serrel}
S_{L,i,j}(0,0;0) &= -2q^{-1}\delta_{\mu_i,0}\KK_iA_{j,0}, \\ \label{eq: Serrel2}
S_{L,i,j}(0,0;-1) &= -2q^{-1}\delta_{\mu_i,0}\KK_iA_{j,-1}, \\ 
\label{eq: Serrel3}
S_{L,i,j}(-1,-1;0) &= -2q^{-1}\delta_{\mu_i,0}\Ck\mi \KK_iA_{j,0}, \\ \label{eq: Serrel4}
S_{L,i,j}(-1,-1;-1) &= -2q^{-1}\delta_{\mu_i,0}\Ck\mi \KK_iA_{j,-1}. 
\end{align}

Inductively define 
\begin{align}
\label{eq: Air defi g} A_{i,r} &= 
\left\{ \begin{array}{l l}
	[2]\mi [H_{i,1}, A_{i,r-1}] + \Ck A_{i,r-2} & \mbox{if } r > 0, \\[2pt]  
	\Ck\mi A_{i,r+2} - [2]\mi \Ck\mi [H_{i,1}, A_{i,r+1}]   & \mbox{if } r < \mu_i - 1, 
\end{array} \right. \\[2pt]  \label{eq: gJ(s)}
\Theta_{i,s} &= -\Ck^{\lceil \frac{s+1}{2} \rceil} R_{L,i,i}(-\lceil \textstyle\frac{s+1}{2} \rceil, \lfloor \textstyle\frac{s-3}{2} \rfloor) + \Ck \Theta_{i, s - 2}, 
\end{align} 
for $s > 1 - \mu_i$. 

\Thm \label{thm: Lev presentation gen}
There is an algebra isomorphism
\[
\Vi_\mu \to \Ui_\mu, 
\]
sending the Levendorskii generators to the corresponding Drinfeld generators (denoted in the same way). 
\enthm

\Proof
The idea of the proof is the same as that of Theorem \ref{thm: 1Lev}. We need to check that all the Drinfeld relations also hold in $\Vi_\mu$. This is done in section \S \ref{sec: gen Levendorskii proof} below. 
\enproof

\subsection{Proof I}
\label{sec: gen Levendorskii proof}

In this subsection, we will establish the relations \eqref{eq: grel1}--\eqref{eq: grel3c}, i.e., all the relations with the exception of the $\imath$Drinfeld--Serre relations. 
We have already derived all the non-mixed (i.e.\ $i=j$) relations via Theorem \ref{thm: 1Lev} and Proposition \ref{pro: no shift Lev}. By \eqref{eq: Urel0} and \eqref{eq: Air defi g}, we get 
\eq \label{eq: H1 rule gen}
[H_{i,1}, A_{j,r}] = [a_{ij}] ( A_{j,r+1} - \Ck A_{j, r-1} ), 
\eneq
for all $r \in \Z$.

\subsubsection{Relations $\eqref{eq: grel3a}$--$\eqref{eq: grel3b}$} 
\label{ss: Rij rel proof}

In this subsection, we will assume $r \leq s$ when considering the relations $R_{i,j}(r,s)$. The opposite case $r \geq s$ can be obtained through a simultaneous swap $r \leftrightarrow s$ and $i \leftrightarrow j$. 

Let $a_{ij} = 0$ so that $R_{L,i,j}(r,s) =  [\Ag_{j, s}, \Ag_{i, r}]$. 
Then, by \eqref{eq: H1 rule gen}, 
\begin{align}
[H_{i,1}, R_{L,i,j}(r,s)] &= [2](R_{L,i,j}(r+1,s) - \Ck R_{L,i,j}(r-1,s)), \label{eq: H1 a0 rule1} \\
[H_{j,1}, R_{L,i,j}(r,s)] &= [2](R_{L,i,j}(r,s+1) - \Ck R_{L,i,j}(r,s-1)). \label{eq: H1 a0 rule2}
\end{align} 
These equalities, together with $R_{i,j}(r,s)$, $R_{i,j}(r,s-1)$ and $R_{i,j}(r+1,s)$, jointly imply $R_{L,i,j}(r-1,s) = 0 = R_{L,i,j}(r,s+1)$. 

Let $a_{ij} = -1$. 
Write $R_{L,i,j}(r,s) =  [\Ag_{j, s+1}, \Ag_{i, r}]_{q^{-1}} + [\Ag_{i, r+1}, \Ag_{j, s}]_{q^{-1}}$. 
Then, by \eqref{eq: H1 rule gen}, 
\begin{align}
[H_{i,1}, R_{L,i,j}(r,s)] &= [2](R_{L,i,j}(r+1,s) - \Ck R_{L,i,j}(r-1,s)) \label{eq: H1 a-1 rule1} \\ 
&\quad - (R_{L,i,j}(r,s+1) - \Ck R_{L,i,j}(r,s-1)), \\
[H_{j,1}, R_{L,i,j}(r,s)] &= - (R_{L,i,j}(r+1,s) - \Ck R_{L,i,j}(r-1,s)) \label{eq: H1 a-1 rule2} \\ 
&\quad + [2](R_{L,i,j}(r,s+1) - \Ck R_{L,i,j}(r,s-1)). 
\end{align} 
These equalities, together with the non-degeneracy of the matrix $\begin{bmatrix}
[2] & -1 \\
-1 & [2]
\end{bmatrix}$, show that $R_{i,j}(r,s)$, $R_{i,j}(r,s-1)$ and $R_{i,j}(r+1,s)$ jointly imply $R_{L,i,j}(r-1,s) = 0 = R_{L,i,j}(r,s+1)$. 

In both ($a_{ij} = 0$ and $a_{ij}=-1$) cases, the arguments above justify an induction where the knowledge of all relations $R_{i,j}(r,s)$ for $|s-r| \leq k$ (for $k \geq 1$) allows us to deduce all relations $R_{i,j}(r,s)$ for $|s-r| = k+1$. Note that this induction differs from the case $i=j$ - there we relied on the knowledge of one relation $R_{i,i}(r,s)$ with $|s-r| = k+1$ given by the definition of $\Theta_{k+2-\mu}$. The non-degeneracy argument used above can be seen as a replacement for this relation. 

It remains to establish the base case, i.e., the cases $R_{i,j}(r,s)$ for $|s-r| \leq 2$. This will be achieved via a nested induction. The argument is similar to the one in \S \ref{ss: early B}, but we do not need an analogue of Lemma \ref{lem: H2 comm}. First let $a_{ij} = 0$. Irrespectively of the shifts $\mu_i, \mu_j$, we know from \S \ref{subsec: gen Lev presen} that at least the relations in Figure \ref{fig: G-Bind1} hold. Using \eqref{eq: H1 a0 rule1}--\eqref{eq: H1 a0 rule2}, with $(r,s) = (-2,-1)$, we can establish all the relations in Figure \ref{fig: G-Bind2}. 

\begin{figure}[h]
\begin{tikzcd}[row sep=-10pt]
	{R_{i,j}(0,0)} && \\
	& {R_{i,j}(-1,0)} \\
	{R_{i,j}(-1,-1)} && 
\end{tikzcd}
\captionof{figure}{Instances of relations $R_{i,j}(r,s)$ which always hold, independently of $\mu_i, \mu_j$, if $a_{ij}=0$. } 
\label{fig: G-Bind1}
\end{figure} 

\begin{figure}[h]
\begin{tikzcd}[row sep=-7pt]
	{R_{i,j}(0,0)} && \textcolor{rgb,255:red,92;green,214;blue,92}{{R_{i,j}(-1,1)}} \\
	& \textcolor{rgb,255:red,69;green,89;blue,237}{{R_{i,j}(-1,0)}} \\
	{R_{i,j}(-1,-1)} && \textcolor{rgb,255:red,92;green,214;blue,92}{{R_{i,j}(-2,0)}} 
	\arrow[color={rgb,255:red,92;green,214;blue,92}, from=2-2, to=1-3]
	\arrow[color={rgb,255:red,69;green,89;blue,237}, no head, from=2-2, to=1-1]
	\arrow[color={rgb,255:red,69;green,89;blue,237}, no head, from=3-1, to=2-2]
	\arrow[color={rgb,255:red,92;green,214;blue,92}, from=2-2, to=3-3]
\end{tikzcd}
\captionof{figure}{Propagation using \eqref{eq: H1 a0 rule1}--\eqref{eq: H1 a0 rule2}. Black and blue relations are assumed to be known. The commutator is applied to the blue relation. The blue lines connect to the other relevant known relations. The green relations, also indicated by the green arrow, are deduced.} 
\label{fig: G-Bind2}
\end{figure}

This establishes the base of the nested induction. All the desired relations can now be obtained using upward and downward induction, with upward inductive hypothesis ($s \geq 0$)
\begin{align}
\operatorname{IH}^\uparrow(s):&\qquad& R_{i,j}(s,s), \quad R_{i,j}(s-1,s), \quad R_{i,j}(s-1,s+1), 
\end{align}
and downward inductive hypothesis ($s \leq -1$) 
\begin{align}
\operatorname{IH}^\downarrow(s):&\qquad& R_{i,j}(s,s), \quad R_{i,j}(s,s+1), \quad R_{i,j}(s-1,s+1). 
\end{align}
In the upward case, the inductive hypothesis $\operatorname{IH}^\uparrow(s)$, together with \eqref{eq: H1 a0 rule1}, yields $R_{i,j}(s,s+1)$. 
Then \eqref{eq: H1 a0 rule1}--\eqref{eq: H1 a0 rule2} applied to $R_{i,j}(s,s+1)$ yield $R_{i,j}(s+1, s+1)$ and $R_{i,j}(s,s+2)$, which establishes $\operatorname{IH}^\uparrow(s+1)$. These steps are illustrated in Figures \ref{fig: G-Bind3}--\ref{fig: G-Bind4} below.  The downward induction is analogous. 

\begin{figure}[h]
\begin{tikzcd}[row sep=-7pt]
	& \textcolor{rgb,255:red,92;green,214;blue,92}{{R_{i,j}(s,s+1)}} \\
	\textcolor{rgb,255:red,69;green,89;blue,237}{{R_{i,j}(s,s)}} && {R_{i,j}(s-1,s+1)} \\
	& {R_{i,j}(s-1,s)} 
	\arrow[color={rgb,255:red,92;green,214;blue,92}, from=2-1, to=1-2]
	\arrow[color={rgb,255:red,69;green,89;blue,237}, no head, from=3-2, to=2-1]
\end{tikzcd}
\captionof{figure}{Propagation using \eqref{eq: H1 a0 rule1}.}
\label{fig: G-Bind3}
\end{figure}

\begin{figure}[h]
\begin{tikzcd}[row sep=-7pt]
	\textcolor{rgb,255:red,92;green,214;blue,92}{{R_{i,j}(s+1,s+1)}} && \textcolor{rgb,255:red,92;green,214;blue,92}{{R_{i,j}(s,s+2)}} \\
	& \textcolor{rgb,255:red,69;green,89;blue,237}{{R_{i,j}(s,s+1)}} \\
	{R_{i,j}(s,s)} && {R_{i,j}(s-1,s+1)} \\
	& {R_{i,j}(s-1,s)} 
	\arrow[color={rgb,255:red,92;green,214;blue,92}, from=2-2, to=1-3]
	\arrow[color={rgb,255:red,92;green,214;blue,92}, from=2-2, to=1-1]
	\arrow[color={rgb,255:red,69;green,89;blue,237}, no head, from=3-1, to=2-2]
	\arrow[color={rgb,255:red,69;green,89;blue,237}, no head, from=3-3, to=2-2]
\end{tikzcd}
\captionof{figure}{Propagation using \eqref{eq: H1 a0 rule1}--\eqref{eq: H1 a0 rule2}.}
\label{fig: G-Bind4}
\end{figure}

Now let $a_{ij} = -1$. 
Irrespectively of the shifts $\mu_i, \mu_j$, we know from \S \ref{subsec: gen Lev presen} (in particular, \eqref{eq: sprel2}--\eqref{eq: sprel5}) that at least the relations in Figure \ref{fig: G-Bind5} hold. 
Moreover, the identities \eqref{eq: H1 a-1 rule1}--\eqref{eq: H1 a-1 rule2}, together with the non-degeneracy argument, show that $R_{i,j}(r,s)$ implies 
\eq \label{eq: H1 a-1 rule3}
R_{L,i,j}(r+1,s) = \Ck R_{L,i,j}(r-1,s), \qquad R_{L,i,j}(r,s+1) = \Ck R_{L,i,j}(r,s-1). 
\eneq
Therefore, the diagram in Figure \ref{fig: G-Bind5} can be completed to the one in Figure \ref{fig: G-Bind6}. 
The rest of the argument proceeds in exactly the same fashion as in the $a_{ij}=0$ case. 

\begin{figure}[h]
\begin{tikzcd}[row sep=-10pt]
	{R_{i,j}(-1,-1)} && \\
	& {R_{i,j}(-2,-1)} \\
	{R_{i,j}(-2,-2)} && 
\end{tikzcd}
\captionof{figure}{Instances of relations $R_{i,j}(r,s)$ which always hold, independently of $\mu_i, \mu_j$, if $a_{ij}=-1$. } 
\label{fig: G-Bind5}
\end{figure}

\begin{figure}[h]
\begin{tikzcd}[row sep=-7pt]
	{R_{i,j}(-1,-1)} && \textcolor{rgb,255:red,92;green,214;blue,92}{{R_{i,j}(-2,0)}} \\
	& \textcolor{rgb,255:red,69;green,89;blue,237}{{R_{i,j}(-2,-1)}} \\
	{R_{i,j}(-2,-2)} && \textcolor{rgb,255:red,92;green,214;blue,92}{{R_{i,j}(-3,-1)}} 
	\arrow[color={rgb,255:red,92;green,214;blue,92}, from=2-2, to=1-3]
	\arrow[color={rgb,255:red,69;green,89;blue,237}, no head, from=2-2, to=1-1]
	\arrow[color={rgb,255:red,69;green,89;blue,237}, no head, from=3-1, to=2-2]
	\arrow[color={rgb,255:red,92;green,214;blue,92}, from=2-2, to=3-3]
\end{tikzcd}
\caption{Propagation using \eqref{eq: H1 a-1 rule3}.}
\label{fig: G-Bind6}
\end{figure}

\subsubsection{Relation $\eqref{eq: grel2}$}

We will show the equivalent relation $\overline{Y}_{i,j}(m,s) = 0$ (see \eqref{eq: Yijms}), using \eqref{eq: grel3a}--\eqref{eq: grel3c}, which have already been established. Note that, by 
\eqref{eq: grel3c}, each $\btheta_{i,a}$ in $\overline{Y}_{i,j}(m,s)$ can be expressed in terms of some $A_{i,b}$. If $a_{ij} = 0$ then \eqref{eq: grel3a} implies that each of these $A_{i,b}$ commutes with the $A_{j,c}$ occurring in $\overline{Y}_{i,j}(m,s)$. This immediately entails that $\overline{Y}_{i,j}(m,s) = 0$. So let $a_{ij} = -1$. We have the following analogue of Lemma \ref{lem: key id}. 

\Lem \label{lem: key lem gen}
For any $r,t,s \in \Z$, the following identity holds:
\begin{align}
&[R_{L,i,i}(r,t),A_{j,s}]+[R_{L,i,i}(r+1,t-1),A_{j,s}]  \\
&+q[A_{j,s+1},R_{L,i,i}(r,t-1)]_{q^{-2}}-q[R_{L,i,i}(r+1,t),A_{j,s-1}]_{q^{-2}} =\\
=& \
[A_{i,t},R_{L,i,j}(r,s)]_{q^3}+q[R_{L,i,j}(t,s),A_{i,r}]_q \label{eq: key lemma gen id} \\
&+[A_{i,r+1},R_{L,i,j}(t-1,s)]_{q^3}+q[R_{L,i,j}(r+1,s),A_{i,t-1}]_q\\
&-q[A_{i,t+1},R_{L,i,j}(r,s-1)]_q+q^3[A_{i,r+1},R_{L,i,j}(t,s-1)]_{q^{-3}}\\
&-q[A_{i,r+2},R_{L,i,j}(t-1,s-1)]_q+q^3[A_{i,t},R_{L,i,j}(r+1,s-1)]_{q^{-3}} \\ 
=& \ 0. 
\end{align}
\enlem 

\Proof
The first equality is a general algebraic identity which also holds in the free non-commutative algebra. The second equality then follows immediately from \eqref{eq: grel3b}. The rest of the proof is devoted to proving the first equality. 
For the sake of legibility readability, we abbreviate 
\[
X_a=A_{i,a},\qquad Z_b=A_{j,b}.
\]
We will use the following elementary identities, which hold for arbitrary elements $a,b,c$ of a free non-commutative  algebra: 
\begin{align}
q[c,[a,b]_{q^2}]_{q^{-2}} &= [a,[c,b]_{q^{-1}}]_{q^3} + q[[c,a]_{q^{-1}},b]_q, \label{eq: tag1} \\ 
-q[[a,b]_{q^2},c]_{q^{-2}} &= -q[a,[b,c]_{q^{-1}}]_q + q^3[b,[a,c]_{q^{-1}}]_{q^{-3}}, \label{eq: tag2} \\ 
[[a,b]_{q^2},c] &= q[[a,c]_{q^{-1}},b]_q - q[a,[c,b]_{q^{-1}}]_q, \label{eq: tag3} \\ 
[[a,b]_{q^2}+[b,a]_{q^2},c] &= [a,[b,c]_{q^{-1}}]_{q^3} + [b,[a,c]_{q^{-1}}]_{q^3} \label{eq: tag4} \\
&\quad + q^3[a,[c,b]_{q^{-1}}]_{q^{-3}} + q^3[b,[c,a]_{q^{-1}}]_{q^{-3}}.  
\end{align} 
All the four identities above follow by direct expansion and comparison of monomials on both sides. Let $\RHS$ denote the right-hand side of \eqref{eq: key lemma gen id}. 

Below we will separately treat terms containing $Z_{s+1}$, $Z_s$, and $Z_{s-1}$. 
Applying \eqref{eq: tag1}, we get 
\begin{align}
q[Z_{s+1},R_{L,i,i}(r,t-1)]_{q^{-2}} &= [X_t,[Z_{s+1},X_r]_{q^{-1}}]_{q^3}
+ q[[Z_{s+1},X_t]_{q^{-1}},X_r]_q \\ \label{eq: otag1}
&\quad + [X_{r+1},[Z_{s+1},X_{t-1}]_{q^{-1}}]_{q^3} + q[[Z_{s+1},X_{r+1}]_{q^{-1}},X_{t-1}]_q.
\end{align} 
These will supply the monomials containing $Z_{s+1}$ in the first four commutators in $\RHS$. 
Next, applying \eqref{eq: tag2} yields 
\begin{align}
-q[R_{L,i,i}(r+1,t),Z_{s-1}]_{q^{-2}} &= -q[X_{t+1},[X_{r+1},Z_{s-1}]_{q^{-1}}]_q + q^3[X_{r+1},[X_{t+1},Z_{s-1}]_{q^{-1}}]_{q^{-3}} \\ \label{eq: otag2}
&\quad - q[X_{r+2},[X_t,Z_{s-1}]_{q^{-1}}]_q + q^3[X_t,[X_{r+2},Z_{s-1}]_{q^{-1}}]_{q^{-3}}.
\end{align} 
These will supply the monomials containing $Z_{s-1}$ in the last four commutators in $\RHS$. 
Next, applying \eqref{eq: tag3}--\eqref{eq: tag4} yields  
\begin{align}
[R_{L,i,i}(r,t),Z_s] &+ [R_{L,i,i}(r+1,t-1),Z_s] = \\  \label{eq: otag3}
&= q[[X_{t+1},Z_s]_{q^{-1}},X_r]_q - q[X_{t+1},[Z_s,X_r]_{q^{-1}}]_q \\ 
&\quad + [X_t,[X_{r+1},Z_s]_{q^{-1}}]_{q^3} + [X_{r+1},[X_t,Z_s]_{q^{-1}}]_{q^3} \\
&\quad + q^3[X_t,[Z_s,X_{r+1}]_{q^{-1}}]_{q^{-3}} + q^3[X_{r+1},[Z_s,X_t]_{q^{-1}}]_{q^{-3}} \\
&\quad + q[[X_{r+2},Z_s]_{q^{-1}},X_{t-1}]_q - q[X_{r+2},[Z_s,X_{t-1}]_{q^{-1}}]_q.
\end{align} 
Adding \eqref{eq: otag1}, \eqref{eq: otag2} and \eqref{eq: otag3} sidewise yields \eqref{eq: key lemma gen id}.  
\enproof

Choose any $n,x$ with $n-x=m$. By \eqref{eq: grel3c}, we have 
\begin{align}
R_{L,i,i}(x,n) &=-\Ck^x\overline\Theta_{i,m+1},\\
R_{L,i,i}(x,n-1) &= -\Ck^x\overline\Theta_{i,m},\\
R_{L,i,i}(x+1,n-1) &= -\Ck^{x+1}\overline\Theta_{i,m-1},\\
R_{L,i,i}(x+1,n) &= -\Ck^{x+1}\overline\Theta_{i,m}.
\end{align}
Therefore, 
\begin{align}
-\Ck^x\overline Y_{i,j}(m,s) ={}& [R_{L,i,i}(x,n),A_{j,s}] + [R_{L,i,i}(x+1,n-1),A_{j,s}] \\
&+q[A_{j,s+1},R_{L,i,i}(x,n-1)]_{q^{-2}} -q[R_{L,i,i}(x+1,n),A_{j,s-1}]_{q^{-2}}.
\end{align} 
The fact that the expression on the RHS above vanishes now follows directly from Lemma~\ref{lem: key lem gen}. 

\subsubsection{Relation \eqref{eq: grel1}} 

It suffices to prove the equivalent relation $[\overline{H}_{i,m}, \btheta_{j,n}] = 0$ for $i \neq j$ (cf.\ \eqref{eq: Hm overline}).  
We have 
\begin{align}
- [\overline{H}_{i,m}, \btheta_{j,n}] &= [\overline{H}_{i,m}, R_{L,j,j}(0, n-1)] \\
&= 
R_{L,j,j}(m, n{-}1) + R_{L,j,j}(0, m{+}n{-}1) \\
&\quad - \Ck^m R_{L,j,j}({-}m, n{-}1) - \Ck^m R_{L,j,j}(0, n{-}m{-}1) = 0. 
\end{align}
The first and third equalities follow from \eqref{eq: grel3c}, while the second one follows from \eqref{eq: grel2}, both of which have already been established.

\subsection{Proof II: $\imath$Drinfeld--Serre relations} 
\label{sss: iSerre proof gen}

We will now derive \eqref{eq: Serre modewise}. We can assume all the relations \eqref{eq: grel1}--\eqref{eq: grel3c}, established in \S \ref{sec: gen Levendorskii proof}. For the rest of this subsection, we fix $i,j \in \indx$ with $a_{ij} = -1$. Recall that the definition of expression $S_{L,i,j}(r,s;t)$, which appears repeatedly below, was given in \eqref{eq: Sij defi}. 

\subsubsection{Key lemma} 

The following Lemma can be seen as somewhat analogous to Lemmata \ref{lem: key id} and \ref{lem: key lem gen}. 

\Lem \label{lem: Serre key lem}
Let $i,j\in \indx$ with $a_{ij}=-1$. Assume that $R_{i,j}(r,t)$ and $R_{i,i}(r,t)$ (i.e., \eqref{eq: grel3b}--\eqref{eq: grel3c}) hold for all $r,t\in\mathbb Z$. Then, for any $r,s,t\in\mathbb Z$: 
\begin{align} \label{eq: Lem iSerre1}
S_{L,i,j}(r{+}1,s;t)+S_{L,i,j}(r,s{+}1;t) &= [2]S_{L,i,j}(r,s;t{+}1) \\ 
&\quad - 2^{\delta_{r,s}}\Ck^{\min(r,s)}[\overline{\Theta}_{i,|r-s|+1},A_{j,t}]_{q^{-2}}, \\ \label{eq: Lem iSerre2}
S_{L,i,j}(r{-}1,s;t)+S_{L,i,j}(r,s{-}1;t) &= [2]S_{L,i,j}(r,s;t{-}1) \\ 
&\quad - 2^{\delta_{r,s}} \Ck^{\min(r,s)-1}
[A_{j,t},\overline{\Theta}_{i,|r-s|+1}]_{q^{-2}}.
\end{align}
\enlem

\begin{proof}
We only prove \eqref{eq: Lem iSerre1} explicitly since \eqref{eq: Lem iSerre2} is analogous. 
Let us first record two identities which hold for arbitrary elements $x,y,z,u$ in any non-commutative algebra:
\begin{align} \label{eq: Serre aux1}
[x,[y,z]_{q^{-1}}]_q + [y,[x,z]_{q^{-1}}]_q &= [[x,y]_{q^2},z]_{q^{-2}}
 + q[2][y,[x,z]_{q^{-1}}]_{q^{-1}}, \\ \label{eq: Serre aux2}
[y,[u,x]_{q^{-1}}]_{q^{-1}} + [x,[u,y]_{q^{-1}}]_{q^{-1}} &= -q^{-1}
 \bigl(
 [x,[y,u]_{q^{-1}}]_q
 +[y,[x,u]_{q^{-1}}]_q
 \bigr).
\end{align} 
Both identities follow by direct expansion. 
Applying \eqref{eq: Serre aux1} twice, we obtain
\begin{align}
S_{L,i,j}(r{+}1,s;t) &+ S_{L,i,j}(r,s{+}1;t) =
\bigl[ [A_{i,r+1},A_{i,s}]_{q^2} + [A_{i,s+1},A_{i,r}]_{q^2}, A_{j,t} \bigr]_{q^{-2}}\\ \label{eq: Serre aux3}
&
+q[2]\bigl(
 [A_{i,s},
   [A_{i,r+1},A_{j,t}]_{q^{-1}}
 ]_{q^{-1}} +[A_{i,r},
   [A_{i,s+1},A_{j,t}]_{q^{-1}}
 ]_{q^{-1}}
 \bigr).
\end{align}
We will next use the following instances of \eqref{eq: grel3b}: 
\[
[A_{i,r+1},A_{j,t}]_{q^{-1}}
=-[A_{j,t+1},A_{i,r}]_{q^{-1}}, \qquad 
[A_{i,s+1},A_{j,t}]_{q^{-1}}
=-[A_{j,t+1},A_{i,s}]_{q^{-1}}.
\]
Substituting these into \eqref{eq: Serre aux3}, we get 
\begin{align}
S_{L,i,j}(r+1,s;t) &+ S_{L,i,j}(r,s+1;t) = [R_{L,i,i}(r,s),A_{j,t}]_{q^{-2}} \\ \label{eq: Serre aux4}
 &- q[2]\bigl(
 [A_{i,s},
   [A_{j,t+1},A_{i,r}]_{q^{-1}}
 ]_{q^{-1}} 
+[A_{i,r},
   [A_{j,t+1},A_{i,s}]_{q^{-1}}
 ]_{q^{-1}}
 \bigr).
\end{align}
Applying \eqref{eq: Serre aux2} gives 
\begin{align} \label{eq: Serre aux5}
[A_{i,s}, [A_{j,t+1},A_{i,r}]_{q^{-1}}
 ]_{q^{-1}} 
+[A_{i,r},
  [A_{j,t+1},A_{i,s}]_{q^{-1}}
 ]_{q^{-1}} = -q^{-1}S_{L,i,j}(r,s;t+1).
\end{align}
Substituting \eqref{eq: Serre aux5} into \eqref{eq: Serre aux4} yields
\begin{align}
S_{L,i,j}(r+1,s;t)+S_{L,i,j}(r,s+1;t) = [R_{L,i,i}(r,s),A_{j,t}]_{q^{-2}} + [2]S_{L,i,j}(r,s;t+1). 
\end{align} 
Finally, applying \eqref{eq: grel3c} yields \eqref{eq: Lem iSerre1}. 
\end{proof}

\Cor \label{cor: 3rel}
Suppose that we know that two\footnote{Note that if $r=s$ then the first two relations in each set coincide thanks to the $r,s$-symmetry. In that case, it suffices to know one of the two unique relations.} 
 of the relations in $$\{ S_{i,j}(r+1,s;t), \ S_{i,j}(r,s+1;t), \ S_{i,j}(r,s;t+1) \}$$ or $$\{ S_{i,j}(r-1,s;t), \ S_{i,j}(r,s-1;t), \ S_{i,j}(r,s;t-1) \}$$ hold in $\Vi_\mu$. Then the corresponding third relation holds as well. 
\encor

\Proof
Since we have already established that \eqref{eq: grel1}--\eqref{eq: grel3c} hold  in $\Vi_\mu$, the Corollary follows directly from Lemmata \ref{lem: Serre prekey lem} and \ref{lem: Serre key lem}.  
\enproof

\subsubsection{Non-degeneracy argument} 

Let $m \geq 1$. 
It follows straightforwardly from the definitions and \eqref{eq: grel2} that 
\begin{align} \label{eq: Sij nondeg1}
[H_{i,m}, S_{L,i,j}(r,s;t)]
&= \textstyle\frac{[2m]}{m}  \operatorname{Sym}_{r,s}
\big( S_{L,i,j}(r+m,s;t) -\Ck^m S_{L,i,j}(r-m,s;t) \big)\\
&\quad - \textstyle\frac{[m]}{m} \big( S_{L,i,j}(r,s;t+m)-\Ck^m S_{L,i,j}(r,s;t-m) \big), \\ \label{eq: Sij nondeg2}
[H_{j,m},S_{L, i, j}(r,s;t)]
&= - \textstyle\frac{[m]}{m} \operatorname{Sym}_{r,s} \big( S_{L,i,j}(r+m,s;t) - \Ck^m S_{L,i,j}(r-m,s;t) \big) \\
&\quad + \textstyle\frac{[2m]}{m}  \big( S_{L,i,j}(r,s;t+m) -\Ck^m S_{L,i,j}(r,s;t-m) \big).
\end{align} 

\Cor \label{cor rel2}
Let $r,t \in \Z$ and $m \geq 1$. 
\be
\item Suppose that we know that $S_{i,j}(r,r;t)$ and one of the relations $$\{S_{i,j}(r+m,r;t), \ S_{i,j}(r-m,r;t) \}$$ holds in $\Vi_\mu$. Then the other relation holds as well. \label{cor rel2a}
\item Suppose that we know that $S_{i,j}(r,r;t)$ and one of the relations $$\{S_{i,j}(r,r;t+m), \ S_{i,j}(r,r;t-m) \}$$ holds in $\Vi_\mu$. Then the other relation holds as well. \label{cor rel2b}
\ee
\encor 

\Proof 
By $S_{i,j}(r,r;t)$ and \eqref{eq: Sprop3}, we have 
\[
S_{L,i,j}(r,r;t) = S_{R,i,j}(r,r;t) = -2q\mi \delta_{\mu_i,0} \Ck^r \KK_i A_{j,t} := D_r(t). 
\] 
Abbreviate 
\begin{align}
X &= S_{L,i,j}(r+m,r;t) - \Ck^m S_{L,i,j}(r-m,r;t), \\ 
Y &= S_{L,i,j}(r,r;t+m) - \Ck^m S_{L,i,j}(r,r;t-m), \\ 
Z &= D_r(t+m) - \Ck^m D_r(t-m). 
\end{align}
Using \eqref{eq: grel2}, we get 
\[
[H_{i,m}, D_r(t)] = - \textstyle\frac{[m]}{m} Y, \qquad [H_{j,m}, D_r(t)] = \textstyle\frac{[2m]}{m} Z.  
\] 
Therefore, \eqref{eq: Sij nondeg1}--\eqref{eq: Sij nondeg2} become 
\[
2[2m]X - [m]Y = -[m]Z, \qquad -2[m]X + [2m]Y = [2m]Z. 
\]
The non-degeneracy of the matrix $\begin{bmatrix}
2[2m] & -[m] \\
-2[m] & [2m]
\end{bmatrix}$ 
allows us to conclude that $X = 0$ and $Y = Z$. 

We can now prove the first part of the Corollary. 
Without loss of generality, suppose that we know $S_{i,j}(r-m,r;t)$ holds in $\Vi_\mu$. Then 
\begin{align}
S_{L,i,j}(r+m,r;t) = \Ck^m S_{L,i,j}(r-m,r;t) = \Ck^m S_{R,i,j}(r-m,r;t) = S_{R,i,j}(r+m,r;t), 
\end{align}
where the first equality follows from $X=0$ and the final equality follows from \eqref{eq: Sprop2}, establishing $S_{i,j}(r+m,r;t)$. For the second part, suppose we know $S_{i,j}(r,r;t-m)$ holds in $\Vi_\mu$. This, combined with $Y = Z$, implies that $$S_{L,i,j}(r,s;t+m) = D_r(t+m) = S_{R,i,j}(r,s;t+m),$$ where the final equality follows from \eqref{eq: Sprop3}. 
\enproof

\subsubsection{Base case} 
\label{sss: Serre base}

First observe that \eqref{eq: Serrel}--\eqref{eq: Serrel2} and Corollary \ref{cor rel2}(\ref{cor rel2b}) (with $m=1$) entail $S_{i,j}(0,0;1)$ and $S_{i,j}(0,0;-2)$. Repeating this process yields $S_{i,j}(0,0;t)$ for all $t \in \Z$. The same argument, using   \eqref{eq: Serrel3}--\eqref{eq: Serrel4}, shows that $S_{i,j}(-1,-1;t)$ holds for all $t \in \Z$. Corollary \ref{cor: 3rel} now allows us to deduce that the relations in Figure \ref{fig: Serre} hold. 
This forms the basis of a nested induction. 

\begin{figure}[h]
\begin{tikzcd}[row sep=-7pt]
	& {S_{i,j}(0,1;t)} & \\
	{S_{i,j}(0,0;t)} && {S_{i,j}(-1,1;t)} \\
	& {S_{i,j}(-1,0;t)} \\
	{S_{i,j}(-1,-1;t)} && {S_{i,j}(-2,0;t)} \\
	& {S_{i,j}(-2,-1;t)}
\end{tikzcd}
\caption{The base of the nested induction: $\operatorname{IH}^\uparrow(0)$ and $\operatorname{IH}^\downarrow(0)$.}
\label{fig: Serre}
\end{figure}

Given $r \geq 0$, consider the following upward inductive hypothesis: 
\begin{align}
\operatorname{IH}^\uparrow(r):&\qquad& S_{i,j}(r,r;t), \quad S_{i,j}(r,r+1;t), \quad S_{i,j}(r-1,r+1;t), \quad S_{i,j}(r-2,r;t) 
\end{align}
hold for all $t \in \Z$. The downward inductive hypothesis reads: 
\begin{align}
\operatorname{IH}^\downarrow(r):&\qquad& S_{i,j}(r{-}1,r{-}1;t), \quad S_{i,j}(r{-}2,r{-}1;t), \quad S_{i,j}(r{-}2,r;t), \quad S_{i,j}(r{-}1,r{+}1;t) 
\end{align} 
hold for all $t \in \Z$. 
In particular, in the previous paragraph, we established $\operatorname{IH}^\uparrow(0)$ and $\operatorname{IH}^\downarrow(0)$.  
We will now explain the upward induction step (the downward induction step is completely analogous). 
Assuming $\operatorname{IH}^\uparrow(r)$, we proceed in the following steps: 
\be
\item apply Corollary \ref{cor rel2}(\ref{cor rel2a}) (with $m=2$) to $S_{i,j}(r,r;t)$ to deduce $S_{i,j}(r,r+2;t)$, 
\item apply Corollary \ref{cor: 3rel} to $S_{i,j}(r,r+1;t)$ to deduce $S_{i,j}(r+1,r+1;t)$, 
\item apply Corollary \ref{cor: 3rel} to $S_{i,j}(r+1,r+1;t)$ to deduce $S_{i,j}(r+1,r+2;t)$. 
\ee 
This procedure yields $\operatorname{IH}^\uparrow(r+1)$. Overall, the upward and downward inductions entail that $S_{i,j}(r,s;t)$ holds for any $r,s,t \in \Z$ with $|r-s| \leq 2$.

\subsubsection{Induction} 

Given the $r,s$-symmetry of the expressions $S_{L,i,j}(r,s;t)$ and $S_{R,i,j}(s,r;t)$, we may assume, without loss of generality, that $s \geq r$. 
Take $m \geq 2$ and suppose that $S_{i,j}(r,s;t)$ holds in $\Vi_\mu$ for all $r,s,t \in \Z$ with $s-r \leq m$. Then Corollary \ref{cor: 3rel} immediately implies that $S_{i,j}(r,s+1;t)$ also holds for all $r, s, t \in \Z$ (with $r \leq s$), allowing the induction to proceed. We conclude that all the $\imath$Drinfeld--Serre relations $S_{i,j}(r,s;t)$ hold in $\Vi_\mu$, completing the derivation of \eqref{eq: Serre modewise} and the proof of Theorem \ref{thm: Lev presentation gen}.

\subsection{Redundancies in higher rank} 
\label{ss: redundancies in high rank}

In this subsection, let $\boldsymbol\Gamma$ be a connected simply-laced graph with at least two vertices. Thanks to the availability of $\imath$Drinfeld--Serre relations, some of the Group 2 Levendorskii relations from \S \ref{subsec: gen Lev presen} turn out to be redundant. Let $\widetilde{\mathfrak{U}}^\imath_\mu$ be the algebra defined by the same generators and relations as in \S \ref{subsec: gen Lev presen}, with the exception of relations \eqref{eq: sprel3}--\eqref{eq: sprel4} for every $i \in \indx$ such that $\mu_i = 0$ and there exists $j \in \indx$ with $a_{ij} = -1$ and $\mu_j \leq -1$.

\Thm \label{thm: redundancies in high rank}
There is an isomorphism $\widetilde{\mathfrak{U}}^\imath_\mu \cong \Vi_\mu$. 
\enthm 

\Proof 
Take any $i,j \in \indx$ with $a_{ij} = -1$, $\mu_i = 0$ and $\mu_j \leq -1$. 
Carefully inspecting the argument from \S \ref{sss: Serre base}, we observe that the derivation of $S_{i,j}(0,0;t)$ and $S_{i,j}(-1,-1;t)$, for all $t \in \Z$, only requires \eqref{eq: Serrel}--\eqref{eq: Serrel4} and \eqref{eq: H1 rule gen}, which itself follows from the definitions and \eqref{eq: Urel0}. 
Moreover, we claim that $S_{i,j}(-1,0;t)$ can be derived without using \eqref{eq: sprel3}--\eqref{eq: sprel4}.   
Indeed, tracing through the argument of Lemma \ref{lem: Serre key lem}, and using only $R_{i,j}(r,t)$ and $R_{i,j}(s,t)$, we get 
\begin{align} \label{eq: Serre rec sp1}
 \ \ \ S_{L,i,j}(r+1,s;t)+S_{L,i,j}(r,s+1;t) = [R_{L,i,i}(r,s), A_{j,t}]_{q^{-2}} + [2]S_{L,i,j}(r,s;t+1). 
\end{align} 
Specializing to $r=-1=s$ and using $R_{i,i}(-1,-1)$, we get 
\[
2S_{L,i,j}(-1,0;t)
=[2]S_{L,i,j}(-1,-1;t+1)
-2\Ck^{-1}[\KK_iH_{i,1}, A_{j,t}]_{q^{-2}}.
\]
Comparing with the formula in Lemma \ref{lem: Serre prekey lem} (which only requires \eqref{eq: Sij defi} and \eqref{eq: H1 rule gen}) and using $S_{i,j}(-1,-1,t+1)$, we deduce $S_{i,j}(-1,0;t)$. Then, tracing through the proof of Corollary \ref{cor rel2}, we observe that only \eqref{eq: H1 rule gen} is needed to derive $S_{i,j}(0,1;t)$. 

Next, write $\mathcal{X} = [A_{i,1},A_{i,0}]_{q^2} + \KK_iH_{i,1}$. We claim that 
\eq \label{eq: spec Serre aux1}
[\mathcal{X}, A_{j,t}]_{q^{-2}} = 0. 
\eneq
Specializing \eqref{eq: Serre rec sp1} to $r = 0 =s$ instead, we obtain 
\begin{align}
2S_{L,i,j}(0,1;t) = [2]S_{L,i,j}(0,0;t+1) +2 [[A_{i,1}, A_{i,0}]_{q^2}, A_{j,t}]_{q^{-2}} 
\end{align}
On the other hand, \eqref{eq: Sij defi} and \eqref{eq: H1 rule gen} imply  
\begin{align}
2S_{R,i,j}(0,1;t) = [2]S_{R,i,j}(0,0;t+1) -2\KK_i [H_{i,1}, A_{j,t}]_{q^{-2}}. 
\end{align}
Using the already established relations $S_{i,j}(0,1;t)$ and $S_{i,j}(0,0,t+1)$, we conclude that $[\mathcal{X}, A_{j,t}]_{q^{-2}} = 0$, as claimed. 

Now suppose that $\mu_j \leq -1$ and recall that $R_{j,j}(\mu_j, -1)$ states that 
\[
[A_{j,0}, A_{j,\mu_j}]_{q^2} + [A_{j,\mu_j+1}, A_{j,-1}]_{q^2} = 
-2^{\delta_{\mu_j,-1}} \Ck^{\mu_j}(q-q\mi)\mi \KK_j. 
\]
It follows from \eqref{eq: spec Serre aux1} that $\mathcal{X} R_{L,j,j} = q^{-4} R_{L,j,j} \mathcal{X}$. Then, using $R_{j,j}(\mu_j, -1)$ and the fact that $R_{R,j,j}(\mu_j, -1)$ is central, we conclude that $(1-q^{-4}) R_{R,j,j}(\mu_j, -1) \mathcal{X} = 0$. The invertibility of $R_{R,j,j}(\mu_j, -1)$ implies that $\mathcal{X} = 0$, i.e., 
\[
[A_{i,1},A_{i,0}]_{q^2} = - \KK_iH_{i,1},
\]
establishing $R_{i,i}(0,0)$. In particular, this derivation required \eqref{eq: spec Serre aux1} for $t = 0, -1, \mu_j+1, \mu_j$, which relied on $S_{i,j}(0,1;t)$ and $S_{i,j}(0,0,t+1)$ for these values of $t$. The derivation of the latter, in turn, only required \eqref{eq: Serrel}--\eqref{eq: Serrel4}, \eqref{eq: H1 rule gen} and $R_{i,j}(-1,t)$, $R_{i,j}(0,t)$ for these values of $t$ (i.e., \eqref{eq: Urel3a}). These are all Group 1 and Group 3 Levendorskii relations, which we may legitimately assume without creating a circularity, with the exception of $R_{i,j}(0,t)$. 
Tracing through the argument in \S \ref{ss: Rij rel proof} shows that only \eqref{eq: sprel2}--\eqref{eq: sprel5} are additionally required to obtain $R_{i,j}(0,t)$. 
A similar argument can be used to establish $R_{i,i}(-2,-2)$. 
\enproof

\section{Shifted quantum affine algebras}

In this section, we recall the definition of shifted quantum affine algebras and their coproduct from \cite{FT19}. 

\subsection{Series definition}  

Let $\boldsymbol\Gamma$ be the graph of type $\mathsf{A}_N$ with vertex set $\indx$. 
We associate to it a generalized Cartan matrix $(a_{ij})_{i,j \in \indx}$, coweight lattice $\Lambda$, set of anti-dominant coweights $\Lambda_-$ and other data as in \S \ref{subsec: Gamma}. 

Let $\mu \in \Lambda$ be \emph{any} coweight and set $\mu_i = \ai(\mu)$. 
Consider the following generating series 
\[
x_i^{\pm}(z)=\sum_{r\in\Z}x_{i,r}^{\pm}z^r,
\quad
\phi_i^{+}(z)=\sum_{r\ge 0}\phi_{i,r}^{+}z^r,
\quad
\phi_i^{-}(z)=\sum_{r\le \mu_i}
\phi_{i,r}^{-}z^{r}, 
\quad \delta(z)=\sum_{n\in\Z}z^n. 
\]
Below, we recall the definition of shifted quantum affine algebras from \cite[\S 5]{FT19}. The definition below corresponds to the algebra which is denoted in \emph{loc.\ cit.} as $\mathcal{U}_{0,\mu}^{sc}$. While the original definition contains two shifts $\mu_-, \mu_+$, it is well known that $\mathcal{U}_{\mu_+,\mu_-}^{sc} \cong \mathcal{U}_{0,\mu_+ + \mu_-}^{sc}$. Our series conventions are opposite to those in \cite{FT19}, but match the ones in \cite{hernandez-05}.

\Defi \label{defi: qaa}
The \emph{$\mu$-shifted quantum affine algebra} $\Uq_\mu$ is the algebra generated by 
\[
x_{i,r}^{\pm}\;(r\in\mathbb{Z}), \qquad  \phi_{i,r}^{+} \; (r\ge0), \qquad \phi_{i,r}^- \; ( r \leq \mu_i ), \qquad (\phi_{i, 0}^+)^{-1}, \qquad (\phi_{i, \mu_i}^-)^{-1},
\]
for $i \in \indx$, 
subject to the relations 
\begin{align}
\phi_i^{\pm}(z)\phi_j^{\pm}(w)
&=\phi_j^{\pm}(w)\phi_i^{\pm}(z),\\
\phi_i^{+}(z)\phi_j^{-}(w)
&=\phi_j^{-}(w)\phi_i^{+}(z),\\[2mm]
(\phi_{i,0}^+)^{-1}\phi_{i,0}^+
&=\phi_{i,0}^+(\phi_{i,0}^+)^{-1}=1,\\
(\phi_{i,\mu_i}^-)^{-1}\phi_{i,\mu_i}^-
&=\phi_{i,\mu_i}^-(\phi_{i,\mu_i}^-)^{-1}=1,\\
\phi_i^{\epsilon}(z)x_j^{\pm}(w)
&=\frac{q^{\pm a_{ij}}w-z}{w-q^{\pm a_{ij}}z}\,
x_j^{\pm}(w)\phi_i^{\epsilon}(z)
\qquad (\epsilon\in\{\pm1\}),\\[2mm]
[x_i^{+}(z),x_j^{-}(w)]
&=(q-q\mi)\mi \delta_{ij} \delta(\textstyle\frac{w}{z})
( \phi_i^{+}(z) - \phi_i^{-}(w) ),\\[2mm]
(w-q^{\pm a_{ij}}z)x_i^{\pm}(z)x_j^{\pm}(w)
&=(q^{\pm a_{ij}}w-z)x_j^{\pm}(w)x_i^{\pm}(z),\\[2mm]
[x_i^{\pm}(z),x_j^{\pm}(w)]
&=0\qquad (a_{ij}=0),
\end{align}
and the Drinfeld--Serre relations 
\begin{align*}
\Sym_{z_1, z_2} \Bigl( x_i^{\pm}(z_1)x_i^{\pm}(z_2)x_j^{\pm}(w) 
-[2] x_i^{\pm}(z_{1})x_j^{\pm}(w)x_i^{\pm}(z_{2}) 
+ x_j^{\pm}(w)x_i^{\pm}(z_{1}) x_i^{\pm}(z_{2})
\Bigr)=0 
\end{align*}
for $a_{ij}=-1$. 
\enDefi
We will abbreviate
\[
K_{i,+} = \phi_{i,0}^+, \qquad K_{i,-} = \phi_{i,\mu_i}^-. 
\]
Let $\bar{\Uq}_\mu$ be the quotient of $\Uq_\mu$ by the two-sided ideal generated by $K_{i,+}K_{i,-} - 1$. 
We will write $K_i$ for the image of $K_{i,+}$ in this quotient. 

\subsection{Relations term-by-term}

The relations in Definition \ref{defi: qaa} can also be written in terms of modes as:
\begin{align}
[\phi_{i,r}^{\epsilon},\phi_{j,s}^{\epsilon'}]
&=0
\qquad (\epsilon,\epsilon'\in\{\pm\}),\\[2mm]
K_{i,+}^{-1}K_{i,+}
&=K_{i,+}K_{i,+}^{-1}=1,\\
K_{i,-}^{-1}K_{i,-}
&=K_{i,-}K_{i,-}^{-1}=1,\\[2mm] \label{eq: qaa Kx+}
K_{i,+}x_{j,s}^{\pm}
&=q^{\pm a_{ij}}x_{j,s}^{\pm}K_{i,+},\\ \label{eq: qaa Kx-}
K_{i,-}x_{j,s}^{\pm}
&=q^{\mp a_{ij}}x_{j,s}^{\pm}K_{i,-},\\[2mm] \label{eq: qaa hx}
[h_{i,m},x_{j,r}^{\pm}]
&=\pm\textstyle\frac{[ma_{ij}]}{m}x_{j,r+m}^{\pm},\\[2mm] \label{eq: qaa x+-}
[x_{i,r}^{+},x_{j,s}^{-}]
&= \delta_{ij}(q-q^{-1})\mi 
\left(\phi_{i,r+s}^{+}-\phi_{i,r+s}^{-}\right),
\\[2mm] \label{eq: qaa xx}
[ x_{i,r+1}^{\pm}, x_{j,s}^{\pm}]_{q^{\pm a_{ij}}} 
&= - [ x_{j,s+1}^{\pm}, x_{i,r}^{\pm}]_{q^{\pm a_{ij}}} 
\\[2mm]
[x_{i,r}^{\pm},x_{j,s}^{\pm}]
&=0 \qquad (a_{ij}=0),
\end{align}
\begin{align*} 
\Sym_{r_1, r_2} \Bigl( x_{i,r_{1}}^{\pm} x_{i,r_{2}}^{\pm}x_{j,s}^{\pm}
-[2] x_{i,r_{1}}^{\pm} x_{j,s}^{\pm}x_{i,r_{2}}^{\pm} 
+x_{j,s}^{\pm} x_{i,r_{1}}^{\pm} x_{i,r_{2}}^{\pm}
\Bigr)=0 \qquad (a_{ij}=-1), 
\end{align*} 
where 
\begin{align}
\sum_{k\ge 0}\phi_{i,k}^{\pm}z^k
&= z^{\mu_{i,\pm}} K_{i,\pm} 
\exp\left( \pm (q-q^{-1})
\sum_{m>0}h_{i,\pm m}z^{\pm m} \right),
\end{align}
for $\mu_{i,+} = 0$ and $\mu_{i,-} = \mu_i$.

\subsection{Comultiplication} 

By \cite[Theorem 10.22]{FT19}, 
for any $\mu_1, \mu_2 \leq 0$, there exists a unique algebra homomorphism
\eq \label{eq: comult gen}
\Delta = \Delta_{\mu_1, \mu_2} \colon {\Uq}_{\mu_1 + \mu_2} \to {\Uq}_{\mu_1} \otimes {\Uq}_{\mu_2} 
\eneq
given by 
\begin{align}
\Delta(K_{i,\pm}^{\pm}) &= K_{i,\pm}^{\pm1} \otimes K_{i,\pm}^{\pm1} \\
\Delta(x^+_{i,0}) &= x^+_{i,0} \otimes K_{i,+} + 1 \otimes x^+_{i,0}, \\ 
\Delta(x^+_{i,\mu_2-1}) &= x^+_{i,-1} \otimes K_{i,-} + 1 \otimes x^+_{i,\mu_2-1} + X_i^+, \\ 
\Delta(x^-_{i,1}) &= x^-_{i,1} \otimes 1 + K_{i,+} \otimes x^-_{i,1} + X_i^-, \\ 
\Delta(x^-_{i,\mu_1}) &= x^-_{i,\mu_1} \otimes 1 + K_{i,-} \otimes x^-_{i,0}, \\ 
\Delta(x^+_{i,r}) &= 1 \otimes x^+_{i,r} \qquad (\mu_2 \leq r \leq -1), \\ 
\Delta(x^-_{i,s}) &= x^-_{i,s} \otimes 1 \qquad (\mu_1 +1 \leq s \leq 0), \\
\Delta(h_{i,1}) &= h_{i,1} \otimes 1 + 1 \otimes h_{i,1} - (q^2 - q^{-2}) x^+_{i,0} \otimes x^{-}_{i,1} + H_i^+ , \\ 
\Delta(h_{i,-1}) &= h_{i,-1} \otimes 1 + 1 \otimes h_{i,-1} - (q^2 - q^{-2}) x^+_{i,-1} \otimes x^{-}_{i,0} + H_i^-, 
\end{align} 
where we refer to \cite[Theorem 10.22]{FT19} for the precise definiton of the terms $X_i^{\pm}, H_i^{\pm}$.

\subsection{The total anti-dominantly shifted algebra} 


We may also collect the full family of anti-dominantly shifted $\Uq_\mu$ (resp.\ $\bar{\Uq}_\mu$) into a single (non-unitary) algebra  
\[
\Uq_\bullet = \bigoplus_{\mu \in \Lambda_-} \Uq_\mu \mathbf{e}(\mu), \qquad 
\bar{\Uq}_\bullet = \bigoplus_{\mu \in \Lambda_-} \bar{\Uq}_\mu \mathbf{e}(\mu),
\]
with extra relations 
\[
\mathbf{e}(\mu)\mathbf{e}(\nu) = \delta_{\mu,\nu}\mathbf{e}(\mu), \qquad \mathbf{e}(\mu) a = a \mathbf{e}(\mu) \quad (a \in \Uq_\lambda). 
\]
Observe that, in this formulation, only the comultiplication in $\Uq_\bullet$ respects the $\Lambda_{-}$-grading while multiplication remains concentrated in each graded piece $\Uq_\mu$.

\section{Coideal realization: rank 1}
\label{sec: sub real1}

In this section assume that $| \indx | = 1$ so that $\Ui_\mu$ is the $\mu$-shifted $q$-Onsager algebra. We can drop the subscript $i$ everywhere. 

\subsection{Construction of the homomorphism} 

We show below that, for anti-dominant shifts, any central reduction $\uig$ admits a homomorphism to the corresponding shifted quantum affine $\mathfrak{sl}_2$. Remarkably, the images of the ``boundary'' Levendorskii generators $A_0$ and $A_{\mu-1}$ are given by almost the same formulae as in the ordinary non-shifted $q$-Onsager case.

\Thm \label{thm: homo rank 1}
For any anti-dominant shift $\mu \leq 0$, any $\boldsymbol\gamma \in (\C^\times)^{2}$ and $\boldsymbol\sigma = (\sigma, \sigma') \in \C^2$, there exists an algebra homomorphism 
\[
\iota_{\mu, \boldsymbol\sigma} \colon \uig \ \to \ \Uq_\mu
\]
sending 
\begin{align}
A_{0} &\mapsto x^-_{\mu} -  q^{-2} C^{\mu} \gamma x^+_{0} K_{+}\mi + \sigma K_{+}\mi,   \\
A_{\mu-1} &\mapsto x^-_{1} - C^{\mu-1} \gamma x^+_{-1} K_{-}\mi + \sigma' K_{-}\mi,   \\ 
A_{r} & \mapsto x^-_{\mu - r}, 
\end{align}
for  $\mu \leq r \leq -1$. 
\enthm

\Proof 
If $\mu = 0$, then the formulae above reduce to the usual embedding of the $q$-Onsager algebra into quantum affine $\mathfrak{sl}_2$ as in, e.g., \cite{kolb-14}. Therefore, we may assume that $\mu \leq - 1$.  
By Theorem \ref{thm: 1Lev}, it suffices to show that all the Levendorskii relations are preserved. 

First consider relation \eqref{eq: 1Urel2} with $|r-s| = -\mu - 1$. If $\mu \leq - 2$, the LHS are sent to
\begin{align}
[A_{0}, A_{\mu}]_{q^2} + [A_{\mu+1}, A_{-1}]_{q^2} &\mapsto [x^-_{\mu} - q^{-2}C^\mu \gamma  x^+_0 K_+\mi + \sigma K_+\mi, x^-_0]_{q^2} + [x^-_{-1}, x^-_{\mu+1}]_{q^2}   \\ 
&= -q^{-2} C^\mu \gamma  [ x^+_0 K_+\mi, x^-_0]_{q^2} = - C^\mu \gamma \rho, \\ 
[A_{-1}, A_{\mu-1}]_{q^2} + [A_{\mu}, A_{-2}]_{q^2} &\mapsto [x^-_{\mu+1}, x^-_{1} - C^{\mu-1} \gamma  x^+_{-1} K_-\mi + \sigma' K_-\mi]_{q^2} + [x^-_{0}, x^-_{\mu+2}]_{q^2}   \\ 
&= - C^{\mu-1} \gamma  [x^-_{\mu+1}, x^+_{-1} K_-\mi]_{q^2} = - C^{\mu-1} \gamma \rho, 
\end{align}
which clearly match the RHS. In the two cases above, the former equalities follow from the fact that 
the expressions $[x^-_{\mu}, x^-_0]_{q^2} + [x^-_{-1}, x^-_{\mu+1}]_{q^2}$ and $[x^-_{\mu+1}, x^-_1]_{q^2} + [x^-_{0}, x^-_{\mu+2}]_{q^2}$ vanish by \eqref{eq: qaa xx}, and $\sigma[K_+\mi, x^-_0]_{q^2}$ and $\sigma' [x^-_{\mu+1}, K_-\mi]_{q^2}$ vanish by \eqref{eq: qaa Kx+}--\eqref{eq: qaa Kx-}. The latter equalities follow from \eqref{eq: qaa x+-}. 
If $\mu = -1$, the calculation above has to be adjusted in the obvious way. 

If $|r-s| \leq -\mu-2$ then there are three cases to be considered: $R(r,-1)$, $R(\mu-1,s)$ and $R(r,s)$ with $r,s \notin \{\mu-1,-1\}$. In the third case, $R(r,s)$ reduces to \eqref{eq: qaa xx} under $\iota_\mu$. 
In the remaining cases, the calculation is similar to the one above. Indeed, we get 
\begin{align}
[A_{0}, A_{r}]_{q^2} + [A_{r+1}, A_{-1}]_{q^2} &\mapsto [x^-_{\mu} - q^{-2}C^\mu \gamma  x^+_0 K_+\mi + \sigma K_+\mi, x^-_{\mu-r}]_{q^2} + [x^-_{\mu-r-1}, x^-_{\mu+1}]_{q^2}   \\ 
&= -q^{-2} C^\mu \gamma  [ x^+_0 K_+\mi, x^-_{\mu-r}]_{q^2} = 0, \\ 
[A_{s+1}, A_{\mu-1}]_{q^2} + [A_{\mu}, A_{s}]_{q^2} &\mapsto [x^-_{\mu-s-1}, x^-_{1} - C^{\mu-1} \gamma  x^+_{-1} K_-\mi + \sigma' K_-\mi]_{q^2} + [x^-_{0}, x^-_{\mu-s}]_{q^2}   \\ 
&= - C^{\mu-1} \gamma  [x^-_{\mu-s-1}, x^+_{-1} K_-\mi]_{q^2} = 0, 
\end{align}
as required. 

Next, we calculate the image of $H_{1} = - \gamma\mi \Ck^{1-\mu} \big( [A_{\mu}, A_{-1}]_{q^2} + [A_0, A_{\mu-1}]_{q^2}  \big)$. We have 
\begin{align}
- \gamma \Ck^{\mu-1} H_1 &\mapsto [x^-_{0}, x^-_{\mu+1}]_{q^2} + [x^-_{\mu}, x^-_{1}]_{q^2} \\ 
&\quad + \sigma'[x^-_\mu, K_-\mi]_{q^2} + \sigma[K_+\mi, x^-_1]_{q^2} \\ 
&\quad -C^{\mu-1} \gamma [x^-_\mu, x^+_{-1} K_-\mi]_{q^2} - q^{-2} C^\mu \gamma [x^+_0 K_+\mi, x^-_1]_{q^2} \\ 
&\quad + q^{-2} \Ck^{2\mu-1} \gamma^2 [x^+_0K_+\mi, x^+_{-1}K_-\mi]_{q^2} \\
&\quad - q^{-2} C^\mu \gamma \sigma' [x^+_0 K_+\mi, K_-\mi]_{q^2} - C^{\mu-1} \gamma \sigma [x^+_{-1} K_-\mi, K_+\mi]_{q^2} \\
&\quad +\sigma \sigma' [K_+\mi, K_-\mi]_{q^2}. 
\end{align}
The first two rows on the RHS vanish by \eqref{eq: qaa xx} and \eqref{eq: qaa Kx+}--\eqref{eq: qaa Kx-}. After multiplication by $- \gamma \Ck^{\mu-1}$, the third row on the RHS becomes $Ch_1 - h_{-1}$ by \eqref{eq: qaa x+-}, and the fourth to sixth rows become $Y_1 K_+\mi K_-\mi$ to $Y_3 K_+\mi K_-\mi$, respectively, with the $Y_k$ defined below. In effect,  
\begin{align}
\iota_\mu(H_{1}) \label{eq: H1 emb formula1}
&= C h_{1}-h_{-1} + (Y_1 + Y_2 + Y_3) K_{+}^{-1}K_{-}^{-1}, \\
Y_1 &= (1-q^{-4})\gamma C^{\mu}
x^+_{0}x^+_{-1}, \\
Y_2 &= (q^{-2}-q^2) \big(C\sigma' x^+_{0}+\sigma x^+_{-1}\big), \\
Y_3 &= (q^2-1)\gamma^{-1}C^{1-\mu} \sigma\sigma'. 
\end{align} 

Let us prove \eqref{eq: 1Urel1}. Given the admissible values of $r$, we have $\iota_\mu(A_r) = x^-_{\mu - r}$. Moreover, using \eqref{eq: qaa hx}, we get 
\eq \label{eq: embpr 1}
[Ch_1 - h_{-1}, x^-_{\mu-r}] = - [2]C x^-_{\mu-r+1} + [2]x^-_{\mu-r-1}. 
\eneq 
First take $\mu < r < -1$. Then the RHS of \eqref{eq: embpr 1} already equals the image of $[2](A_{r+1} - C A_{r-1})$. It also follows straightforwardly from \eqref{eq: qaa x+-} and the range of $r$ that $[Y_1 + Y_2 + Y_3, x^-_{\mu-r}] = 0$. Next, suppose that $r = -1$. Then 
\begin{align} \label{eq: embpr 2}
[Y_1, x^-_{\mu+1}] = -q^{-2}[2] \gamma C^\mu x^+_0 K_-, \qquad [Y_2, x^-_{\mu+1}] = [2] \sigma K_-, \qquad [Y_3, x^-_{\mu+1}] = 0. 
\end{align} 
Summing \eqref{eq: embpr 1}--\eqref{eq: embpr 2} gives precisely the image of $[2](A_0 - \Ck A_{-2})$. 
The calculation for $r = \mu$ is analogous. 

It remains to show that \eqref{eq: 1HHlev} is preserved. This is the only relation which is difficult to check by hand. 
As the first step, one has to calculate the image of $A_1$. 
Then, by Remark \ref{rem: theta}, we may take $\overline{\Theta}_{2-\mu} = - C^{\mu-1}R_L(\mu-1,0)$, which reduces checking \eqref{eq: sprel1} to 
\[
\iota_\mu\Big( \big[H_1, [A_1, A_{\mu-1}]_{q^2} + [A_\mu, A_0]_{q^2} \big] \Big) = 0. 
\] 
This calculation is elementary, but somewhat involved. A computational verification can be found in \cite{LevQSP}. 
\enproof

\subsection{Comodule property}

By \cite[Theorem 10.6]{FT19}, the comultiplication maps in \S \ref{eq: comult gen} specialize in the following way in rank one. For any $\mu_1, \mu_2 \leq 0$, there exists a unique algebra homomorphism
\[
\Delta = \Delta_{\mu_1, \mu_2} \colon \ \bar{\Uq}_{\mu_1 + \mu_2} \to \bar{\Uq}_{\mu_1} \otimes \bar{\Uq}_{\mu_2} 
\]
given by 
\begin{align}
\Delta(K^{\pm1}) &= K^{\pm1} \otimes K^{\pm1} \\
\Delta(x^+_0) &= x^+_0 \otimes K + 1 \otimes x^+_0, \\ 
\Delta(x^+_{\mu_2-1}) &= x^+_{-1} \otimes K\mi + 1 \otimes x^+_{\mu_2-1}, \\ 
\Delta(x^-_{1}) &= x^-_{1} \otimes 1 + K \otimes x^-_1, \\ 
\Delta(x^-_{\mu_1}) &= x^-_{\mu_1} \otimes 1 + K\mi \otimes x^-_0, \\ 
\Delta(x^+_r) &= 1 \otimes x^+_r \qquad (\mu_2 \leq r \leq -1), \\ 
\Delta(x^-_s) &= x^-_s \otimes 1 \qquad (\mu_1 +1 \leq s \leq 0), \\
\Delta(h_1^+) &= h_1 \otimes 1 + 1 \otimes h_1 - (q^2 - q^{-2}) x^+_0 \otimes x^{-}_1, \\ \label{eq: h-1 coprod}
\Delta(h_{-1}^+) &= h_{-1} \otimes 1 + 1 \otimes h_{-1} - (q^2 - q^{-2}) x^+_{-1} \otimes x^{-}_0. 
\end{align} 

Let us abbreviate $\mathbf{U}^\imath_{\mu,\boldsymbol\gamma,\boldsymbol\sigma} = \iota_{\mu, \boldsymbol\sigma}(\uig)$. We will write $\bar{\mathbf{U}}^\imath_{\mu,\boldsymbol\gamma,\boldsymbol\sigma}$ for its  image in~$\bar{\mathbf{U}}_\mu$.

\Thm \label{thm: coideal property 1}
For any $\mu_1, \mu_2 \leq 0$, we have 
\[
\Delta_{\mu_1, \mu_2} (\bar{\mathbf{U}}^\imath_{\mu,\boldsymbol\gamma,\boldsymbol\sigma} ) \subseteq \bar{\Uq}_{\mu_1} \otimes \bar{\mathbf{U}}^\imath_{\mu_2,\boldsymbol\gamma',\boldsymbol\sigma},
\]
with $\boldsymbol\gamma' = (C^{\mu_1}\gamma_1, C)$. 
\enthm

\Proof
Abbreviate $\Delta = \Delta_{\mu_1, \mu_2} \circ \iota_{\mu, \boldsymbol\sigma}$. 
It suffices to show that the images of the Levendorskii generators lie in $\bar{\Uq}_{\mu_1} \otimes \bar{\mathbf{U}}^\imath_{\mu_2,\boldsymbol\gamma',\boldsymbol\sigma}$. In fact, by Remark \ref{rem: bound gen}, it suffices to check that the images of the boundary generators $A_{0}, A_{-1}, A_{\mu}$ and $A_{\mu-1}$ are in $\bar{\Uq}_{\mu_1} \otimes \bar{\mathbf{U}}^\imath_{\mu_2,\boldsymbol\gamma',\boldsymbol\sigma}$. Below we explicitly consider the generic case where none of the four boundary generators coincide - the remaining cases follow by a straightforward modification. 
It is immediately clear that $A_{\mu-1}$ and $A_{\mu}$ are in $\bar{\Uq}_{\mu_1} \otimes \bar{\mathbf{U}}^\imath_{\mu_2,\boldsymbol\gamma',\boldsymbol\sigma}$. Indeed, 
\begin{align} 
\Delta(A_\mu) &= \Delta_{\mu_1, \mu_2}(x^-_0) = 
\left\{ \begin{array}{l l}
x^-_0 \otimes 1 + K\mi \otimes A_0 & \mbox{if } \mu_1 = 0, \\[2pt] 
x^-_0 \otimes 1  & \mbox{if } \mu_1 \leq -1, 
\end{array} \right. \\
\Delta(A_{\mu-1}) &= \Delta_{\mu_1, \mu_2}(x^-_{1} - C^{\mu-1} \gamma x^+_{-1} K + \sigma' K) \\ 
&=
\left\{ \begin{array}{l l}
(x^-_1 - C^{\mu-1}\gamma x^+_{-1}K )\otimes 1 + K \otimes A_0 & \mbox{if } \mu_2 = 0, \\[2pt] 
x^-_1 \otimes 1 + K \otimes A_0 & \mbox{if } \mu_2 \leq -1.  
\end{array} \right. 
\end{align} 

Next, consider $A_{-1}$ and $A_0$. We have 
\begin{align}
\Delta(A_{-1}) &= \Delta_{\mu_1, \mu_2}(x^-_{\mu+1}), \\ 
\Delta(A_0) &= \Delta_{\mu_1, \mu_2}(x^-_{\mu} -  q^{-2} C^{\mu} \gamma x^+_{0} K\mi + \sigma K\mi) \\ 
&= \Delta_{\mu_1, \mu_2}(x^-_{\mu}) - q^{-2} C^{\mu} \gamma (x^+_0K\mi \otimes 1 + K\mi \otimes x^+_0 K\mi) + \sigma K\mi \otimes K\mi. 
\end{align}
We know that $\Delta_{\mu_1, \mu_2}(x^-_{\mu_1}) = x^-_{\mu_1} \otimes 1 + K\mi \otimes x^-_0$. 
Repeatedly applying the commutator with \eqref{eq: h-1 coprod}, we see that 
\[
\Delta_{\mu_1, \mu_2}(x^-_{\mu+1}) \equiv x^-_{\mu + 1} \otimes 1 + K\mi \otimes x^-_{\mu_2+1}, \qquad 
\Delta_{\mu_1, \mu_2}(x^-_{\mu}) \equiv x^-_{\mu} \otimes 1 + K\mi \otimes x^-_{\mu_2}
\]
modulo $\bar{\Uq}_{\mu_1} \otimes \bar{\mathbf{U}}^\imath_{\mu_2,\boldsymbol\gamma',\boldsymbol\sigma}$. It follows that
\[
\Delta(A_{-1}) \equiv x^-_{\mu + 1} \otimes 1 + K\mi \otimes A_{-1}, \qquad 
\Delta(A_{0}) \equiv x^-_{\mu} \otimes 1 + K\mi \otimes A_0
\]
modulo $\bar{\Uq}_{\mu_1} \otimes \bar{\mathbf{U}}^\imath_{\mu_2,\boldsymbol\gamma',\boldsymbol\sigma}$, completing the proof. 
\enproof

\section{Coideal realization: higher rank}
\label{sec: sub real2} 

In this section, we assume that $\boldsymbol\Gamma$ is the graph of type $\mathsf{A}_N$ with vertex set $\indx = \{1, \cdots, N\}$. 
We generalize the results of \S \ref{sec: sub real1} to arbitrary rank and establish the injectivity of $\iota_\mu$. 
In other words, we realize a central reduction of the $\mu$-shifted quantum affine symmetric pair algebra $\Ui_\mu$, for $\mu$ anti-dominant, as a subalgebra of the corresponding $\mu$-shifted quantum affine algebra $\Uq_\mu$, and show that it satisfies a graded coideal property.

\subsection{Correction terms} 

Fix $\boldsymbol\gamma \in (\mathbb C^\times)^{N+1}$. Set $\kappa_j=K_{j,+}^{-1}K_{j,-}^{-1}$ and, for $1 \leq a \leq i \leq b \leq N$, define 
\begin{align*}
 E_{a,b} &=[x^+_{a,0},[x^+_{a+1,0},\cdots,[x^+_{b-1,0},x^+_{b,-1}]_{q^{-1}} \cdots]_{q^{-1}}]_{q^{-1}}, \\
 L_{a,i} &=[x^+_{i-1,0},[x^+_{i-2,0},\cdots,x^+_{a,0}]_q\cdots]_q, \\ 
 R_{i,b} &=[x^+_{i+1,0},[x^+_{i+2,0},\cdots,x^+_{b,0}]_q\cdots]_q.
\end{align*}
Moreover, set 
\[
 \ell_{ab}=b-a,\quad \epsilon_{ab}(i)=\delta_{a<i}+\delta_{i<b},
 \quad \gamma_{ab}=\prod_{j=a}^b\gamma_j,
 \quad \mu_{ab}=\sum_{j=a}^b\mu_j, \quad 
\kappa_{ab} = \prod_{\substack{a\le j\le b\\j\ne i}}\kappa_j, 
\]
\[
\tau_{ab} = (-1)^{i-a+1+\epsilon_{ab}(i)} q^{-\ell_{ab}-\epsilon_{ab}(i)}(q-q\mi)^{\epsilon_{ab}(i)} 
\gamma_{ab}C^{\mu_{ab}-1} \kappa_{ab}.
\]
Define the ``correction term'' $Q_i$ as  
\begin{equation}\label{eq:ta-correction} 
 Q_i=\sum_{\substack{a\le i\le b\\a<b}} \tau_{ab} \cdot L_{a,i}R_{i,b}E_{a,b}K_{i,-}^{-1}. 
\end{equation}

\subsection{Construction of the homomorphism} 

The construction is similar to the one in Theorem \ref{thm: homo rank 1} with two major differences. Firstly, the formula for $A_{i,\mu_i-1}$ requires a complicated correction term, reminiscent of those appearing in Finkelberg and Tsymbaliuk's comultiplication formulae (c.f.\ \cite[Theorem 10.22]{FT19}) and those in \cite[Lemma 9.9]{Przez-23}. Secondly, we assume $\sigma_i = \sigma'_i = 0$. This constraint could perhaps be circumvented by introducing more general correction terms, but we do not consider this here. 

\Thm \label{thm: homo rank 2} 
For any anti-dominant shift $\mu \in \Lambda_-$ and any $\boldsymbol\gamma \in (\C^\times)^{N+1}$, there exists an algebra homomorphism 
\[
\iota_\mu \colon \uig \ \to \ \Uq_\mu
\]
sending 
\begin{align}
A_{i,0} &\mapsto x^-_{i, \mu_i} -  q^{-2} C^{\mu} \gamma_i x^+_{i, 0} K_{i,+}\mi,   \\
A_{i,\mu_i-1} &\mapsto x^-_{i, 1} - C^{\mu-1} \gamma_i x^+_{i, -1} K_{i,-}\mi  +  Q_i, \\ 
A_{i,r} & \mapsto x^-_{i, \mu_i - r}, 
\end{align}
for  $\mu_i \leq r \leq -1$. 
\enthm 

We will prove Theorem \ref{thm: homo rank 2} in \S \ref{sss: unshifted proof}--\S \ref{sss: shifted proof} below. 

\subsubsection{Unshifted case} 
\label{sss: unshifted proof}

We first prove the following special case of Theorem \ref{thm: homo rank 2}. 

\Lem
Theorem \ref{thm: homo rank 2} holds for $\mu = 0$. Moreover, $\iota_0$ coincides with the standard embedding of $\mathbf{U}^\imath_{0,\boldsymbol\gamma}$ into $\mathbf{U}_0$ from \cite{kolb-14} (see, e.g., \cite[\S 8.4]{Przez-23} for an explicit definition). 
\enlem

\Proof
This is already implicitly computed in \cite[\S 9]{Przez-23} and constitutes a refinement of Lemma 9.9 in \emph{op.\ cit.} Indeed, \cite[(9.4)]{Przez-23} shows that the image of $A_{i,-1}$ under the standard embedding equals 
\[
\iota_0(A_{i,-1}) = x^-_{i,1} - (-1)^{N-i} C^{-1} \gamma_i \mathcal{P}_i(Z_{1,N}), 
\]
with 
\begin{align}
\mathcal P_i(-) &= \operatorname{ad}_q(B_{i-1})\cdots\operatorname{ad}_q(B_1) \operatorname{ad}_q(B_{i+1})\cdots\operatorname{ad}_q(B_N)(-), 
\end{align}
\[
B_j = x^-_{j,0} - q^{-2}\gamma_j x^+_{0,j} K_{j,+}^{-1}, \qquad 
\operatorname{ad}_q(X)(Y) = [X,Y]_q, \qquad 
Z_{a,b} = E_{a,b}\prod_{j=a}^b K_{j,-}^{-1}. 
\] 

It is easy to check that 
\begin{align} \label{eq: mu0 form 1}
[x^-_{b,0}, Z_{a,b}]_q &= - Z_{a, b-1}, \\  
[x^+_{b,0} K_{b,+}\mi, Z_{a,b}]_q &= - (q-q\mi) x^+_{b,0} Z_{a,b}  K_{b,+}\mi, \label{eq: mu0 form 2}
\end{align} 
with analogous formulae for $q$-commutators with $x^-_{a,0}$ and $x^+_{a,0}K_{a,+}\mi$. 
Moreover, each $\ad_q(B_j)$ can be split into a linear combination of $\ad_q(x^-_{j,0})$ and $\ad_q(x^+_{j,0}K_{j,+}\mi)$. Let us call these ``negative'' and ``positive adjoints''. It follows from the arguments of \cite[\S 9]{Przez-23} that in each of the segments $\operatorname{ad}_q(B_{i-1})\cdots\operatorname{ad}_q(B_1)$ and $\operatorname{ad}_q(B_{i+1})\cdots\operatorname{ad}_q(B_N)$, only strings of the form $(+, \cdots, +, -, \cdots, -)$ are non-zero, i.e., a negative adjoint after a positive adjoint vanishes. Therefore, tuples of admissible strings are indexed by $a \leq i \leq b$ with $a \neq b$, which corresponds to the sequence 
\begin{align} 
&\ad_q(x^+_{i-1,0} K_{i-1,+}\mi) \cdots \ad_q(x^+_{a,0} K_{a,+}\mi) \ad_q(x^-_{a-1,0}) \cdots \ad_q(x^-_{1,0}) \ \circ \\ 
\circ &\ad_q(x^+_{i+1,0} K_{i+1,+}\mi) \cdots \ad_q(x^+_{b,0} K_{b,+}\mi) \ad_q(x^-_{b+1,0}) \cdots \ad_q(x^-_{N,0}) 
\end{align}
of adjoint operators. 
Given a sequence as above and using \eqref{eq: mu0 form 1}, we see that the initial seqment of negative adjoints transforms $Z_{1,N}$ into $Z_{1,b}$, and the first positive adjoint gives a scalar multiple of $x^+_{b,0} Z_{1,b} K_{b,+}\mi$, by \eqref{eq: mu0 form 2}. Subsequent positive adjoints affect only $x^+_{a,0}$ and give rise to $R_{i,b} Z_{a,b} \prod_{j=i+1}^b K_{j,+}\mi$. Repeating these steps with the second string of adjoint operators yields the contribution $ \tau_{ab} \cdot L_{a,i}R_{i,b}E_{a,b}K_{i,-}^{-1}$. 
\enproof

\subsubsection{Basic relation check} 

We will now check that a few easy relations are preserved. 

\Lem \label{lem: basic rels lemma}
The map $\iota_\mu$, defined as in Theorem \ref{thm: homo rank 2}, preserves relations \eqref{eq: Urel1} (for $i=j$) and \eqref{eq: Urel3}. 
\enlem

\Proof
There is nothing to show if $\mu_i = 0$, so suppose $\mu_i \leq -1$. First consider relation \eqref{eq: Urel3}. The proof is the same as in Theorem \ref{thm: homo rank 1}, except the relations involving $A_{i, \mu_i - 1}$. Observe that 
\eq \label{eq: bas rel ch1}
 [x^-_{i,s}, Q_i]_{q^2} = 0 \qquad (\mu_i<s\le0). 
\eneq
Indeed, $x^-_{i,s}$ obviously commutes with all the generators appearing in $Q_i$ except for $K_{i,-}\mi$ and $x^+_{i, 0}$. However, the associated Cartan mode $\phi^-_{i,s}$ vanishes since $s > \mu_i$. This implies that $\iota_\mu$ preserves all the relations \eqref{eq: Urel3}(the relation $R_{i,i}(\mu_i-1,-1)$ is the definition of $H_{i,1}$ and does not need to be checked). 

Next, we check \eqref{eq: Urel1}. 
Let $\widehat{H}_{i,1}$ be the expression obtained by omitting $\mathcal Q_i$ from the definition of $\iota_\mu(A_{i,\mu_i-1})$. As in \eqref{eq: H1 emb formula1}, we get 
\[
 \iota_\mu(\widehat H_{i,1}) = Ch_{i,1} - h_{i,-1}
 +(1-q^{-4})\gamma_iC^{\mu_i}x^+_{i,0} x^+_{i,-1}\kappa_i.
\]
Abbreviating $u_i=-q^{-2}C^{\mu_i}\gamma_i x^+_{i,0} K_{i,+}^{-1}$, we have 
\[
\iota_\mu(H_{i,1} - \widehat H_{i,1}) = -\gamma_i^{-1}C^{1-\mu_i} \bigl([x^-_{i,\mu_i}, Q_i]_{q^2}
+ [u_i, Q_i]_{q^2} \bigr).
\]
To show that \eqref{eq: Urel1} is preserved, we need to check that 
\eq \label{eq: bas rel ch2}
[ \iota_\mu(H_{i,1} - \widehat H_{i,1}), x^-_{i,0} ] = -[2] C \delta_{s, 0} Q_i, 
\eneq
for $\mu_i < s \leq 0$. 
Observe that 
\begin{align}[x^-_{i,\mu_i}, Q_i]_{q^2} \in \langle x^+_{j,r}, \kappa_k^{\pm1} \mid j\ne i, \ r \in \indx \rangle. 
\label{eq:ta-local-other}
\end{align}
Indeed, commuting $x^-_{i,\mu_i}$ with $x^+_{i,0}$ inside $Q_i$ produces a remainder term with $K_{i,-}$, which can be moved to the right to cancel $K_{i,-}\mi$, leaving an expression in $ x^+_{j,r}, \kappa_k^{\pm1}$ only. It follows immediately from \eqref{eq:ta-local-other} that $[x^-_{i,\mu_i}, Q_i]_{q^2}$ commutes with all $x^-_{i,s}$. 
Using \eqref{eq: bas rel ch1}, one can also easily show that $[u_i, Q_i]_{q^2}$ commutes with $x^-_{i,s}$ for $\mu_i<s<0$. At $s=0$, using $Q_ix^-_{i,0} = q^{-2}x^-_{i,0} Q_i$ and $[u_i, x^-_{i,0}]_{q^2}=(q-q\mi)\mi  \gamma_iC^{\mu_i}$, we obtain
\[
 [[u_i, Q_i]_{q^2},x^-_{i,0}] = [2]\gamma_iC^{\mu_i} Q_i, 
\]
which, after multiplication by $-\gamma_i^{-1}C^{1-\mu_i}$, matches the RHS of \eqref{eq: bas rel ch2}. 
\enproof

\subsubsection{Proof}
\label{sss: shifted proof}

We now complete the proof of Theorem \ref{thm: homo rank 2}. 
Let $\mathcal R=\mathbb C(q)[t]$ and $\mathcal R_{\mathrm{loc}}=\mathbb C(q)[t,t^{-1}]$. Define
\[
 p_i(z)=\begin{cases}1+tz^{-\mu_i},&\mu_i<0,\\1,&\mu_i=0,\end{cases}
 \qquad
 \gamma_i(t)=\begin{cases}tC^{\mu_i}\gamma_i,&\mu_i<0,\\
 \gamma_i,&\mu_i=0.\end{cases}
\]
By a version of \cite[Lemma~10.24]{FT19}, there is a homomorphism
\begin{equation}\label{eq:ta-shift}
 \sigma_t: \mathbf{U} _0\otimes\mathcal R_{\mathrm{loc}}
       \longrightarrow \mathbf{U}_\mu\otimes\mathcal R_{\mathrm{loc}},
 \qquad
 \begin{cases}
 x_i^+(z)\mapsto x_i^+(z),\\
 x_i^-(z)\mapsto p_i(z)x_i^-(z),\\
 \phi_i^\pm(z)\mapsto p_i(z)\phi_i^\pm(z).
 \end{cases}
\end{equation}
Consider $\mathbf{U}^\imath_{0,\boldsymbol\gamma}$ with parameters $\gamma(t)=(\gamma_1(t),\ldots,\gamma_N(t),C)$, and the composition 
\[
\Psi_t \colon \mathbf{U}^\imath_{0,\boldsymbol\gamma} \overset{\iota_0}{\longrightarrow} \mathbf{U}_{0} \otimes\mathcal R_{\mathrm{loc}} \overset{\sigma_t}{\longrightarrow} \mathbf{U}_{\mu} \otimes\mathcal R_{\mathrm{loc}}. 
\]

Let us now introduce some notation. We will use the superscript $0$ to denote elements of $\mathbf{U}^\imath_{0,\boldsymbol\gamma}$, e.g., $A_{i,0}^0$. Moreover, in $\mathbf{U}_\mu$, let 
\begin{align}
B_{i,0} = x^-_{i, \mu_i} -  q^{-2} C^{\mu} \gamma_i x^+_{i, 0} K_{i,+}\mi, \qquad 
B_{i,\mu_i-1} =  x^-_{i, 1} - C^{\mu-1} \gamma_i x^+_{i, -1} K_{i,-}\mi  +  Q_i, 
\end{align}
and $B_{i,r} = x^-_{i, \mu_i - r}$ for $\mu_i \leq r \leq -1$. 
Let 
\[
G_{i,1} = -\frac{C^{1-\mu_i}}{2^{\delta_{\mu_i,0}}\gamma_i} \bigl([B_{i,0},B_{i,\mu_i-1}]_{q^2}
+[B_{i,\mu_i},B_{i,-1}]_{q^2}\bigr), 
\]
and extend the definition of $B_{i,r}$ to all $r \in \Z$ via the rule
\eq \label{eq: proof gen homo 1}
[G_{i,1}, B_{i,r}] = [2] (B_{i,r+1} - C B_{i,r-1}).
\eneq
By Lemma \ref{lem: basic rels lemma}, this is consistent with the definitions already made. 

Using Lemma \ref{lem: basic rels lemma} again, we get 
\begin{align} \label{eq: proof gen homo 2}
\ \ \ \Psi_t(H_{i,1}^0) = -C \gamma_i(t)\mi [ \Psi_t(A_{i,0}^0), \Psi_t(A_{i,-1}^0)]_{q^2} 
= G_{i,1} + \delta_{\mu_i,-1} (q-q\mi)\mi (t\mi + Ct). 
\end{align} 
It follows from \eqref{eq: proof gen homo 1}--\eqref{eq: proof gen homo 2} that 
\eq
\label{eq:ta-all-real}
 \Psi_t(A^0_{i,r}) =
 \begin{cases}B_{i,r+\mu_i}+tB_{i,r}, & \mu_i<0,\\
 B_{i,r},& \mu_i=0,
 \end{cases}
\eneq
for all $r \in \Z$. In particular, the images of all the real modes $A^0_{i,r}$ under $\Psi_t$ are polynomial in $t$. 
The same is true for the images of all the imaginary modes $\Theta_{i,s}^0$. Indeed, $\Theta^0_{i,0} = (q-q\mi)\mi \gamma_i(t)$, and all the other imaginary modes are obtained from real modes via the expression $R^0_{L,i,i}$. Hence the polynomiality of \eqref{eq:ta-all-real} implies the polynomiality of $\Psi_t(\Theta_{i,s}^0)$.  

Since both $\Psi_t(A^0_{i,r})$ and $\Psi_t(\Theta_{i,s}^0)$ lie in $\mathbf{U}_{\mu} \otimes\mathcal R$, they can be specialized at $t=0$. 
Set 
\[
 S_i(z)= C^{-\mu_i} \left.\Psi_t(\Theta_i^0(z))\right|_{t=0},\qquad 
 B_i(z)=\sum_{r\in\mathbb Z}B_{i,r}z^r.
\]
The initial modes of $S_i(z)$ satisfy 
\begin{equation}\label{eq:ta-leading}
 S_{i,s}=0\ (s<-\mu_i),\qquad
 S_{i,-\mu_i}=(q-q\mi)\mi\gamma_i,\qquad
 S_{i,1-\mu_i}=\gamma_iG_{i,1}.
\end{equation}
Moreover, it follows that $\Psi_t$, composed with specialization, sends 
\begin{equation}\label{eq:ta-specialization}
 A_i^0(z)\longmapsto z^{-\mu_i}B_i(z),\qquad
 \Theta_i^0(z)\longmapsto C^{\mu_i}S_i(z),
\end{equation} 
and so the series on the RHS above must satisfy the unshifted version of equations \eqref{eq: drel1}--\eqref{eq: drel6}. 
The shifted version is obtained by multiplying both sides by appropriate Laurent monomials. 
The series $B_i(z), S_i(z)$ therefore satisfy Definition \ref{defi: shifted iqsp} with the
required leading terms, completing the proof. 

\subsection{Injectivity} 

We will now establish that the homomorphism $\iota_\mu$ from Theorem \ref{thm: homo rank 2} is injective. 

\Thm \label{thm: injectivity}
The homomorphism $\iota_\mu \colon \uig \to \Uq_\mu$ is injective. 
\enthm

We will prove Theorem \ref{thm: injectivity} by an associated graded argument. The proof is split over \S \ref{sss: centralizers}--\S \ref{sss: injectivity proof final}.

\subsubsection{Centralizers} 
\label{sss: centralizers}

Given a subalgebra $A \subset B$ (resp.\ element $X \in B$), let $Z_B(A)$ (resp.\ $Z_B(X)$) denote the centralizer of $A$ (resp.\ $X$) in $B$. 
Abbreviate 
\[
 D_{i,m}=C^mh_{i,m}-h_{i,-m},\qquad
 \kappa_i=K_{i,+}^{-1}K_{i,-}^{-1},\qquad
 \mathcal{Z} = \C(q) [\kappa_i^{\pm1}\mid i\in \indx].
\]
Write $\mathbf U_\mu^\pm$ for the subalgebras generated by the
$x_{i,r}^\pm$, and
\[
 \mathbf U_\mu^0
 = \C(q) [K_{i,+}^{\pm1},K_{i,-}^{\pm1},h_{i,\pm m}\mid i\in \indx,\ m\ge1].
\]
The elements $\kappa_i$ are central, and relation \eqref{eq: qaa hx} gives 
\begin{equation}\label{eq:grinj-D}
 [D_{i,m},x_{j,s}^-] = \frac{[ma_{ij}]}m(x_{j,s-m}^--C^mx_{j,s+m}^-).
\end{equation}
We will need the following auxiliary lemma. 

\Lem \label{lem:grinj-centralizers}
One has
\begin{equation}\label{eq:grinj-centralizers}
 \bigcap_{i\in \indx} Z_{\mathbf U_\mu}(D_{i,1}) = \mathbf U_\mu^0,\qquad
 \mathbf U_\mu^0 \cap Z_{\mathbf U_\mu}(\mathbf U_\mu^-) = \mathcal{Z}.
\end{equation}
\enlem

\Proof
We sketch a proof using the triangular decomposition $\mathbf U_\mu^-\otimes\mathbf U_\mu^0\otimes\mathbf U_\mu^+
\xrightarrow{\sim}\mathbf U_\mu$ from \cite[Proposition~5.1]{FT19}, together with the shuffle realizations
of the positive and negative halves, which are independent of $\mu$ (see \cite[Theorem~3.66 and Remark~3.72]{FT19b}).

Let $z_{j,a}$ and $w_{j,b}$ be the shuffle variables in $\mathbf U_\mu^-$ and $\mathbf U_\mu^+$, respectively.  
In the shuffle realization, the generators $x_{j,r}^-$ and $x_{j,r}^+$ correspond respectively to $z_{j,1}^r$ and $w_{j,1}^r$ while multiplication in each half is realized by the shuffle product. 
Choose multidegrees $\mathbf n^\pm=(n_j^\pm)_{j\in \indx}$. 
On the triangular decomposition summand
\[
(\mathbf U_\mu^-)_{\mathbf n^-}\otimes \mathbf U_\mu^0
\otimes(\mathbf U_\mu^+)_{\mathbf n^+},
\]
realized using $n_j^-$-many negative and $n_j^+$-many positive
shuffle variables of each color $j$, the operator
$\operatorname{ad}D_{i,1}$ acts as multiplication by
\begin{equation}\label{eq:grinj-shuffle}
 \sum_{j\in \indx}[a_{ij}]\left(
 \sum_{b=1}^{n_j^+}(Cw_{j,b}-w_{j,b}^{-1})
 -\sum_{a=1}^{n_j^-}(Cz_{j,a}-z_{j,a}^{-1})\right).
\end{equation} 
In particular, it follows that all $\operatorname{ad}D_{i,1}$ simultaneously act as multiplication by zero precisely on $\mathbf{U}_\mu^0$, yielding the first statement.

For the second statement, take $P \in \mathbf{U}_\mu^0$ and write it as 
$P=\sum_{\mathbf a,\mathbf b}(\prod_iK_{i,+}^{a_i}K_{i,-}^{b_i})
p_{\mathbf a,\mathbf b}(h)$, for $\mathbf{a} = (a_i)_{i \in \indx} \in \Z^N$, where each $p_{\mathbf a,\mathbf b}(h)$ is a polynomial in finitely many $h_{i,m}$. Suppose that $P \in Z_{\mathbf U_\mu}(\mathbf U_\mu^-)$. 
The triangular decomposition, together with 
\[
h_{i,m}x^+_j(z)
=x^+_j(z)\left(h_{i,m} - \textstyle\frac{[ma_{ij}]}m z^m\right), 
\]
implies that 
\eq \label{eq: centralizer proof1}
q^{\sum_i a_{ij}(b_i-a_i)}
p_{\mathbf a,\mathbf b}\bigl(h-v_j(z)\bigr)
=p_{\mathbf a,\mathbf b}(h),
\qquad
(v_j(z))_{i,m} = \textstyle\frac{[ma_{ij}]}m z^m. 
\eneq 
Comparing the top degree (with respect to the number of the variables $h_{i,m}$ occurring) monomials on both sides above, we see that $\sum_i a_{ij}(b_i-a_i)$ for each $i \in \indx$. Since the Cartan matrix is invertible, $\mathbf{a} = \mathbf{b}$. In particular $\prod_iK_{i,+}^{a_i}K_{i,-}^{b_i} = \prod_i \kappa_i^{-a_i} \in \mathcal{Z}$. 

It remains to be shown that each $p_{\mathbf a,\mathbf a}(h)$ is constant. Equality \eqref{eq: centralizer proof1} has now become $p(h-v_j(z))=p(h)$. Iterating this identity yields $p(h-nv_j(z))=p(h)$ for every $n >0$. 
Since this is a polynomial identity in \(n\) over a field of characteristic zero, it holds with an indeterminate \(t\) in place of \(n\). 
Differentiating at \(t=0\) yields
\[
\sum_{m\ne0}\sum_{i \in \indx}
\frac{[ma_{ij}]}m z^m
\frac{\partial p}{\partial h_{i,m}}=0.
\]
Comparing powers of $z$, we obtain, for each $m \geq 1$ and $j \in \indx$: 
\[
\sum_{i\in \indx}[ma_{ij}]
\frac{\partial p}{\partial h_{i,m}}=0. 
\]
The matrix \(\bigl([ma_{ij}]\bigr)_{i,j}\) is invertible over \(\mathbb C(q)\) -- at \(q=1\) it specializes to \(m(a_{ij})\), whose determinant is nonzero. Hence every partial derivative of \(p\) vanishes, so \(p\) is constant. This proves $P\in \mathcal{Z}$. The reverse inclusion $\mathcal{Z} \subset \mathbf U_\mu^0 \cap Z_{\mathbf U_\mu}(\mathbf U_\mu^-)$ is obvious. 
\enproof

\subsubsection{Leading terms of Drinfeld generators} 

Endow the algebra $\mathbf U_\mu$ with the $\mathbb Z$-grading $\mathbf U_\mu = \bigoplus_d \mathbf{U}_\mu[d]$ via 
\[
 \deg x_{i,r}^-=1,\qquad \deg x_{i,r}^+=-1,\qquad
 \deg h_{i,\pm m}=\deg K_{i,\pm}^{\pm1}=0,
\]
and let $F_d\mathbf U_\mu=\bigoplus_{e\le d}\mathbf U_\mu[e]$ be the associated filtration.

\Lem
The following congruences hold, for some $z_{i,m}\in \mathcal{Z}$:
\begin{align}
 \iota_\mu(A_{i,r})&\equiv x_{i,\mu_i-r}^- 
       &&\pmod{F_0\mathbf U_\mu},\label{eq:grinj-real}\\
 \iota_\mu(H_{i,m})&\equiv D_{i,m}+z_{i,m} 
       &&\pmod{F_{-1}\mathbf U_\mu}\label{eq:grinj-imag},
\end{align}
for all $r \in \Z$ and $m \geq 1$. 
\enlem

\Proof
A summand of the correction term $Q_i$ labelled by the interval $[a,b]$ contains $2(b-a)+1$ generators corresponding to positive simple roots. 
Since $a<b$, it has degree at most $-3$. Thus  $Q_i\in F_{-3}\mathbf U_\mu$, and the explicit description of $\iota_\mu$ from Theorem \ref{thm: homo rank 2} yields \eqref{eq:grinj-real} on the Levendorskii generators.
Moreover, refining \eqref{eq: H1 emb formula1} shows that 
\[
\iota_\mu(H_{i,1}) = 
 D_{i,1}+(1-q^{-4})\gamma_iC^{\mu_i}x_{i,0}^+x_{i,-1}^+\kappa_i 
 - 2^{-\delta_{\mu_i,0}}\gamma_i\mi [ \iota_\mu(A_{i,0}), Q_i]_{q^2}. 
\]
The final $q^2$-commutator has degree at most $-2$, so 
\begin{equation}\label{eq:grinj-H1}
 \iota_\mu(H_{i,1})\equiv D_{i,1} \pmod{F_{-1}\mathbf U_\mu}.
\end{equation} 
The inductive definition \eqref{eq: Air defi g} and \eqref{eq:grinj-D} allow us to deduce \eqref{eq:grinj-real} for all $r \in \Z$. 

Now let $m \geq 1$. The element $\iota_\mu(H_{i,m})$ commutes with every $\iota_\mu(H_{j,1})$. If its highest non-zero homogeneous component had positive degree, \eqref{eq:grinj-H1} would force that component to commute with every $D_{j,1}$, contradicting Lemma~\ref{lem:grinj-centralizers}. Hence $\iota_\mu(H_{i,m})\in F_0\mathbf U_\mu$. Its degree-zero component $P_{i,m}$ likewise commutes with all $D_{j,1}$, and therefore belongs to $\mathbf U_\mu^0$. 
Taking homogeneous components of degree one in \eqref{eq: grel2} gives
\[
 [P_{i,m},x_{j,s}^-] = \frac{[ma_{ij}]}m(x_{j,s-m}^--C^mx_{j,s+m}^-).
\]
Therefore, by \eqref{eq:grinj-D}, $P_{i,m}-D_{i,m}$ commutes with every
$x_{j,s}^-$. The sought congruence \eqref{eq:grinj-imag} now follows from the second assertion of Lemma~\ref{lem:grinj-centralizers}. 
\enproof

\subsubsection{Proof}
\label{sss: injectivity proof final} 

We will now complete the proof of Theorem \ref{thm: injectivity}. 
Define a filtration $\mathscr{F}$ on $\mathbf U^\imath_{\mu,\gamma}$ by setting $\deg A_{i,r}=1$ and $\deg H_{i,m}=0$. 
Consider the semi-direct product $\mathbf{U}_\mu^- \rtimes \C(q)[t_{i,m}\mid i\in \indx,\ m\ge1]$, where $t_{i,m}$ acts on $\mathbf{U}_\mu^-$ via \eqref{eq:grinj-D}. 
We claim that there is a surjective algebra homomorphism 
\eq \label{eq:grinj-theta}
 \pi \colon \mathbf U_\mu^- \rtimes \C(q)[t_{i,m}\mid i\in \indx,\ m\ge1] \twoheadrightarrow \operatorname{gr}_{\mathscr{F}} \mathbf U^\imath_{\mu,\boldsymbol\gamma}, \qquad  x_{i,s}^- \mapsto \bar{A}_{i,\mu_i-s}, \qquad t_{i,m} \mapsto \bar{H}_{i,m},
\eneq
where the bars denote leading terms. 
Indeed, passing to leading terms, relation \eqref{eq: drel5} becomes the usual quadratic relation, and the $\imath$Drinfeld--Serre relation \eqref{eq: drel6} becomes the usual Drinfeld--Serre relation. 
The imaginary generators commute, and the real-imaginary relations are preserved, matching \eqref{eq:grinj-D}. 

By \eqref{eq:grinj-real}--\eqref{eq:grinj-imag}, $\iota_\mu$ is
filtered. The composition $(\operatorname{gr}\iota_\mu)\circ\pi$ sends 
\[
 u\otimes p(t_{i,m}) \longmapsto
 u\,p\bigl(C^mh_{i,m}-h_{i,-m}+z_{i,m}\bigr).
\]
The Cartan elements in the expression on the RHS are algebraically independent:
each contains its own variable $h_{i,-m}$ with coefficient $-1$, and the $z_{i,m}$ are central. 
The triangular decomposition of $\mathbf U_\mu$ (see \cite[Proposition 5.1]{FT19}) consequently makes this composition injective. Since $\pi$ is surjective, $\operatorname{gr}\iota_\mu$ must be injective. As the source filtration $\mathscr{F}$ is exhaustive and nonnegative, it follows that $\iota_\mu$ is injective, completing the proof of Theorem~\ref{thm: injectivity}. 

As a corollary, we also obtain the following PBW theorem for $\mathbf U^\imath_{\mu,\boldsymbol\gamma}$. 

\Cor \label{cor: PBW} 
There is an algebra isomorphism 
\eq 
 \pi \colon \mathbf U_\mu^- \rtimes \C(q)[t_{i,m}\mid i\in \indx,\ m\ge1] \twoheadrightarrow \operatorname{gr}_{\mathscr{F}} \mathbf U^\imath_{\mu,\boldsymbol\gamma}, \qquad  x_{i,s}^- \mapsto \bar{A}_{i,\mu_i-s}, \qquad t_{i,m} \mapsto \bar{H}_{i,m},
\eneq
\encor 

\Proof
Indeed, the proof above established that  $(\operatorname{gr}\iota_\mu)\circ\pi$ is injective, which forces the injectivity of $\pi$. 
\enproof

\subsection{Graded coideal property} 

By Theorems \ref{thm: homo rank 2} and \ref{thm: injectivity}, we may view $\mathbf{U}^\imath_{\mu,\boldsymbol\gamma}$ as a subalgebra of $\mathbf{U}_\mu$ via $\iota_{\mu}$. We will write $\bar{\mathbf{U}}^\imath_{\mu,\boldsymbol\gamma}$ for its image in~$\bar{\mathbf{U}}_\mu$.

\Thm \label{thm: coideal property 2}
For any $\mu_1, \mu_2 \in \Lambda_-$, we have 
\[
\Delta_{\mu_1, \mu_2} ( \bar{\mathbf{U}}^\imath_{\mu,\boldsymbol\gamma} ) \subseteq \bar{\Uq}_{\mu_1} \otimes \bar{\mathbf{U}}^\imath_{\mu_2,\boldsymbol\gamma'},  
\]
with  $\boldsymbol\gamma' = (C^{\mu_{1,1}}\gamma_1, \cdots, C^{\mu_{1,N}}\gamma_N, C)$. 
\enthm

The proof is an adaptation of the proof of Theorem \ref{thm: coideal property 1}. We will sketch it below in \S \ref{sss: est FT}--\S \ref{sss: coideal proof gen rank}. 

\subsubsection{Estimates of Finkelberg--Tsymbaliuk coproduct} 
\label{sss: est FT} 

Abbreviate $\alpha_i = \mu_{1,i}$ and $\beta_i = \mu_{2,i}$. 

\Lem \label{lem:recurrence}
For every $k\geq0$,
\begin{equation} \label{eq:recurrence}
\Delta_{\mu_1,\mu_2} (x^-_{i,\alpha_i-k}) = x^-_{i,\alpha_i-k}\otimes1+K_i^{-1}\otimes x^-_{i,-k}+T_{i,k},
\end{equation}
where $T_{i,0}=0$, and, for $k\geq1$,
\begin{equation} \label{eq:remainder-algebra}
T_{i,k}\in\Uq_{\mu_1}\otimes \left\langle x^-_{j,0}\ (j\neq i),\ x^-_{i,s}\ (-k+1\leq s\leq0) \right\rangle.
\end{equation} 
\enlem 

\Proof
We prove this by induction on $k$. 
The $k=0$ case follows directly from the definition of $\Delta_{\mu_1,\mu_2}$. 
Define the auxiliary Cartan element 
\begin{equation}\label{eq:dual-cartan}
h^\vee_{i,-1} = -\sum_{j=1}^N \frac{[\min(i,j)][N+1-\max(i,j)]}{[N+1]}h_{j,-1}, 
\end{equation}
so that 
\begin{equation} \label{eq:dual-action}
[h^\vee_{i,-1},x^-_{j,r}]=\delta_{ij}x^-_{i,r-1}, \qquad [h^\vee_{i,-1},K_j]=0.
\end{equation}
The coproduct formula for $h_{i,-1}$ in \cite[Theorem~10.22]{FT19} shows that
\begin{equation}\label{eq:cartan-coproduct}
\Delta(h^\vee_{i,-1}) = h^\vee_{i,-1}\otimes1+1\otimes h^\vee_{i,-1}+Z_i, \qquad
 Z_i\in\Uq_{\mu_1}\otimes\langle x^-_{j,0}\mid j \in \indx \rangle.
\end{equation} 

Consider the commutator of \eqref{eq:cartan-coproduct} and \eqref{eq:recurrence}. The commutator of  $h^\vee_{i,-1}\otimes1+1\otimes h^\vee_{i,-1}$ and  $x^-_{i,\alpha_i-k}\otimes1+K_i^{-1}\otimes x^-_{i,-k}$   produces  $x^-_{i,\alpha_i-k-1}\otimes1+K_i^{-1}\otimes x^-_{i,-k-1}$. Moreover, commuting $h^\vee_{i,-1}\otimes1+1\otimes h^\vee_{i,-1}$ with $T_{i,k}$ only lowers the second index in $x^-_{i,s}$, ensuring that the resulting expressions is in \eqref{eq:remainder-algebra}. 
Commuting $Z_i$ with $T_{i,k}$ or $x^-_{i,\alpha_i-k}\otimes1+K_i^{-1}\otimes x^-_{i,-k}$ introduces only zero modes and, possibly, $x^-_{i,-k}$.
Thus, the new overall remainder $T_{i,k+1}$ satisfies \eqref{eq:remainder-algebra} for $k+1$. 
\enproof

\subsubsection{Lower boundary generator} 

Given $i \in \indx$, recursively define: 
\begin{align}
\mathcal A^{(i)}_{i,i} &=A_{i,\beta_i-1},\\
\mathcal A^{(i)}_{a,i} &=[\mathcal A^{(i)}_{a+1,i},A_{a,\beta_a}]_q\quad(a<i),\label{eq:nested}\\
\mathcal A^{(i)}_{a,b} &=[\mathcal A^{(i)}_{a,b-1},A_{b,\beta_b}]_q\quad(b>i), \\ 
s_{ab}&=(-1)^{\ell_{ab}+\epsilon_{ab}(i)}q^{-\ell_{ab}}d^{\epsilon_{ab}(i)}.
\end{align} 

\Lem \label{lem:bottom}
For every $i \in \indx$,
\begin{equation} \label{eq:bottom-coproduct}
 \Delta_{\mu_1,\mu_2} \bigl(\iota_{\mu}(A_{i,\mu_i-1})\bigr)
 =x^-_{i,1}\otimes1+
 \sum_{a\leq i\leq b}s_{ab}L_{a,i}R_{i,b}K_i
       \otimes\mathcal A^{(i)}_{a,b}.
\end{equation}
In particular, its image lies in $\bar{\Uq}_{\mu_1}\otimes \bar{\mathbf{U}}_{\mu_2,\boldsymbol\gamma'}$.
\enlem 

\Proof
This follows by a straightforward but somewhat laborious calculation from \cite[Theorem 10.22]{FT19}. The details can be found in \cite{LevQSP}. 
\enproof

\subsubsection{Proof} 
\label{sss: coideal proof gen rank}

We can now complete the proof of Theorem \ref{thm: coideal property 2}. Thanks to Lemma \ref{lem:bottom}, we only  need to check that, for each $i \in \indx$, the elements $\iota_\mu(A_{i,0}), \iota_\mu(A_{i,-1})$ and $\iota_\mu(A_{i,\mu_i})$ lie in $\bar{\Uq}_{\mu_1}\otimes \bar{\mathbf{U}}_{\mu_2,\boldsymbol\gamma'}$. Indeed, 
\begin{equation}\label{eq:zero-mode}
\Delta_{\mu_1, \mu_2}(A_{i,\mu_i}) = \Delta_{\mu_1, \mu_2}(x^-_{i,0}) =
\begin{cases}
x^-_{i,0}\otimes1, &\alpha_i<0,\\
x^-_{i,0}\otimes1+K_i^{-1}\otimes A_{i,\beta_i}, &\alpha_i=0.
\end{cases}
\end{equation}
Moreover, it follows from Lemma~\ref{lem:recurrence}, with $k=-\beta_i-1$, that 
\begin{equation}\label{eq:minus-one}
\Delta_{\mu_1, \mu_2}(A_{i,-1}) = \Delta_{\mu_1, \mu_2}(x^-_{i,\mu_i+1}) = 
x^-_{i,\mu_i+1}\otimes1 + K_i^{-1}\otimes A_{i,-1} + T_{i,-\beta_i-1}, 
\end{equation}
where $T_{i,-\beta_i-1} \in \bar{\Uq}_{\mu_1}\otimes  \bar{\mathbf{U}}_{\mu_2,\boldsymbol\gamma'}$. 
For the final generator $\iota_\mu(A_{i,0})$, we again use Lemma~\ref{lem:recurrence}, with $k=-\beta_i$, and 
\[
\Delta_{\mu_1, \mu_2}(x^+_{i,0}K_i^{-1}) = x^+_{i,0}K_i^{-1}\otimes1 + K_i^{-1}\otimes x^+_{i,0}K_i^{-1}.
\]
Since $C^{\mu_i}\gamma_i=C^{\beta_i}\gamma'_i$, this gives
\begin{align}
\Delta_{\mu_1, \mu_2}(A_{i,0}) = \bigl(x^-_{i,\mu_i} - q^{-2}C^{\mu_i}\gamma_i x^+_{i,0}K_i^{-1}\bigr)\otimes1 
+ K_i^{-1}\otimes A_{i,0} + T_{i,-\beta_i}, 
\end{align}
which lies in $\bar{\Uq}_{\mu_1}\otimes  \bar{\mathbf{U}}_{\mu_2,\boldsymbol\gamma'}$. 
This completes the proof of Theorem \ref{thm: coideal property 2}.

\subsubsection{Graded coideal reformulation} 

Now choose $\boldsymbol\gamma \in (\C^\times)^{N+1}$ with $C=1$. Set 
\[
\bar{\Uq}_{\bullet, \boldsymbol\gamma}^\imath = \bigoplus_{\mu \in \Lambda_-} \bar{\Uq}_{\mu, \boldsymbol\gamma}^\imath \mathbf{e}(\mu). 
\]
Theorem \ref{thm: coideal property 2} can be reformulated in the following way. 

\Cor
The subalgebra $\bar{\Uq}_{\bullet, \boldsymbol\gamma}^\imath$ is a graded coideal of $\bar{\mathbf{U}}_\bullet$. 
\encor

\newcommand{\etalchar}[1]{$^{#1}$}
\providecommand{\bysame}{\leavevmode\hbox to3em{\hrulefill}\thinspace}
\providecommand{\MR}{\relax\ifhmode\unskip\space\fi MR }
\providecommand{\MRhref}[2]{%
  \href{http://www.ams.org/mathscinet-getitem?mr=#1}{#2}
}
\providecommand{\href}[2]{#2}

\end{document}